\documentclass[11pt, reqno]{amsart}
\usepackage[utf8]{inputenc}
\usepackage[margin=1in]{geometry}
\usepackage{amsmath, amssymb, amsfonts, amsthm}
\usepackage{comment}
\usepackage[bookmarks=true]{hyperref}
\usepackage{enumitem}
\usepackage{tikz}
\usepackage{caption}
\usetikzlibrary{decorations.markings}
\usepackage{microtype}
\usepackage{subcaption}
\usepackage{comment}

\theoremstyle{definition}
\newtheorem{thm}{Theorem}[section]
\newtheorem{lem}[thm]{Lemma}
\newtheorem{df}[thm]{Definition}
\newtheorem{prop}[thm]{Proposition}

\newtheorem{rem}[thm]{Remark}
\newtheorem{nota}[thm]{Notation}

\newcommand{\C}{\mathbb{C}}

\newcommand{\R}{\mathbb{R}}
\newcommand{\Z}{\mathbb{Z}}
\newcommand{\N}{\mathbb{N}}
\newcommand{\M}{\mathcal{M}}

\newcommand{\ii}{\mathbf{i}}
\newcommand{\bfb}{\mathbf{b}}
\newcommand{\bft}{\mathbf{t}}
\newcommand{\bfy}{\mathbf{y}}
\newcommand{\X}{\mathfrak{X}}
\newcommand{\Y}{\mathfrak{Y}}
\newcommand{\LL}{\mathcal{L}}
\newcommand{\GT}{\mathbb{GT}}
\newcommand{\GTp}{\mathbb{GT}^+}
\newcommand{\half}{\frac{1}{2}}
\newcommand{\Res}{\textrm{Res}}
\newcommand{\Prob}{\textrm{Prob}}
\newcommand{\Martin}{\textrm{Martin}}

\numberwithin{equation}{section}

\begin{document}

\title[$q$-deformed Character Theory $\cdots$]{$q$-deformed Character Theory for Infinite-Dimensional Symplectic and Orthogonal Groups}

\author{
  Cesar Cuenca
  and
  Vadim Gorin
}

\date{}

\maketitle

\begin{abstract}

The classification of irreducible, spherical characters of the infinite-dimensional 
unitary/orthogonal/symplectic groups can be obtained by finding all 
possible limits of normalized, irreducible characters of 
the corresponding finite-dimensional groups, as the rank tends to infinity.
We solve a $q$-deformed version of the latter problem for orthogonal and 
symplectic groups, extending previously known results for the unitary group.
The proof is based on novel determinantal and double-contour integral formulas for the $q$-specialized characters.

\end{abstract}

\tableofcontents

\section{Introduction}

\subsection{Preface}

For each of the three series of classical compact Lie groups: unitary $U(N)$, orthogonal $SO(N)$,
and symplectic $Sp(N)$, one can naturally embed groups of the smaller rank into the larger ones
and form inductive limits $U(\infty)=\bigcup_{N=1}^{\infty} U(N)$, $SO(\infty)=\bigcup_{N=1}^{\infty} SO(N)$, $Sp(\infty)=\bigcup_{N=1}^{\infty} Sp(N)$. The study of
such infinite--dimensional or ``big'' groups has been a  central topic of the asymptotic
representation theory during the last 40 years. These groups are wild, which means that one needs
to restrict the class of representations, in order to get a meaningful theory. One point of view is
to deal with characters, i.e.\ central, positive--definite, continuous functions on the group. The
\emph{extreme} characters then correspond to finite factor representations of the group; for
$U(\infty)$ all of them were classified by Voiculescu \cite{Voiculescu}, while for $SO(\infty)$ and
$Sp(\infty)$ similar results were obtained by Boyer \cite{Boyer_OSp}. More general characters can be
identified with spherical representations of the Gelfand pairs $(U(\infty)\times U(\infty),
U(\infty))$, $(SO(\infty)\times SO(\infty), SO(\infty))$, $(Sp(\infty)\times Sp(\infty),
Sp(\infty))$, and such a theory was introduced and thoroughly studied by Olshanski \cite{Olsh_repesentations_U_short},
\cite{Olsh_repesentations_U_long}. From another direction, as first noticed in \cite{VK_U},\cite{Boyer_U}, one can identify extreme characters with totally--positive
Toeplitz matrices, and then their classification theorem becomes equivalent to the earlier Edrei's
theorem \cite{Edrei} from classical analysis.

Yet another \emph{approximative} approach was suggested by Vershik and Kerov \cite{VK_U}; in this
approach, the parameters of the characters of $U(\infty)$, $SO(\infty)$, $Sp(\infty)$ become
limits of normalized lengths as $N\to\infty$ of rows and columns in the Young diagrams
parameterizing the irreducible representations of finite--dimensional
groups $U(N)$, $SO(N)$, $Sp(N)$. This idea can be used to produce several distinct proofs of the character
classification theorems of Voiculescu and Boyer, see \cite{OO2}, \cite{OkOlsh_BC}, \cite{BO_new}, \cite{Petrov},
\cite{GP}, \cite{Olsh_RelDim}.

\medskip

Classical Lie groups admit a $q$--deformation to \emph{quantum groups}, cf.\ \cite{BK}, \cite{CP}, which
leads to a natural question of whether a similar deformation is possible for the character theory
of their infinite--dimensional versions. This question for the unitary groups $U(N)$ was first
addressed by the second author \cite{G} and based on the following observation.
In the Vershik--Kerov approach, the characters of
$U(\infty)$ can be treated as the limits of normalized \emph{Schur polynomials} (characters
of irreducible representation of $U(N)$) as the number of variables goes to infinity
\begin{equation}
\label{eq_intro_norm_Schur}
 \lim_{N\to\infty}\frac{s_\lambda(x_1,x_2,\dots,x_k, 1^{N-k})}
 {s_\lambda(1^N)}, \quad k=1,2,\dots,\quad |x_1|=|x_2|=\dots=|x_k|=1,
\end{equation}
where $\lambda=\lambda(N)$ changes in an appropriate way (as $N\to\infty$) to guarantee the existence
of the limit, and $1^N$ means $N$ variables all equal to $1$. Then the $q$--deformation of \cite{G}
is based on the replacement of $1^{N-k}$ and $1^{N}$ in \eqref{eq_intro_norm_Schur} by the
geometric series with ratio $q$. Such a point of view turned out to be fruitful: in \cite{G} a
classification theorem for the new $q$--characters was obtained (they are parameterized by nondecreasing
sequences of integers) and in \cite{BG}, \cite{C}, \cite{Petrov}, \cite{GP}, \cite{GO},
\cite{Ol2} the topic was further developed. It was noticed in \cite{G} that the
$q$--characters have a link to the \emph{quantum traces} for the representations of the quantized
universal enveloping algebra $U_{q}(\mathfrak{gl}_N)$; however, no infinite--dimensional object was constructed. A more elaborate representation--theoretic
interpretation for the $q$--characters of $U(\infty)$ was presented recently by Sato
\cite{Sato} in the language of compact quantum groups.

After \cite{G} appeared, an immediate question arose: can the results be extended to other root
systems, i.e.\ to orthogonal and symplectic groups? Despite several other approaches to
$q$--characters of unitary groups appearing in the subsequent years, it remained unclear whether the existence of a ``good''
$q$--deformation of the character theory for the infinite--dimensional group is an artifact of the
root system of type $A$, or if it exists for all classical series of Lie groups? In the present article we
resolve this question by constructing a rich $q$--deformed character theory for $SO(\infty)$ and
$Sp(\infty)$.

\subsection{$q$--deformed characters: main results} The Vershik--Kerov approach \cite{VK_U}, \cite{OO} to the asymptotic
representation theory makes the classification of all extreme characters of $U(\infty)$ equivalent to the following problem. Find all sequences of \emph{signatures} (i.e.\ highest weights of irreducible representations) $\lambda(N)=(\lambda(N)_1 \ge \lambda(N)_2 \ge \dots \ge \lambda(N)_N)\in \mathbb Z^{N}$, $N=1,2,\dots$, such that for each $k=1,2,\dots$, the limit \eqref{eq_intro_norm_Schur} exists uniformly on the torus $\{ (x_1,\dots,x_k)\in \mathbb C^k\, :\,  |x_i|=1,\, 1\le i \le k\}$. The limit itself is then identified with the value of the character on a unitary matrix from the group $U(\infty)$ with non-trivial (i.e.\ different from $1$) eigenvalues $x_1,\dots,x_k$.

\medskip

The $q$--deformation of \cite{G} suggests to fix a parameter $0<q<1$ and find the sequences of signatures $\lambda(N)$, such that for each $k=1,2,\dots$ there exists a limit
\begin{equation}
\label{eq_intro_norm_Schur_q}
 \lim_{N\to\infty}\frac{s_{\lambda(N)}(x_1,x_2,\dots,x_k, q^{-k}, q^{-k-1},\dots, q^{1-N})}
 {s_\lambda(1, q^{-1}, q^{-2},\dots, q^{1-N})}.
\end{equation}
A priori in \eqref{eq_intro_norm_Schur_q} the convergence is assumed to be uniform over $|x_i|=q^{1-i}$, $1\le i \le k$, i.e.\ on the torus of $q$--growing radius. However, a posteriori, \eqref{eq_intro_norm_Schur_q} converges for all $x_i\ne 0$. The limit of \eqref{eq_intro_norm_Schur_q} is then the desired $q$--character. There are several reformulations of this asymptotic problem: one of them is the identification of the minimal boundary (=extreme Gibbs measure on the space of paths) of a certain branching graph, called the $q$--Gelfand--Tsetlin graph. The vertices of this graph are labels of irreducible representations of $U(N)$, $N=1,2,\dots$, i.e.\ signatures, and the edges encode the branching rules upon restrictions from $U(N)$ onto $U(N-1)$. Each edge comes with a $q$--dependent weight, and the whole combinatorics can be linked to the notion of the quantum dimension for the representations of the quantized unversal envelopping algebra $U_q(\mathfrak{gl}_N)$. We refer to \cite{G} and Section \ref{sec:qschurgraph} for the details. Another reformulation deals with $q$--Toeplitz matrices, see \cite[Section 1.5]{G}.

The limits in \eqref{eq_intro_norm_Schur_q} turn out to be parameterized by infinite sequences of integers $\nu_1\le \nu_2\le \nu_3\le\dots$, $\nu_i\in\Z$, which are identified with the \emph{last} rows of $\lambda(N)$:
\begin{equation}
\label{eq_intro_stab_q_Schur_1}
 \nu_i = \lim_{N\to\infty} \lambda(N)_{N+1-i},\quad i=1,2,\dots.
\end{equation}

One clear feature of \eqref{eq_intro_norm_Schur_q} is its asymmetry under the change $q\mapsto q^{-1}$. The properties of Schur polynomials imply that such change is equivalent to the transformation $\lambda(N)\mapsto\widetilde{\lambda}(N) := (-\lambda(N)_N \ge \dots \geq -\lambda(N)_1)$ and therefore, when $q > 1$, as $N\to\infty$ we would need to consider the first rows of $\lambda(N)$ instead of the last rows in \eqref{eq_intro_stab_q_Schur_1}.

\medskip

If we now switch to orthogonal and symplectic groups, then the symmetry becomes important. Indeed, the eigenvalues for the matrices of these groups come in pairs $\{z_i,z_i^{-1}\}$ and the characters are invariant under inversion of the variables. Thus, the first step towards the $q$--deformation for $SO(\infty)$, $Sp(\infty)$ is to make the setting of \eqref{eq_intro_norm_Schur_q} symmetric by changing it to
\begin{equation}
\label{eq_intro_norm_Schur_q_symmetric}
 \lim_{N\to\infty}\frac{s_{\lambda(N)}(q^{-N},q^{1-N},\dots,q^{-k-1},x_{-k},\dots,x_k, q^{k+1}, q^{k+1},\dots, q^{N})}
 {s_{\lambda(N)}(q^{-N}, q^{1-N},\dots,q^{N-1},q^N)},\quad x_1,\dots,x_k\in \mathbb C^*,
\end{equation}
where $\lambda(N)$ now has $2N+1$ rows: $\lambda(N)_{-N} \ge \lambda(N)_{1-N} \ge\dots\ge \lambda(N)_N$. Note that, since the Schur polynomials are homogeneous, there is some further freedom, as we can multiply all its variables by arbitrary $q^M$; the most general setup is described in Section \ref{Section_Type_A_results}.

Our first main result (Theorem \ref{thm:asymptoticsmultivariate}) proves that the limits of \eqref{eq_intro_norm_Schur_q_symmetric} are parameterized by \emph{two--sided} sequences of integers $\dots\ge \nu_{-1}\ge \nu_0 \ge \nu_1 \ge\dots$ and the limit exists if and only if
\begin{equation}
\label{eq_intro_stabilization_A}
 \lim_{N\to\infty} \lambda(N)_i=\nu_{i}\quad \text{ for all } i\in\mathbb Z.
\end{equation}
The limiting functions (``symmetric $q$--characters'') are given by explicit contour integral formulas.

\medskip

Proceeding to the orthogonal and symplectic cases, let $G(N)$ denote either $SO(2N+1)$, or $Sp(N)$, or $SO(2N)$, corresponding to the root systems $B_N$, $C_N$, $D_N$, respectively. The irreducible representations are still parameterized by signatures $\lambda_1\ge\lambda_2\ge\dots \ge\lambda_N$, but this time all the coordinates are required to be \emph{positive}.\footnote{For $SO(2N)$, $\lambda_N$ is allowed to be negative, but one should have $\lambda_1\ge\dots\ge\lambda_{N-1}\ge|\lambda_N|$. In this case we deal instead with the direct sum of two twin representations that differ by a flip of the sign of $\lambda_N$.} The characters $\chi^{G}_{\lambda}$ of the representations can be then identified with symmetric Laurent polynomials in $N$ variables $z_1,\dots,z_N$, invariant under the inversions $z_i \mapsto z_i^{-1}$, see Section \ref{sec:BCqchars} for more details. Set $\epsilon=1/2,1,0$ for types $B,C,D$, respectively, and consider the limits
\begin{equation}
\label{eq_intro_norm_BC_q}
 \lim_{N\to\infty} \frac{\chi_{\lambda(N)}^G(x_1,\dots,x_k, q^{k+\epsilon},\dots,q^{N-1+\epsilon})}
 {\chi_{\lambda(N)}^G(q^{\epsilon},q^{1+\epsilon},\dots,q^{N-1+\epsilon})}, \qquad x_1,\dots,x_k\in \mathbb C^*,
\end{equation}
where $\lambda(N) = (\lambda(N)_1\ge\lambda(N)_2\ge\dots\ge\lambda(N)_N\ge 0)$.
Our second main result (Theorem \ref{thm:multivariablesymplectic}) proves that the limits in \eqref{eq_intro_norm_BC_q} are parameterized by growing sequences of nonnegative integers $0\le \nu_1\le \nu_2\le \dots$ and the limit exists if and only if
\begin{equation}
\label{eq_intro_stabilization_BC}
 \lim_{N\to\infty} \lambda(N)_{N+1-i} = \nu_i, \quad i=1,2,\dots.
\end{equation}
The limit functions are given by explicit contour integral formulas.

\smallskip

One might be wondering about our choice of normalization in
\eqref{eq_intro_norm_BC_q}. It has several explanations. First, the numbers $\epsilon,\epsilon+1,\dots,N-1+\epsilon$ are precisely the coordinates of $\rho$ --- the half--sum of the positive roots in the corresponding root system. The value of ${\chi_{\lambda(N)}^G(q^{\epsilon},q^{1+\epsilon},\dots,q^{N-1+\epsilon})}$ is typically called the \emph{quantum dimension} in the quantum groups literature --- this matches the normalization in \eqref{eq_intro_norm_Schur} by the ordinary dimension $\mathrm{Dim}(\lambda)=s_\lambda(1^N)$ of the representation.  From a more technical point of view,
under the normalization of \eqref{eq_intro_norm_BC_q}, the characters possess the label--variable duality, see Lemma \ref{lem:dualitysymplectic}; this duality is one of the key ingredients of our proofs.

Another point worth emphasizing is that we get two different $q$--deformed character theories for $SO(\infty)$: one from the approximation by odd-dimensional groups and another one from even--dimensional groups. In the $q=1$ case there was no difference, however the answers (e.g.\ the functional form of the limit in \eqref{eq_intro_norm_BC_q})  look different in our $q$--deformation. From the representation--theoretic point of view this can be linked to the fact that quantum groups related to $SO(2N)$ and $SO(2N+1)$ are also very different, and no embeddings of one into another are known. In the infinite--dimensional setting related to quantum groups, the distinction between $B$ and $D$ series also appeared before, e.g.\ in \cite{Olsh_Yang}.

\subsection{Further results and outlook}

There is another reformulation of the classification theorem for characters of $U(\infty)$, $SO(\infty)$, $Sp(\infty)$ in the language of combinatorial probability. Then one needs to describe all sequences of measures $P_N$, $N=1,2,\dots$ on the labels of irreducible representations of groups of rank $N$, which would agree through certain \emph{coherency relations}, obtained from the branching rules for the representations. Such a reformulation is also possible in our $q$--deformed setting and we describe it in Section \ref{sec:qschurgraph}. In particular, the stabilization property of \eqref{eq_intro_stabilization_A}, \eqref{eq_intro_stabilization_BC} then turns into the \emph{Law of Large Numbers} as explained in Sections \ref{sec:martin} and \ref{sec:martin_BC}.

We believe that there should be a way to transform our constructions from the language of special functions theory (as study of limits in \eqref{eq_intro_norm_Schur_q_symmetric}, \eqref{eq_intro_norm_BC_q}) and from the language of combinatorial probability (as classifying the coherence family of probability measures in Section \ref{sec:qschurgraph}) into a truly representation--theoretic framework dealing with certain infinite--dimensional analogues of the quantum groups. Development of such a framework is an important open problem.

In \cite{G} it was found that by using certain inhomogeneous analogues of Schur functions, one can encode the limits to \eqref{eq_intro_norm_Schur_q} through simple multiplicative formulas (in particular, avoiding any contour integrals). It would be interesting to try to find similar formulas also for the symmetric setup \eqref{eq_intro_norm_Schur_q_symmetric} as well as for the $B,C,D$ series of \eqref{eq_intro_norm_BC_q}, since the conceptual understanding for the limiting functions in \eqref{eq_intro_norm_Schur_q_symmetric} and \eqref{eq_intro_norm_BC_q} (given by the contour integral formulas in Theorems \ref{thm:asymptoticsmultivariate} and \ref{thm:multivariablesymplectic}) is currently missing.

\subsection{Methodology} Let us indicate one important idea that made the developments of this article possible. In \cite{GP} an approach to the study of asymptotics of \eqref{eq_intro_norm_Schur_q} was suggested based on contour integral formulas for $k=1$ and $k\times k$ determinantal formulas reducing the general $k$ case to the base case $k=1$. Through label--variable symmetry duality for the normalized Schur functions and analytic continuation, these formulas are dual to the closed generating function expression for complete homogeneous symmetric functions $h_k$ and Jacobi--Trudi formulas expressing Schur functions through $h_k$.

Although certain formulas for symplectic and orthogonal characters were also presented in \cite{GP}, they were not suitable for the purpose of the constructions of asymptotic representation theory (the normalization in an analogue of \eqref{eq_intro_norm_BC_q} was different and, more importantly, the particular choice of the geometric progression was $k$--dependent). The approach of \cite{GP} also did not work for the symmetric case of \eqref{eq_intro_norm_Schur_q_symmetric}.

In the present paper (as well as in the companion article \cite{GS} where similar ideas are used for the study of products of random matrices) we make the following observation. Through the label--variable duality and analytic continuation, the $k=1$ versions of the expressions \eqref{eq_intro_norm_Schur_q_symmetric}, \eqref{eq_intro_norm_BC_q} are linked to the characters corresponding to \emph{hook signatures}. The latter also have explicit generating functions, though more complicated than the ones for $h_k$, and therefore, we need to use \emph{double} contour integrals. The next step is to reduce general $k$ case to $k=1$. Instead of Jacobi--Trudi, our formulas are now related to Giambelli and Frobenius formulas, expressing characters with arbitrary signatures as determinants of hooks. We refer to Sections \ref{Section_A_char_proofs}, \ref{Section_BC_char_integrals} for the details.

\subsection{Terminology and conventions}

For convenience of the reader, we collect here some terminology that is used throughout the paper.

\medskip

Unless otherwise stated, we assume $q\in(0, 1)$ is a fixed parameter.

We often use the $q$-Pochhammer symbols
\begin{equation*}
(a; q)_n := \prod_{i=1}^n (1 - aq^{i-1}),\ n \geq 0; \ \ \ (a; q)_{\infty} := \prod_{i = 1}^{\infty} (1 - aq^{i-1}).
\end{equation*}

Moreover,
\begin{equation*}
(a_1, \ldots, a_m; q)_n := (a_1; q)_n\cdots (a_m; q)_n; \ \ \ (a_1, \ldots, a_m; q)_{\infty} := (a_1; q)_{\infty}\cdots (a_m; q)_{\infty}.
\end{equation*}

\medskip

We also use the following terminology and conventions:

\begin{itemize}[noitemsep]
	\item We denote $\N := \{1, 2, 3, \ldots\}$ and $\N_0 := \N \cup \{0\} = \{0, 1, 2, \ldots\}$.
	\item For $N\in\N$, denote by $\GT_N$ the set of signatures of length $N$, i.e., the set of $N$-tuples $\lambda = (\lambda_1, \ldots, \lambda_N)\in\Z^N$ such that $\lambda_1 \geq \ldots \geq \lambda_N$.
	\item For $N\in\N$, denote by $\GTp_N$ the set of nonnegative signatures of length $N$, i.e., the set of $N$-tuples $\lambda = (\lambda_1, \ldots, \lambda_N)\in\GT_N$ such that $\lambda_N \geq 0$.
	\item For $\lambda\in\GT_N$, denote $|\lambda| := \sum_{i=1}^N{\lambda_i}$ and $n(\lambda) := \lambda_2 + 2\lambda_3 + \ldots + (N-1)\lambda_N$. For $\lambda\in\GT_N, \mu\in\GT_{N-1}$, write $\lambda \succ \mu$ (or $\mu \prec \lambda$) if $\lambda_1 \geq \mu_1 \geq \lambda_2 \geq \cdots \geq \mu_{N-1} \geq \lambda_N$.
Similarly, when $\lambda, \mu\in\GTp_{N}$, write $\lambda \succ \mu$ if $\lambda_1 \geq \mu_1 \geq \lambda_2 \geq \mu_2 \geq \dots \geq \lambda_N \geq \mu_N$.
	\item We use $i$ as an index very often, thus we use the bold letter $\ii := \sqrt{-1}$ for the imaginary unit.
	\item The notation $(a^n)$, $n\in\N_0$, indicates the string $(a, \ldots, a)$ ($n$ entries). This is an empty string if $n = 0$.
\end{itemize}

\subsection{Acknowledgements}

We would like to thank G.~Olshanski for encouraging us to study whether the extension of \cite{G}
to orthogonal and symplectic groups is possible and for a number of fruitful discussions. V.G.~was
partially supported by the NSF grant DMS-1664619, by the NEC Corporation Fund for Research in
Computers and Communications, and by the Sloan Research Fellowship.
The authors also thank the organizers of the Park City Mathematics Institute research program on Random Matrix Theory, where part of this work was carried out.

\section{Characters of $U(N)$, $SO(N)$ and $Sp(N)$}

\subsection{Type A characters}

For any $\lambda\in\GT_N$, the \textit{Schur polynomial} $s_{\lambda}(x_1, \ldots, x_N)$ is
\[
s_{\lambda}(x_1, \ldots, x_N) := \frac{\det_{1\leq i, j\leq N} \left[ x_i^{\lambda_j + N - j} \right]}{V(x_1, \ldots, x_N)},
\]
where
\[
V(x_1, \ldots, x_N) := \prod_{1\leq i<j\leq N}(x_i - x_j)
\]
is the Vandermonde determinant.

Recall that the unitary group is $U(N) := \{ A \in GL(N, \C) : AA^* = A^*A = I \}$.
By Weyl's formula, $s_{\lambda}(x_1, \ldots, x_N)$ is the value of the character of the irreducible representation of $U(N)$ with highest weight $\lambda$ on a matrix with eigenvalues $x_1, \ldots, x_N$, \cite[Ch. 24]{FH}.
In general, $s_{\lambda}(x_1, \ldots, x_N)$ is a symmetric homogeneous \textit{Laurent polynomial} of degree $|\lambda|$, and it is a polynomial when $\lambda\in\GTp_N$.

Two special Schur polynomials occur for the partitions $\lambda = (m, 0^{N-1}) = (m, 0, \ldots, 0)$, for some $m \geq 0$, and $\lambda = (1^m, 0^{N-m}) = (1, \ldots, 1, 0, \ldots, 0)$, for some $N \geq m \geq 0$.
In fact, $s_{(m, 0^{N-1})}(x_1, \ldots, x_N) = h_m(x_1, \ldots, x_N)$ is the $m$-th complete homogeneous symmetric polynomial, whereas $s_{(1^m, 0^{N-m})}(x_1, \ldots, x_N) = e_m(x_1, \ldots, x_N)$ is the $m$-th elementary symmetric polynomial.

The reader can refer to \cite[Ch. I]{M} for a very thorough combinatorial study of Schur polynomials.
Below we only list the properties that will be used in this article.

\begin{prop}[branching rule; \cite{M}, Ch. I, 5.11]\label{prop:branchingrule}
For any $N\in\N$, $\lambda\in\GT_{N+1}$,
\[
s_{\lambda}(x_1, \ldots, x_N, u) = \sum_{\substack{\mu\in\GT_N\\ \mu \prec \lambda}} s_{\mu}(x_1, \ldots, x_N) u^{|\lambda| - |\mu|}.
\]
\end{prop}

\begin{lem}[label-variable duality]\label{lem:Schurduality}
For any $\lambda, \mu\in\GT_N$, we have
$$\frac{s_{\lambda}(q^{\mu_1+N-1}, \ldots, q^{\mu_{N-1}+1}, q^{\mu_N})}{s_{\lambda}(q^{N-1}, \ldots, q, 1)}
= \frac{s_{\mu}(q^{\lambda_1+N-1}, \ldots, q^{\lambda_{N-1}+1}, q^{\lambda_N})}{s_{\mu}(q^{N-1}, \ldots, q, 1)}.$$
\end{lem}

\begin{proof}
Obvious from the definition of Schur polynomials.
\end{proof}

\begin{prop}[$q$-geometric specialization; \cite{M}, Ch. I, Ex. 1]\label{prop:evaluation}
For any $\lambda\in\GT_N$,
\[
s_{\lambda}(1, q, \ldots, q^{N-1}) = q^{n(\lambda)}\prod_{1\leq i < j\leq N}\frac{1 - q^{\lambda_i - \lambda_j + j - i}}{1 - q^{j - i}}.
\]
\end{prop}

\begin{prop}[\cite{M}, Ch. I, 3.9]\label{prop:hookidentity}
For any integers $a, b \geq 0$, $N \geq b+1$, we have
\[
s_{(a+1, 1^b, 0^{N-b-1})}(x_1, \ldots, x_N) = \sum_{i=0}^b{(-1)^i h_{a+1+i}(x_1, \ldots, x_N)e_{b-i}(x_1, \ldots, x_N)}.
\]
\end{prop}

\subsection{Type B-C-D characters}\label{sec:BCqchars}

For $N\in\N$, let $G(N)$ be one of the rank $N$ classical compact Lie groups
\begin{equation*}
SO(2N+1),\hspace{.2in} Sp(N),\hspace{.2in} SO(2N).
\end{equation*}
In these cases, $G(N)$ is \textit{of type} $B, C$ or $D$, respectively.

We recall that the orthogonal group is realized as a matrix group by $O(N') := \{ A\in GL(N', \R) : AA^t = A^tA = I \}$, and the special orthogonal group $SO(N')$ is the subgroup of $O(N')$ consisting of those matrices with determinant $1$.
Similarly, the symplectic group is $$Sp(2N) := \left\{  A \in GL(2N, \C) :\ A^tBA = B , \ \ B =  \begin{bmatrix}
    0 & I_N  \\
    I_N & 0
  \end{bmatrix} \right\}$$
and the compact symplectic group (which is of our interest) is
$$Sp(N) := Sp(2N) \cap U(2N).$$

The irreducible representations of $G(N)$ are parametrized by $\GTp_N$ if $G(N)$ is of type $B$ or $C$.
We denote the character that corresponds to $\lambda\in\GTp_N$ by $so^{2N+1}_{\lambda}$ and $sp^{N}_{\lambda}$.
On the other hand, if $G(N)$ is of type $D$, then the irreducible representations of $G(N)$ are parametrized by $N$-tuples of integers $(\lambda_1, \ldots, \lambda_N)$ such that $\lambda_1 \geq \ldots \geq \lambda_{N-1} \geq |\lambda_N|$.
For any $\lambda\in\GTp_N$, denote by $so_{\lambda}^{2N}$ the \textit{reducible} character of $SO(2N)$ that is the sum of the twin characters corresponding to $\lambda^+ := (\lambda_1, \ldots, \lambda_{N-1}, \lambda_N)$ and $\lambda^- := (\lambda_1, \ldots, \lambda_{N-1}, -\lambda_N)$.

The characters $so_{\lambda}^{2N+1}, sp^{N}_{\lambda}$ and $so^{2N}_{\lambda}$ are central functions of the groups $SO(2N+1)$, $Sp(N)$ and $SO(2N)$, respectively.
Consequently, if $M$ is a matrix in any of these groups, the value of a corresponding character on $M$ depends only on the eigenvalues of $M$.
The eigenvalues of any matrix $M \in G(N)$ come in pairs $\{ z_i, z_i^{-1} \}$, $|z_i| = 1$ (any matrix in $SO(2N+1)$ has an additional eigenvalue $1$).
Therefore all the information about the characters $so_{\lambda}^{2N+1}, sp^{N}_{\lambda}, so^{2N}_{\lambda}$, is encoded in some functions of $N$ variables $z_1, \ldots, z_N$, which are written down below in \eqref{eqn:weylformula}.

\smallskip

To uniformize notation, let $\mathbb{T} := \{ \zeta \in \C : |\zeta| = 1 \}$, and for any $\lambda\in\GTp_N$, define the functions $\chi^{B}_{\lambda}, \chi^{C}_{\lambda}, \chi^{D}_{\lambda} : \mathbb{T}^N \rightarrow \C$ by
\begin{multline*}
\chi_{\lambda}^{G}(z_1, \ldots, z_N) :=\\
\begin{cases}
    so_{\lambda}^{2N+1}(M),& \text{if } G = B;\ \text{ $M\in SO(2N+1)$ has eigenvalues $\{z_i, z_i^{-1}\}_{i=1}^n \cup \{1\}$};\\
    sp_{\lambda}^{N}(M),& \text{if }G = C;\ \text{ $M\in Sp(N)$ has eigenvalues $\{z_i, z_i^{-1}\}_{i=1}^n$};\\
    so_{\lambda}^{2N}(M),& \text{if }G = D;\ \text{ $M\in SO(2N)$ has eigenvalues $\{z_i, z_i^{-1}\}_{i=1}^n$}.
\end{cases}
\end{multline*}

They are symmetric with respect to permutation of their $N$ variables and also with respect to the involutions $z_i \mapsto z_i^{-1}$, i.e., they are symmetric with respect to the natural action of the Weyl group $W_N = (\Z/2\Z)^N \rtimes S_N$.
Moreover they are Laurent polynomials in $z_1, \ldots, z_N$, i.e., they belong to the ring $\C[z_1^{\pm}, \ldots, z_N^{\pm}]$.
In fact, if we denote
\begin{equation*}
V^s(z_1, \ldots, z_N) := \prod_{1\leq i < j\leq N}{(z_i + z_i^{-1} - z_j - z_j^{-1})},
\end{equation*}
then Weyl's formulas (e.g., see \cite{B}, \cite{FH}) imply
\begin{equation}\label{eqn:weylformula}
  \chi_{\lambda}^{G}(z_1, \ldots, z_N) = \frac{1}{V^s(z_1, \ldots, z_N)}
\times\begin{cases}
\displaystyle    \det_{1\leq i, j\leq N}\left[ \frac{z_j^{\lambda_i + N - i + \half} - z_j^{-\lambda_i - N + i - \half}}{z_j^{\half}- z_j^{-\half}} \right],& \text{if } G = B,\\
\displaystyle    \det_{1\leq i, j\leq N}\left[ \frac{z_j^{\lambda_i + N - i + 1} - z_j^{-(\lambda_i + N - i + 1)}}{z_j- z_j^{-1}} \right],& \text{if }G = C,\\
\displaystyle    \det_{1\leq i, j\leq N}\left[ z_j^{\lambda_i + N - i} + z_j^{-(\lambda_i + N - i)} \right],& \text{if }G = D.
\end{cases}
\end{equation}

Thus we can (and we will) treat $\chi^G_{\lambda}(z_1, \ldots, z_N)$ either as $W_N$-symmetric Laurent polynomials or as functions on $(\C^*)^N$, via the formulas in \eqref{eqn:weylformula}.

\smallskip

\begin{prop}[branching rule]\label{prop:branchsymplectic}
For any $N\in\N$, $\lambda\in\GTp_{N+1}$, and $G \in \{B, C, D\}$, we have
\[
\chi^G_{\lambda}(z_1, \ldots, z_N, u) =
\sum_{\mu \in \GTp_N}{ \chi_{\lambda/\mu}^G(u) \chi_{\mu}^G(z_1, \ldots, z_N) },
\]
where
\begin{equation}\label{chilambdamu}
\chi_{\lambda/\mu}^G(u) := \sum_{\substack{\nu\in\GTp_{N+1} \\ \lambda \succ \nu \succ \mu}} { \tau^G(u; \lambda, \nu, \mu) \ u^{2|\nu| - |\lambda| - |\mu|} },
\end{equation}
and for $\lambda \succ \nu \succ \mu$:
\begin{eqnarray*}
\tau^B(u; \lambda, \nu, \mu) &:=&
\begin{cases}
1, & \textrm{ if } \nu_1' \leq N,\\
1 + u^{-1},& \textrm{ otherwise};
\end{cases}\\
\tau^C(u; \lambda, \nu, \mu) &:=& 1;\\
\tau^D(u; \lambda, \nu, \mu) &:=&
\begin{cases}
0, & \textrm{ if } 0 < \nu_{N+1} < \min\{ \mu_N,  \lambda_{N+1} \},\\
2, & \textrm{ if }\lambda_1' = \mu_1' = N,\\
1,& \textrm{ otherwise}.
\end{cases}
\end{eqnarray*}
\end{prop}
\begin{proof}
This proposition is a reformulation of \cite[Prop. 10.2]{P}.
More specifically, the branching of characters $\chi_{\lambda}^C$ is given in part (a) of that proposition.
This statement appeared first in \cite{Zh}.
The branching of characters $\chi_{\lambda}^B$ is given in part (b).
The branching of characters $\chi_{\lambda}^D$ is given in part (c), which discusses the characters of certain (possibly reducible) tensor representations of $SO(2N)$.
An equivalent description of the branching for the characters $\chi^D_{\lambda}$ appeared earlier in \cite{KES}.
\end{proof}

\begin{rem}\label{symmetry:inverse}
The Laurent polynomial $\chi_{\lambda/\mu}^G(u)$ is symmetric with respect to the inversion $u \leftrightarrow 1/u$,
even though this is not evident from $(\ref{chilambdamu})$. Then we have the alternative formula
\begin{equation*}
\chi_{\lambda/\mu}^G(u) = \sum_{\substack{\nu\in\GTp_{N+1} \\ \lambda \succ \nu \succ \mu}} { \tau^G(1/u; \lambda, \nu, \mu) \ u^{|\lambda| + |\mu| - 2|\nu|} }.
\end{equation*}
\end{rem}

\begin{nota}\label{nota:BC}
For the propositions below, and for the remainder of the paper, we consider the real parameter $\epsilon = \epsilon(G)$ defined by
\begin{equation*}
    \epsilon :=
\begin{cases}
    1/2,& \text{if } G = B,\\
    1,& \text{if }G = C,\\
    0,& \text{if }G = D.
\end{cases}
\end{equation*}
Also, given nonnegative signatures $\lambda, \nu\in\GTp_N$, we denote
\begin{equation*}
l_i^G := \lambda_i + N - i + \epsilon, \ n_i^G :=  \nu_i + N - i + \epsilon, \textrm{ for all }i = 1, 2, \ldots, N.
\end{equation*}
These numbers are integers (when $G = C, D$) or half-integers (when $G = B$).
We will simply write $l_1, \ldots, l_N$ (resp. $n_1, \ldots, n_N$) instead of $l_1^G, \ldots l_N^G$ (resp. $n_1^G, \ldots, n_N^G$), as the type $G \in \{B, C, D\}$ will be clear from the context.
\end{nota}

\begin{lem}[label-variable duality]\label{lem:dualitysymplectic}
For $G \in \{B, C, D\}$, and any $\lambda, \nu\in\GTp_N$, we have
\begin{equation*}
\frac{\chi^G_{\lambda}(q^{n_1}, q^{n_2}, \ldots, q^{n_N})}
{\chi^G_{\lambda}(q^{N-1+\epsilon}, q^{N-2+\epsilon}, \ldots, q^{\epsilon})}
= \frac{\chi^G_{\nu}(q^{l_1}, q^{l_2}, \ldots, q^{l_N})}
{\chi^G_{\nu}(q^{N-1+\epsilon}, q^{N-2+\epsilon}, \ldots, q^{\epsilon})}.
\end{equation*}
\end{lem}
\begin{proof}
This is a consequence of Weyl's formulas $(\ref{eqn:weylformula})$.
\end{proof}

\begin{prop}[$q$-geometric specialization]\label{prop:evalsymplectic}
For $G \in \{B, C, D\}$, and any $\lambda\in\GTp_N$, we have
\begin{eqnarray*}
\chi_{\lambda}^{B}(q^{\half}, q^{\frac{3}{2}}, \ldots, q^{N - \half}) &=&
\prod_{i=1}^N{\frac{q^{-\half l_i} - q^{\half l_i}}{q^{-\half(N - i + \half)} - q^{\half(N - i + \half)}}}
\frac{V^s(q^{l_N}, q^{l_{N-1}}, \ldots, q^{l_1})}{V^s(q^{\half}, q^{\frac{3}{2}, }\ldots, q^{N - \half})},\\
\chi_{\lambda}^C(q, q^2, \ldots, q^N) &=&
\prod_{i=1}^N{\frac{q^{-l_i} - q^{l_i}}{q^{-(N+1-i)} - q^{N+1-i}}}
\frac{V^s(q^{l_N},q^{l_{N-1}}, \ldots, q^{l_1})}{V^s(q, q^2, \ldots, q^N)},\\
\chi_{\lambda}^D(1, q, \ldots, q^{N-1}) &=&
2 \cdot \frac{V^s(q^{l_N}, q^{l_{N-1}}, \ldots, q^{l_1})}{V^s(1, q, \ldots, q^{N-1})}.
\end{eqnarray*}
\end{prop}
\begin{proof}
The formulas for $q$-geometric specializations can be derived from Weyl's denominator formulas; for instance, for types C and D, see \cite[(4.4), (4.5)]{JN}.
One also can give direct proofs; let us do it for type B.
From $(\ref{eqn:weylformula})$, we have that $\chi_{\lambda}^B(q^{\half}, \ldots, q^{N - \half})$ equals $\prod_{i=1}^{N}{(q^{\half(i - \half)} - q^{-\half(i - \half)})^{-1}}$ times $V^s(q^{\half}, q^{\frac{3}{2}}, \ldots, q^{N - \half})^{-1}$ times the determinant
\[
\det\begin{vmatrix}
y_1^{\half} - y_1^{-\half} & y_2^{\half} - y_2^{-\half} & \dots & y_N^{\half} - y_N^{-\half} \\
y_1^{\frac{3}{2}} - y_1^{-\frac{3}{2}} & y_2^{\frac{3}{2}} - y_2^{-\frac{3}{2}} & \dots & y_N^{\frac{3}{2}} - y_N^{-\frac{3}{2}} \\
\hdotsfor{4} \\
y_1^{N - \half} - y_1^{-(N - \half)} & y_2^{N - \half} - y_2^{-(N - \half)} & \dots & y_N^{N - \half} - y_N^{-(N - \half)}
\end{vmatrix},
\]
where we denoted $y_i := q^{l_i}$ for all $i$.
By elementary row operations, we see that the determinant above equals $\det_{1\leq i, j\leq N}\left[ (y_j^{\half} - y_j^{-\half})^{2i - 1} \right] = \prod_{i=1}^N{(y_i^{\half} - y_i^{-\half})}\times \det_{1\leq i, j\leq N}\left[ (y_j + y_j^{-1} - 2)^{i - 1} \right] = \prod_{i=1}^N{(y_i^{\half} - y_i^{-\half})}\times\prod_{1\leq i < j\leq N}{(y_j + y_j^{-1} - y_i - y_i^{-1})}$. We get the final answer after we combine all these factors.
\end{proof}

For the next proposition, let $N$ be a fixed positive integer, and let $G \in \{B, C, D\}$.
Define symmetric Laurent polynomials $\{H^G_m(x_1, \ldots, x_N)\}_{m \in \Z}$, $\{E^G_m(x_1, \ldots, x_N)\}_{m \in \Z}$ as follows:
\begin{align*}
H^G_m(x_1, \ldots, x_N) &:= 0, \textrm{ for all }m < 0;\\
E^G_m(x_1, \ldots, x_N) &:= 0, \textrm{ for all }m < 0 \textrm{ and }m > 2N + \mathbf{1}_{\{G = B\}};\\
H^B_m(x_1, \ldots, x_N) &:= h_m(1, x_1, x_1^{-1}, \ldots, x_N, x_N^{-1}), \textrm{ for all }m \geq 0;\\
E^B_m(x_1, \ldots, x_N) &:= e_m(1, x_1, x_1^{-1}, \ldots, x_N, x_N^{-1}), \textrm{ for all }2N+1 \geq m \geq 0;\\
H^G_m(x_1, \ldots, x_N) &:= h_m(x_1, x_1^{-1}, \ldots, x_N, x_N^{-1}), \textrm{ for all }m \geq 0,\ G \in\{C, D\};\\
E^G_m(x_1, \ldots, x_N) &:= e_m(x_1, x_1^{-1}, \ldots, x_N, x_N^{-1}), \textrm{ for all }2N \geq m \geq 0,\ G \in\{C, D\}.
\end{align*}
In these expressions, recall that $h_m$ (resp. $e_m$) is the $m$-th complete homogeneous (resp. elementary) symmetric polynomial.

\begin{prop}\label{prop:hooksymplectic}
For any integers $N, a, b$ such that $N \geq b+1$ and $a, b \geq 0$, we have
\begin{align*}
\chi_{(a+1, 1^b, 0^{N-b-1})}^B &= (H_{a+1}^B - H_{a-1}^B)E_{b}^B - \ldots + (-1)^b(H_{a+b+1}^B - H_{a-b-1}^B)E_0^B;\\
\chi_{(a+1, 1^b, 0^{N-b-1})}^C &= H_{a+1}^CE_b^C - (H_{a+2}^C + H_a^C)E_{b-1}^C + \ldots + (-1)^b(H_{a+b+1}^C + H_{a-b+1}^C)E_0^C;\\
\chi_{(a+1, 1^b, 0^{N-b-1})}^D &= \left( 2 - \mathbf{1}_{\{ b = N-1 \}} \right) \times \{ (H_{a+1}^D - H_{a-1}^D)E_{b}^D - \ldots + (-1)^b(H_{a+b+1}^D - H_{a-b-1}^D)E_0^D \}.
\end{align*}
\end{prop}

\begin{proof}
This is a well known statement, see for example \cite[(3.14), (3.15)]{ESK}.
The idea for the proof is simple: for hook Young diagrams $\lambda = (a+1, 1^b)$, expand the following Jacobi-Trudi formulas  (which can be found, for instance, also in \cite[Sec. 4]{JN} or \cite{SV}) for symplectic/orthogonal polynomials along the first row:
\begin{eqnarray}
\chi_{\lambda}^C &=& \det_{1\leq i, j \leq b+1}\left[ H^C_{\lambda_i + j - i} + (1 - \delta_{j, 1}) H^C_{\lambda_i - i - j + 2} \right];\\
\chi_{\lambda}^G &=& \det_{1\leq i, j \leq b+1}\left[ H^G_{\lambda_i + j - i} - H^G_{\lambda_i - i - j} \right], \textrm{ for }G = B, D;
\end{eqnarray}
and then identify the $b\times b$ minors with the Laurent polynomials $E_m^G$ by another application of the Jacobi-Trudi formulas.
\end{proof}

\section{Asymptotics of $q$-deformed characters}

In this section, we study the asymptotic behavior of characters as $N \rightarrow \infty$ and all but finitely many variables are specialized in a $q$-geometric series.

\subsection{Type A characters: statements of results}

\label{Section_Type_A_results}

\begin{df}\label{def:stabilization}
We denote by $\X$ the set of all doubly infinite, nondecreasing integer sequences, i.e.,
$$\X :=\{ \bft = (\dots, t_{-2}, t_{-1}, t_0, t_1, t_2, \dots) \in \Z^{\infty}  \mid \dots \leq  t_{-1} \leq t_0 \leq t_1 \leq \cdots \}.$$

Let $\bfb = (b(1), b(2), \ldots)\in\N_0^{\infty}$ be any sequence of nonnegative integers such that $\lim_{N \rightarrow \infty}{b(N)} = \lim_{N \rightarrow \infty}{(N - b(N))} = +\infty$.
A sequence of signatures $\{\lambda(N) \in \GT_N\}_{N \geq 1}$ $\bfb$-\textit{stabilizes} to $\bft\in\X$ if
\begin{equation*}
\lim_{N\rightarrow\infty}{\lambda(N)_{b(N) + 1 - i}} = t_i, \textrm{ for all } i \in \Z.
\end{equation*}
\end{df}

For $\bft\in\X$, let
\begin{multline}\label{eqn:formulalimit}
\Phi^{\bft}(x; q) :=
x(\ln{q})^2 \frac{(q; q)^2_{\infty} }{(qx, q/x; q)_{\infty}}\\
\times \left\{\int^{\frac{\pi\ii}{\ln{q}}}_{-\frac{\pi\ii}{\ln{q}}}
\frac{d v}{2\pi\ii } \int_{\LL} \frac{d u}{2\pi \ii } \frac{x^u}{1 - q^{v - u}}\prod_{i = 1}^{\infty} \frac{1- q^{t_i + i - 1 - v}}{1 - q^{t_i + i - 1 - u}}
\prod_{j = 1}^{\infty} \frac{1 - q^{ j + v - t_{1-j}}}{1 - q^{j + u - t_{1-j}}}
- \int^{\frac{\pi\ii}{2\ln{q}}}_{-\frac{\pi\ii}{2\ln{q}}} \frac{x^v}{\ln{q}}\frac{d v}{2 \pi \ii } \right\},
\end{multline}
where the $u$-contour $\LL$ consists of the directed lines $(+\infty-\frac{\pi\ii}{2\ln{q}},-\infty-\frac{\pi\ii}{2\ln{q}})$ and $(-\infty + \frac{\pi\ii}{2\ln{q}}, +\infty + \frac{\pi\ii}{2\ln{q}})$, see Figure $\ref{fig:contours}$.
The expressions $x^u$ and $x^v$ stand for $e^{u\ln{x}}$ and $e^{v\ln{x}}$, where $\ln{x}$ is taken over the branch of the logarithm with the cut along the negative real axis.

\begin{figure}
\begin{center}
\begin{tikzpicture}[decoration={markings,
mark=at position 0.6cm with {\arrow[line width=1.5pt]{>}},
mark=at position 2.6cm with {\arrow[line width=1.5pt]{>}},
mark=at position 5.6cm with {\arrow[line width=1.5pt]{>}},
mark=at position 7.5cm with {\arrow[line width=1.5pt]{>}},
}
]
\draw[help lines,->] (-4,0) -- (4,0) coordinate (xaxis);
\draw[help lines,->] (0,-2) -- (0,2) coordinate (yaxis);

\path[dash dot,draw=red,line width=1pt,postaction=decorate] (-4,-0.5) -- (4,-0.5);
\path[dash dot,draw=red,line width=1pt,postaction=decorate] (4,0.5) -- (-4,0.5);
\path[draw=blue,line width=3pt,postaction=decorate] (0,1.5) -- (0,-1.5);

\foreach \Point in {(0,1.5), (0,0.5), (0,-0.5), (0,-1.5)}{
\node at \Point {\textbullet};
}
\node at (.5, 1.5) {-$\frac{\pi}{\ln{q}}$};
\node at (.5, -1.5) {$\frac{\pi}{\ln{q}}$};
\node at (.5, 0.8) {-$\frac{\pi}{2\ln{q}}$};
\node at (.5, -0.85) {$\frac{\pi}{2\ln{q}}$};

\node[below] at (xaxis) {$\Re z$};
\node[left] at (yaxis) {$\Im z$};
\end{tikzpicture}
\end{center}
\caption{Contours for Lemma $\ref{lem:yisunrewrite}$ and formula $(\ref{eqn:formulalimit})$: the $u$-contour is red (dashed lines) and the $v$-contour is blue (thick line).}
\label{fig:contours}
\end{figure}
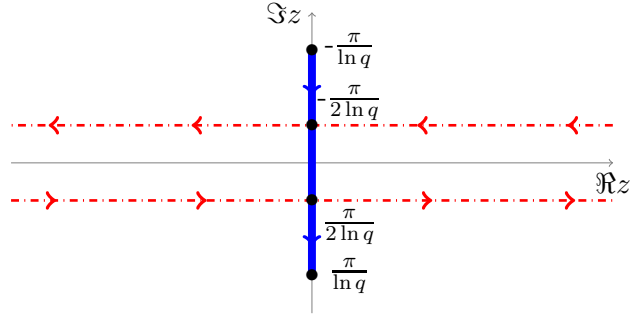

\begin{lem}\label{lem:analyticA}
For any $\bft\in\X$, the function $\Phi^{\bft}(x; q)$ defined by $(\ref{eqn:formulalimit})$ is analytic on the domain $\C \setminus ((-\infty, 0] \cup \{	q^n : n \in \Z \})$.
Moreover, it admits an analytic continuation to $\C^*$.
\end{lem}

We denote the analytic continuation of $\Phi^{\bft}(x; q)$ to $\C^*$ also by $\Phi^{\bft}(x; q)$.
For $m\in\Z$, let $A_m : \X \rightarrow \X$ be the map $A_m \bft = \bft' := (\dots \leq t_{-1}' \leq t_0' \leq t_1' \leq \dots)$, $t'_n := t_{n + 1 - m}$.
Define also the multivariate function
\begin{equation}\label{def:Phik1}
\Phi^{\bft}(x_1, \ldots, x_k; q) := \frac{q^{\frac{k(k-1)(2k-1)}{6}}}{V(x_k, \ldots, x_1)}
\det_{1\leq i, j\leq k}\left[ \Phi^{A_j\bft}(x_i q^{1 - j}; q) \prod_{\substack{1\leq s\leq k\\ s \neq j}}{(x_iq^{1 - s} - 1)} \right].
\end{equation}
By Lemma $\ref{lem:analyticA}$, the functions $\Phi^{A_j \bft}(x; q)$ are analytic on $\C^*$, therefore $\Phi^{\bft}(x_1, \ldots, x_k; q)$ defines an analytic function on $(\C^*)^k$.
In fact, the determinant in $(\ref{def:Phik1})$ has zeroes at the diagonals $x_i = x_j$, which cancel the poles coming from the Vandermonde determinant $V(x_k, \ldots, x_1)$.

\begin{thm}\label{thm:asymptoticsmultivariate}
Let $k\in\N$, and let $\bfb = (b(1), b(2), \ldots)\in\N_0^{\infty}$ be such that $\lim_{N\rightarrow\infty}{b(N)} = \lim_{N\rightarrow\infty}{(N - b(N))} = + \infty$.
Also let $\{\lambda(N)\in\GT_N\}_{N \geq 1}$ be a sequence of signatures that $\bfb$-stabilizes to some $\bft \in \X$.
Then
\begin{equation}\label{mult:limit}
\lim_{N \rightarrow \infty}{\frac{s_{\lambda(N)}(1, q, \ldots, q^{b(N)-1}, q^{b(N)}x_1, \ldots, q^{b(N)}x_k, q^{b(N)+k}, \ldots, q^{N-1})}{s_{\lambda(N)}(1, q, \ldots, q^{N-1})}}
= \Phi^{\bft}(x_1, \ldots, x_k; q)
\end{equation}
holds uniformly over $(x_1, \ldots, x_k)$ in compact subsets of $(\C^*)^k$.
Conversely, if $\{\lambda(N)\in\GT_N\}_{N \geq 1}$ is a sequence of signatures such that the limit in the left hand side of $(\ref{mult:limit})$ exists and is uniform on compact subsets of $(\C^*)^k$ and any $k\in\N$, then there exists $\bft\in\X$ such that $\{\lambda(N)\in\GT_N\}_{N \geq 1}$ $\bfb$-stabilizes to $\bft$, and $(\ref{mult:limit})$ holds.
\end{thm}

\subsection{Type A characters: proofs of results}

\label{Section_A_char_proofs}

We need the following formulas for $q$-specializations of Schur polynomials.

\begin{thm}\label{thm:gorinsun}
Let $\lambda\in\GT_N$, $b \in \{1, 2, \ldots, N - 1\}$, and $a\in\C$; then
\begin{multline}\label{eqn:gorinsun}
\frac{s_\lambda(1, q, \ldots, q^{b-1}, q^a, q^{b+1}, \ldots, q^{N-1})}{s_\lambda(1, \ldots, q^{N -  1})} =
\prod_{i=1}^{N-b-1} \frac{1 - q^{i}}{q^{a-b} - q^i}
\prod_{i=1}^{b} \frac{1 - q^i}{1 - q^{a-b+i}}  \\
\times\oint_{\{q^{\lambda_i+N-i}\}} \frac{d z}{2\pi \ii } \oint_{\{\infty\}} \frac{d w}{2 \pi \ii } \cdot \frac{{ z}^{a} {w}^{-b - 1}}{z - w}\cdot
\prod_{i = 1}^N \frac{w- q^{\lambda_i + N - i}}{z - q^{\lambda_i + N - i}}.
\end{multline}
In $(\ref{eqn:gorinsun})$, the $z$-contour encloses all poles $q^{\lambda_i + N - i}$, $i = 1, \ldots, N$, but not the origin, while the $w$-contour encloses the origin and the $z$-contour.
\end{thm}
\begin{rem}
There are some apparent singularities in the right side of $(\ref{eqn:gorinsun})$, as a function of $a$, coming from factors in the denominators.
However, the left side is a Laurent polynomial in $q^a$ and thus an entire function of $a$.
Therefore the double integral vanishes at the apparent singularities.
\end{rem}

\begin{rem}
For the special case $b = N-1$, a similar formula was proved in \cite{GP}; see also \cite{C}.
\end{rem}

\begin{proof}
The theorem was stated in \cite[(3.10)]{GS} and is similar to Theorem \ref{thm:bctype}.
We will go into the details of the proof of Theorem \ref{thm:bctype}, so we only present a sketch of proof here.

Since both sides of \eqref{eqn:gorinsun} are rational functions on $q^a$ (for the right hand side, do a residue expansion), it suffices to prove this identity for positive integers $a \geq N$.
But then we can use Lemma \ref{lem:Schurduality} to the left hand side of \eqref{eqn:gorinsun} and find that it is proportional (up to a factor not depending on $\lambda$) to $s_{(a+1-N, 1^{N-b-1}, 0^b)}(q^{\lambda_1+N-1}, \ldots, q^{\lambda_N})$.
Apply Proposition \ref{prop:hookidentity} to express this as
\begin{equation}\label{heterm}
\sum_{i=1}^{N-b-1}{h_{a+1-N+i}(\{q^{\lambda_i + N - i}\}) e_{N-b-1-i}(\{q^{\lambda_i + N - i}\})}.
\end{equation}
Then use the following well--known generating functions for the symmetric polynomials $h_n$ and $e_m$ (see \cite[Ch. I.2, (2.2) \& (2.5)]{M}):
\begin{equation*}
\sum_{n = 0}^{\infty}{h_n(x_1, \ldots, x_N)z^n} = \prod_{j=1}^N{(1 - zx_j)^{-1}};\ \ \
\sum_{n = 0}^{N}{e_n(x_1, \ldots, x_N)w^n} = \prod_{j=1}^N{(1 + wx_j)}.
\end{equation*}
Equipped with these generating functions, write each term of \eqref{heterm} as a product of two contour integrals, whose contours are small counterclockwise oriented circles.
After swapping the sum and the integration signs, we find that \eqref{heterm} equals
$$\int_{\{0\}}{ \frac{dz}{2\pi\ii}\int_{\{0\}} \frac{dw}{2\pi\ii} \ \prod_{j=1}^N{\frac{1 + wq^{\lambda_j + N - j}}{1 - zq^{\lambda_j + N - j}}} \sum_{i=1}^{N-b-1} { z^{-(a+2-N+i)} w^{-(N-b-i)} } }.$$
After some simplifications, one then arrives at the double contour integral in the right hand side of \eqref{eqn:gorinsun}.
To obtain the precise factor before the integral, one also needs Proposition \ref{prop:evaluation}.
\end{proof}

\begin{thm}\label{thm:gorinsun2}
Let $b\in\N_0$, $k, N\in\N$, be such that $b + k \leq N$, and also let $\lambda\in\GT_N$.
Then
\begin{multline}\label{eqn:gorinsun2}
\frac{s_{\lambda}(1, q, \ldots, q^{b-1}, q^bx_1, \ldots, q^bx_k, q^{b+k}, \ldots, q^{N-1})}{s_{\lambda}(1, q, \ldots, q^{N-1})}
= \frac{q^{\frac{k(k-1)(2k-1)}{6}}}{V(x_k, \ldots, x_1)}\\
\times \det_{1\leq i, j\leq k}\left[
\frac{s_{\lambda}(1, q, \ldots, q^{b+j-2}, q^bx_i, q^{b+j}, \ldots, q^{N-1})}{s_{\lambda}(1, q, \ldots, q^{N-1})}
\prod_{\substack{1\leq s\leq k\\ s \neq j}}{(x_iq^{1 - s} - 1)}
\right].
\end{multline}
\end{thm}

\begin{proof}
This theorem is similar to \cite[Prop. 3.1]{GS} and also to Theorem \ref{multivariate}.
We go into details for the proof of Theorem \ref{multivariate}, but here we only give a sketch.

Since both sides are rational functions on $x_1, \ldots, x_k$, it suffices to prove the identity for $x_1 = q^{N+a_1-b-1}, \ldots, x_k = q^{N+a_k-b-k}$, for any integers $a_1 \geq a_2 \geq \dots \geq a_k > k$.
But then we can apply Lemma \ref{lem:Schurduality} to the left hand side of \eqref{eqn:gorinsun2}, and conclude that it is proportional (up to a factor independent of $\lambda$) to $s_{(a_1, \ldots, a_k, k^{N-b-k}, 0^b)}(q^{\lambda_1+N-1}, \ldots, q^{\lambda_{N-1}+1}, q^{\lambda_N})$.
Since the partition $(a_1, \ldots, a_k, k^{N-b-k})$ has Frobenius coordinates $(a_1 - 1, \dots, a_k - k \mid N-b-1, \ldots, N-b-k)$, we can use Giambelli formula (see \cite[Ch. I.3, Example 9]{M}) to express $s_{(a_1, \ldots, a_k, k^{N-b-k}, 0^b)}(q^{\lambda_1+N-1}, \ldots, q^{\lambda_{N-1}+1}, q^{\lambda_N})$ as the determinant
\begin{equation}\label{detkk}
\det_{1\leq i, j\leq k}\left[
s_{(a_i - i + 1,\ 1^{N-b-j},\ 0^{b+j-1})}(q^{\lambda_1+N-1}, \ldots, q^{\lambda_{N-1}+1}, q^{\lambda_N})
\right].
\end{equation}
Multiply and divide the $(i, j)$--entry of by $s_{(a_i - i + 1, \ 1^{N-b-j}, \ 0^{b+j-1})}(1, q, \ldots, q^{N-1})$, so that we can apply Lemma \ref{lem:Schurduality} again; then the determinant \eqref{detkk} becomes
\begin{equation*}
\det_{1\leq i, j\leq k}\left[
\frac{s_{\lambda}(1, q, \ldots, q^{b+j-2}, q^{a_i + N - i}, q^{b+j}, \ldots, q^{N-1})}{s_{\lambda}(1, q, \ldots, q^{N-1})}
\ s_{(a_i - i + 1, \ 1^{N-b-j}, \ 0^{b+j-1})}(1, q, \ldots, q^{N-1})
\right].
\end{equation*}
Note that the $(i, j)$--entry here contains the factor $s_{\lambda}(1, q, \ldots, q^{b+j - 2}, q^{a_i + N - i}, q^{b+j}, \ldots, q^{N-1})$, which is precisely the Schur polynomial in the $(i, j)$--entry of the right hand side of \eqref{eqn:gorinsun2} after specializing $x_i = q^{N+a_i-b-i}$.
Finally, after some simplifications and using Proposition \ref{prop:evaluation}, we arrive at the desired \eqref{eqn:gorinsun2}.
\end{proof}

The following equivalent version of Theorem $\ref{thm:gorinsun}$ is needed for the proof of Theorem $\ref{thm:asymptoticsmultivariate}$.

\begin{lem}\label{lem:yisunrewrite}
Let $b, N\in\N$ be such that $b \in \{1, 2, \ldots, N-1\}$.
Let $x\in\C\setminus((-\infty, 0] \cup\{q^n : n\in\Z\})$ be such that $q^{N-b} < |x| < q^{-(b+1)}$.
Also let $\lambda\in\GT_N$ and denote \begin{gather*}
t_i^N := \lambda_{b+1-i}, \textrm{ for all } b + 1 - N \leq i \leq b.
\end{gather*}
Then
\begin{multline}\label{eqn:lemmaprelimit}
\frac{s_\lambda(1, q, \ldots, q^{b-1}, q^bx, q^{b+1}, \ldots, q^{N-1})}{s_\lambda(1, \ldots, q^{N -  1})} =
x(\ln{q})^2 \frac{ (q; q)_b (q; q)_{N - b - 1} }{ (qx; q)_b (q/x; q)_{N - b - 1} }\\
\times \left\{\int^{\frac{\pi\ii}{\ln{q}}}_{-\frac{\pi\ii}{\ln{q}}} \frac{d v}{2\pi\ii }
\int_{\LL} \frac{d u}{2\pi \ii } \frac{x^u}{1 - q^{v - u}}\prod_{i = 1}^b \frac{1- q^{t_i^N + i - 1 - v}}{1 - q^{t_i^N + i - 1 - u}}
\prod_{j = 1}^{N-b} \frac{1 - q^{-t_{1 - j}^N + j + v}}{1 - q^{-t_{1 - j}^N + j + u}}
- \int^{\frac{\pi\ii}{2\ln{q}}}_{-\frac{\pi\ii}{2\ln{q}}} \frac{x^v}{\ln{q}}\frac{d v}{2 \pi \ii } \right\}.
\end{multline}
The expressions $x^u$ and $x^v$ stand for $e^{u\ln{x}}$ and $e^{v\ln{x}}$, where $\ln{x}$ is taken over the branch of the logarithm with the cut along the negative real axis. The $u$-contour $\LL$ consists of the directed lines $(+\infty-\frac{\pi\ii}{2\ln{q}},-\infty-\frac{\pi\ii}{2\ln{q}})$ and $(-\infty + \frac{\pi\ii}{2\ln{q}}, +\infty + \frac{\pi\ii}{2\ln{q}})$, see Figure $\ref{fig:contours}$.
\end{lem}

\begin{rem}
The condition $q^{N-b} < |x| < q^{-(b+1)}$ is in place to assure that the double integral in $(\ref{eqn:lemmaprelimit})$ converges.
In fact, as $\Re u \rightarrow \infty$, $u\in\LL$, the integrand is of order $|xq^{b+1}|^{\Re u}$, whereas if $\Re u \rightarrow -\infty$, $u\in\LL$, then the integrand is of order $|x^{-1}q^{N-b}|^{-\Re u}$.
\end{rem}

\begin{proof}[Proof of Lemma $\ref{lem:yisunrewrite}$]
We massage the formula from Theorem $\ref{thm:gorinsun}$.
Let $a = b + c$, and make the change of variables $z \mapsto q^{N - b}z$, $w \mapsto q^{N - b}w$.
We obtain
\begin{equation}\label{gorinsun2}
\begin{gathered}
\frac{s_\lambda(1, q, \ldots, q^{b-1}, q^{b+c}, q^{b+1}, \ldots, q^{N-1})}{s_\lambda(1, \ldots, q^{N -  1})} =
q^c\prod_{i=1}^{N-b-1} \frac{1 - q^{i}}{1 - q^iq^{-c}}
\prod_{i=1}^{b} \frac{1 - q^i}{1 - q^iq^c}  \\
\times\oint_{\{q^{\lambda_i+b-i}\}} \frac{d z}{2\pi \ii } \oint_{|w| = q^{R}} \frac{d w}{2 \pi \ii } \cdot \frac{{ z}^{b+c} {w}^{-b - 1}}{z - w}\cdot
\prod_{i = 1}^N \frac{w- q^{\lambda_i + b - i}}{z - q^{\lambda_i + b - i}},
\end{gathered}
\end{equation}
where $R$ is chosen to be any negative real number with very large absolute value.
Next we make the change of variables $z = q^u$ and $w = q^v$, or equivalently, $u = \ln{z}/\ln{q}$ and $v = \ln{w}/\ln{q}$, where the principal branch of the logarithm is used. Then $(\ref{gorinsun2})$ becomes
\begin{equation}\label{gorinsun3}
\begin{gathered}
\frac{s_\lambda(1, q, \ldots, q^{b-1}, q^{b+c}, q^{b+1}, \ldots, q^{N-1})}{s_\lambda(1, \ldots, q^{N -  1})} =
q^c\prod_{i=1}^{N-b-1} \frac{1 - q^{i}}{1 - q^iq^{-c}}
\prod_{i=1}^{b} \frac{1 - q^i}{1 - q^iq^c}  \\
\times(\ln{q})^2\oint_{\{\lambda_i+b-i\}} \frac{d u}{2\pi \ii } \int^{R + \ii\pi/\ln{q}}_{R - \ii\pi/\ln{q}} \frac{d v}{2 \pi \ii } \cdot \frac{q^{ub - vb}q^{uc}}{1 - q^{v - u}}\cdot
\prod_{i = 1}^N \frac{q^v- q^{\lambda_i + b - i}}{q^u - q^{\lambda_i + b - i}}.
\end{gathered}
\end{equation}
In the integral formula $(\ref{gorinsun3})$, the $u$-contour is closed, counter-clockwise oriented, encloses all points $\lambda_i + b - i$, $i = 1, \ldots, N$, and is within the strip $\{z\in\C : |\Im z| < \pi \ii/\ln{q}\}$, whereas the $v$-contour is the line $[R - \pi\ii/\ln{q}, R + \pi\ii/\ln{q}]$ going downwards, and is to the left of the $u$-contour.

Let $x \in \C\setminus\left(\{0\} \cup \{q^n : n\in\Z\}\right)$ and set $c := \ln{x}/\ln{q}$, so that $q^c = x$.
Exchange the order of integration for the double integral in $(\ref{gorinsun3})$.
Also, the factors before the integral can be written in terms of $q$-Pochhammer symbols.
Then $(\ref{gorinsun3})$ is rewritten as
\begin{equation}\label{gorinsun4}
\begin{gathered}
\frac{s_\lambda(1, q, \ldots, q^{b-1}, q^bx, q^{b+1}, \ldots, q^{N-1})}{s_\lambda(1, \ldots, q^{N -  1})} =
x(\ln{q})^2\frac{ (q; q)_b(q; q)_{N - b - 1} }{ (qx; q)_b (q/x; q)_{N - b - 1} }\\
\times\int^{R + \ii\pi/\ln{q}}_{R - \ii\pi/\ln{q}} \frac{d v}{2 \pi \ii } \oint_{\{\lambda_i+b-i\}} \frac{d u}{2\pi \ii } \cdot \frac{x^u}{1 - q^{v - u}}\cdot
\prod_{i = 1}^b \frac{1- q^{t_i^N + i - 1 - v}}{1 - q^{t_i^N + i - 1 - u}}
\prod_{j = 1}^{N-b} \frac{1 - q^{-t_{1 - j}^N + j + v}}{1 - q^{-t_{1-j}^N + j + u}}.
\end{gathered}
\end{equation}

Next we modify the contours in $(\ref{gorinsun4})$.
Observe that the integrand, as a function of $u$, has poles at $u = v + 2k\pi\ii/\ln{q}$, $k\in\Z$, and at the points of the set
\[
P_N := \{t_i^N + i - 1 : i = b+1 - N, \ldots, 0, 1, \ldots, b\}= \{\lambda_i + b - i : i = 1, \ldots, N\},
\]
and their shifts by some integral multiple of $2\pi\ii / \ln{q}$.
However, the only poles enclosed by the $u$-contour are those in $P_N$.
Thus we can deform the $u$-contour to be the counter-clockwise oriented rectangle with vertices at the points $A\pm\frac{\pi\ii}{2\ln{q}}$ and $B\pm\frac{\pi\ii}{2\ln{q}}$, for any real numbers $A, B$ such that $R < A < \lambda_N + b - N < \lambda_1 + b - 1 < B$. Call such contour $\LL(A, B)$.

Let $\LL$ be the infinite contour which is described in the statement of the lemma.
We want to replace $\LL(A, B)$ by $\LL$ in the integral formula $(\ref{gorinsun4})$.
For that, we need to guarantee that the absolute value of the integrand decays sufficiently fast as $|u|\rightarrow\infty$ along $\LL$.
In fact, if $\Re u$ is very large, then $\prod_{j=1}^{N-b}{(1 - q^{-t_{1-j}^N + j + u})} \approx 1$ and the factor $x^u(1 - q^{v-u})^{-1}\prod_{i=1}^b{(1 - q^{t_i^N + i - 1 - u})^{-1}}$ decays exponentially fast as $\Re u \rightarrow \infty$ along $\LL$, because $|xq^{b+1}| < 1$.
One can obtain the same conclusion if $\Re u \rightarrow -\infty$ along $\LL$, by using instead $|x^{-1}q^{N-b}| < 1$.

When we replace $\LL(A, B)$ by $\LL$, we pick up residues at the singularities $u = v \in [R - \frac{\ii \pi}{2\ln{q}}, R + \frac{\ii \pi}{2\ln{q}}]$.
The residue is simply $x^v/\ln{q}$ because all factors in the products of $(\ref{gorinsun4})$ equal $1$ when $u = v$. We conclude
\begin{equation}\label{gorinsun5}
\begin{gathered}
\frac{s_\lambda(1, q, \ldots, q^{b-1}, q^bx, q^{b+1}, \ldots, q^{N-1})}{s_\lambda(1, \ldots, q^{N -  1})} =
x(\ln{q})^2 \frac{ (q; q)_b (q; q)_{N - b - 1} }{ (qx; q)_b (q/x; q)_{N - b - 1}}\times\\
\left\{\int^{R + \frac{\pi\ii}{\ln{q}}}_{R - \frac{\pi\ii}{\ln{q}}} \frac{d v}{2\pi\ii } \int_{\LL} \frac{d u}{2\pi \ii } \frac{x^u}{1 - q^{v - u}}\prod_{i = 1}^b \frac{1- q^{t_i^N + i - 1 - v}}{1 - q^{t_i^N + i - 1 - u}}
\prod_{j = 1}^{N-b} \frac{1 - q^{-t_{1-j}^N + j + v}}{1 - q^{-t_{1-j}^N + j + u}}
- \int^{R + \frac{\pi\ii}{2\ln{q}}}_{R - \frac{\pi\ii}{2\ln{q}}} \frac{x^v}{\ln{q}}\frac{d v}{2 \pi \ii } \right\}.
\end{gathered}
\end{equation}

Finally we need to show that $(\ref{gorinsun5})$ does not depend on $R$, and thus we can set $R = 0$, concluding the proof of the lemma. For that we subtract the second line of $(\ref{gorinsun5})$ with $R = R_1$ from the same expression with $R = R_2$ and show that the result is zero; without loss of generality, assume $R_1 < R_2$. If we do this for the second term in the second line of $(\ref{gorinsun5})$, the residue theorem yields
\begin{equation}\label{eqn:match1}
\left(- \int^{R_1 + \frac{\pi\ii}{2\ln{q}}}_{R_1 - \frac{\pi\ii}{2\ln{q}}} + \int^{R_2 + \frac{\pi\ii}{2\ln{q}}}_{R_2 - \frac{\pi\ii}{2\ln{q}}}\right) \frac{x^v}{\ln{q}}\frac{d v}{2 \pi \ii } = \left(\int^{R_2 + \frac{\pi\ii}{2\ln{q}}}_{R_1 + \frac{\pi\ii}{2\ln{q}}} - \int^{R_2 - \frac{\pi\ii}{2\ln{q}}}_{R_1 - \frac{\pi\ii}{2\ln{q}}}\right) \frac{x^v}{\ln{q}}\frac{d v}{2 \pi \ii }.
\end{equation}
For the first term in the second line of $(\ref{gorinsun5})$, the residue theorem again yields
\begin{equation}\label{eqn:subtraction1}
\begin{gathered}
\left(\int^{R_1 + \frac{\pi\ii}{\ln{q}}}_{R_1 - \frac{\pi\ii}{\ln{q}}} \frac{d v}{2\pi\ii } \int_{\LL} \frac{d u}{2\pi \ii } -  \int^{R_2 + \frac{\pi\ii}{\ln{q}}}_{R_2 - \frac{\pi\ii}{\ln{q}}} \frac{d v}{2\pi\ii } \int_{\LL} \frac{d u}{2\pi \ii }\right) {F(u, v)}\\
= \int_{\LL} \frac{d u}{2\pi \ii }\left(\int^{R_2 - \frac{\pi\ii}{\ln{q}}}_{R_1 - \frac{\pi\ii}{\ln{q}}} \frac{d v}{2\pi\ii } -  \int^{R_2 + \frac{\pi\ii}{\ln{q}}}_{R_1 + \frac{\pi\ii}{\ln{q}}} \frac{d v}{2\pi\ii }\right) {F(u, v)} + \int_{\LL}{\Res_{v = u}{F(u, v)}\frac{du}{2\pi\ii}},
\end{gathered}
\end{equation}
where
\begin{equation*}
F(u, v) := \frac{x^u}{1 - q^{v - u}}\prod_{i = 1}^b \frac{1- q^{t_i^N + i - 1 - v}}{1 - q^{t_i^N + i - 1 - u}}
\prod_{j = 1}^{N-b} \frac{1 - q^{-t_{1-j}^N + j + v}}{1 - q^{-t_{1-j}^N + j + u}}
\end{equation*}
is a function of $u, v$.
Observe that $F(u, v)$ depends on $v$ via the variable $q^v$. If $\Im v = \pm \pi\ii/\ln{q}$, we have $q^v = -q^{\Re v}$, thus the integral of $F(u, v)$ with respect to $v$ along the contour $[R_1-\frac{\pi\ii}{\ln{q}}, R_2-\frac{\pi\ii}{\ln{q}}]$ is equal to its integral along the contour $[R_1+\frac{\pi\ii}{\ln{q}}, R_2+\frac{\pi\ii}{\ln{q}}]$; consequently the first term in the second line of $(\ref{eqn:subtraction1})$ vanishes.
On the other hand, it is clear that $\Res_{v = u}{F(u, v)} = -x^u/\ln{q}$ if $R_1 < \Re u < R_2$ and $\Res_{v=u}{F(u, v)} = 0$ otherwise.
Therefore the second line of $(\ref{eqn:subtraction1})$ equals
\begin{equation}\label{eqn:match2}
\int_{u\in\LL\cap\{R_1 < \Re z < R_2\}}{(-x^u/\ln{q})\frac{du}{2\pi\ii}} = \left( - \int_{R_1 + \frac{\pi\ii}{2\ln{q}}}^{R_2 + \frac{\pi\ii}{2\ln{q}}} +  \int_{R_1 - \frac{\pi\ii}{2\ln{q}}}^{R_2 - \frac{\pi\ii}{2\ln{q}}}\right) \frac{x^u}{\ln{q}}\frac{du}{2\pi\ii}
\end{equation}
and since $(\ref{eqn:match2})$ cancels exactly $(\ref{eqn:match1})$, we deduce the claim that $(\ref{gorinsun5})$ does not depend on $R$.
\end{proof}

We prove Theorem $\ref{thm:asymptoticsmultivariate}$ and Lemma $\ref{lem:analyticA}$ simultaneously.

\begin{proof}[Proofs of Theorem $\ref{thm:asymptoticsmultivariate}$ and Lemma $\ref{lem:analyticA}$] We proceed in several steps.

\medskip

\textit{Step 1.}
We prove the first part of Lemma $\ref{lem:analyticA}$, which says that the formula $(\ref{eqn:formulalimit})$ defines an analytic function on $\C\setminus ((-\infty, 0] \cup \{q^n : n\in\Z\})$. The first line of $(\ref{eqn:formulalimit})$ is clearly analytic on $\C \setminus \{q^n : n\in\Z\}$.
As for the second line of $(\ref{eqn:formulalimit})$, the only difficulty is to argue that the double integral is an analytic function of $x$ on $\C\setminus (-\infty, 0]$.
Indeed, the factors $\prod_{i=1}^{\infty}(1 - q^{t_i + i - 1 - u})^{-1}\prod_{j=1}^{\infty}(1 - q^{j+u - t_{1-j}})^{-1}$ ensure that the integrand is exponentially small, as $|\Re u| \rightarrow \infty$, $u\in\LL$, and in particular the integral converges uniformly for $x$ in compact subsets of $\C\setminus (-\infty, 0]$. The analyticity of $\Phi^{\bft}(x; q)$ follows.

\medskip

\textit{Step 2.}
The general case $k > 1$ of Theorem $\ref{thm:asymptoticsmultivariate}$ follows from the $k = 1$ case and the multivariate formula in Theorem $\ref{thm:gorinsun2}$.
We are left to prove the second part of Lemma $\ref{lem:analyticA}$ (about the analytic continuation of $\Phi^{\bft}(x; q)$), Theorem $\ref{thm:asymptoticsmultivariate}$ for $k = 1$, and the last statement of Theorem $\ref{thm:asymptoticsmultivariate}$ (the converse statement).

\medskip

Steps 3--4 below prove the limit $(\ref{mult:limit})$ for $k = 1$, uniformly for $x$ in compact subsets of $\C\setminus (-\infty, 0]$. Step 5 shows that $\Phi^{\bft}(x; q)$ admits an analytic continuation to $\C^*$ and that the limit $(\ref{mult:limit})$ for $k = 1$ continues to hold uniformly for $x$ in compact subsets of $\C^*$.
Finally, step 6 shows the converse statement of Theorem $\ref{thm:asymptoticsmultivariate}$ and finishes the proof.

\medskip

\textit{Step 3.}
We begin with formula $(\ref{eqn:lemmaprelimit})$ from Lemma $\ref{lem:yisunrewrite}$, which is valid for $x$ in the domain $\C\setminus((-\infty, 0]\cup\{q^n : n\in\Z\})$ and $N$ large enough so that $|xq^{b(N)+1}|, |x^{-1}q^{N - b(N)}| < 1$.
We show that the prefactor of $(\ref{eqn:lemmaprelimit})$ converges to the prefactor of $(\ref{eqn:formulalimit})$; we also show that the integrand (of the first integral) of $(\ref{eqn:lemmaprelimit})$ converges pointwise to the integrand (of the first integral) of $(\ref{eqn:formulalimit})$.
In fact, the statement about the prefactors is equivalent to the limit
\begin{equation*}
\lim_{N\rightarrow\infty}{x(\ln{q})^2 \frac{ (q; q)_{b(N)} (q; q)_{N - b(N) - 1} }{ (qx; q)_{b(N)} (q/x; q)_{N - b(N) - 1}}} =
x(\ln{q})^2 \frac{(q; q)_{\infty}^2 }{(qx, q/x; q)_{\infty}},
\end{equation*}
whereas the statement about the integrands is equivalent to the pointwise limit
\begin{equation}\label{eqn:limitintegrand}
\lim_{N\rightarrow\infty}{\prod_{i = 1}^{b(N)} \frac{1- q^{t_i^N + i - 1 - v}}{1 - q^{t_i^N + i - 1 - u}}
\prod_{j = 1}^{N - b(N)} \frac{1 - q^{-t_{1-j}^N + j + v}}{1 - q^{-t_{1-j}^N + j + u}}}
= \prod_{i = 1}^{\infty} \frac{1- q^{t_i + i - 1 - v}}{1 - q^{t_i + i - 1 - u}}
\prod_{j = 1}^{\infty} \frac{1 - q^{-t_{1-j} + j + v}}{1 - q^{-t_{1-j} + j + u}}.
\end{equation}
Both are obvious: the first comes from the assumption $\lim_{N\rightarrow\infty}{b(N)} = \lim_{N\rightarrow\infty}{(N - b(N))} = +\infty$, and the second from the limits $\lim_{N \rightarrow \infty} t_i^N = \lim_{N \rightarrow \infty} \lambda(N)_{b(N) + 1 - i} = t_i$.

\medskip

\textit{Step 4.} From the claim about the integrands in the previous step, we want to conclude that
\begin{equation}\label{eqn:limitintegrals}
\begin{gathered}
\lim_{N\rightarrow\infty}{\int^{\frac{\pi\ii}{\ln{q}}}_{-\frac{\pi\ii}{\ln{q}}} \frac{d v}{2\pi\ii } \int_{\LL} \frac{d u}{2\pi \ii } \frac{x^u}{(1 - q^{v - u})}\prod_{i = 1}^{b(N)} \frac{1- q^{t_i^N + i - 1 - v}}{1 - q^{t_i^N + i - 1 - u}}
\prod_{j = 1}^{N - b(N)} \frac{1 - q^{-t_{1 - j}^N + j + v}}{1 - q^{-t_{1 - j}^N + j + u}}}\\
= \int^{\frac{\pi\ii}{\ln{q}}}_{-\frac{\pi\ii}{\ln{q}}} \frac{d v}{2\pi\ii } \int_{\LL} \frac{d u}{2\pi \ii } \frac{x^u}{(1 - q^{v - u})}
\prod_{i = 1}^{\infty} \frac{1- q^{t_i + i - 1 - v}}{1 - q^{t_i + i - 1 - u}}
\prod_{j = 1}^{\infty} \frac{1 - q^{-t_{1 - j} + j + v}}{1 - q^{-t_{1-j} + j + u}},
\end{gathered}
\end{equation}
uniformly for $x$ on compact subsets of $\C \setminus (-\infty, 0]$.
We need the dominated convergence theorem and bounds of the prelimit integrands.
The product of factors of the form $1 - q^{t_i^N + i - 1 - v}$ and $1 - q^{t_j^N + j + v}$ has a modulus that can be upper bounded by the constant $\prod_{i = 1}^{\infty} (1 + q^{t_1 + i - 1})(1 + q^{-t_0 + i})$, at least when $N$ is large enough so that $t_1^N = t_1$ and $t_0^N = t_0$.
Because $u \in \LL$, we have $|x^u| = |x|^{\Re u}e^{\pm \pi \arg(x)/2\ln{q}} \leq |x|^{\Re u}e^{-\pi^2/(4\ln{q})}$.

Because $q^{v - u}$ is purely imaginary for $u\in\LL$ and $v\in [-\pi\ii/\ln{q}, \pi\ii/\ln{q}]$, then $|1/(1 - q^{v - u})| \leq 1$.
Similarly, because $q^{u}$ is purely imaginary for $u\in\LL$, the products of factors $1/(1 - q^{-t_{1 - j}^N + j + u})$ and the factors $1/(1 - q^{-t_{1 - j}^N + j + u})$ have moduli that are bounded above by $1$.
Moreover, if $u\in\LL$ and $\Re u \rightarrow +\infty$, then the product of factors $1/(1 - q^{t_i^N + i - 1 - u})$ decays exponentially and overcomes the factor $|x|^{\Re u}$ of the previous paragraph, whereas if $u\in\LL$ and $\Re u \rightarrow -\infty$, then the product of factors $1/(1 - q^{t_{1 - j}^N + j + u})$ is the one that decays exponentially and overcomes the factor $|x|^{\Re u}$.

All the bounds above show that the integrand in the left side of $(\ref{eqn:limitintegrals})$ converge uniformly (on $u$ and $x$) to the integrand in the right side of $(\ref{eqn:limitintegrals})$. The limit $(\ref{eqn:limitintegrals})$ follows, and the limit $(\ref{mult:limit})$ of Theorem $\ref{thm:asymptoticsmultivariate}$ has been shown to hold for $k = 1$ and for $x$ belonging to compact subsets of $\C \setminus ((-\infty, 0] \cup \{q^n : n\in\Z\})$.

\medskip

\textit{Step 5.}
We show that $\Phi^{\bft}(x; q)$ admits an analytic continuation to $\C^*$ and that $(\ref{mult:limit})$ continues to hold uniformly on compact subsets of $\C^*$.
It will be convenient to use the notation
\begin{equation*}
\Phi_{\lambda}(x; q, b, N) := \frac{s_{\lambda}(1, q, \ldots, q^{b-1}, q^{b}x, q^{b+1}, \ldots, q^{N-1})}{s_{\lambda}(1, q, \ldots, q^{N-1})}.
\end{equation*}
Let $R > 1$ be an arbitrary large positive number such that $R\notin\{q^n : n\in\Z\}$. For any $x\in\C^*$ with $R^{-1} \leq |x| \leq R$, we have $|x|^n < R^{n} + R^{-n}$ for any $n\in\Z$.
By the triangle inequality and the positivity of branching coefficients of the Schur polynomials, we deduce
\begin{equation*}
\left| \Phi_{\lambda(N)}(x; q, b(N), N) \right| \leq \Phi_{\lambda(N)}(|x|; q, b(N), N) < \Phi_{\lambda(N)}(R; q, b(N), N) + \Phi_{\lambda(N)}(R^{-1}; q, b(N), N).
\end{equation*}
From our choice of $R$ and steps 1--2, the limits $\lim_{N\rightarrow\infty}{\Phi_{\lambda(N)}(R^{\pm 1}; q, b(N), N)}$ exist and thus, both
\[
\{\Phi_{\lambda(N)}(R; q, b(N), N)\}_{N \geq 1},\ \{\Phi_{\lambda(N)}(R^{-1}; q, b(N), N)\}_{N \geq 1},
\]
are bounded sequences.
Then the sequence of functions $\{\Phi_{\lambda(N)}(z; q, b(N), N)\}_{N \geq 1}$ is uniformly bounded on $\{z\in\C : R^{-1} < |z| < R\}$.
Montel's theorem implies that any subsequence of $\{\Phi_{\lambda(N)}(x; q, b(N), N)\}_{N \geq 1}$ has a subsequential limit, and the convergence is uniform on compact subsets of $\{ z\in\C : R^{-1} < |z| < R \}$.
Since $R>1$ was arbitrarily large, each subsequence of $\{\Phi_{\lambda(N)}(x; q, b(N), N)\}_{N \geq 1}$ has subsequential limits, uniformly on compact subsets of $\C^*$; say one of the limiting functions is the analytic function $\widetilde{\Phi}(x; q)$ on $\C^*$.
But we already proved $(\ref{eqn:formulalimit})$ in a restricted domain so that $\widetilde{\Phi}(x; q) = \Phi(x; q)$, for $x\in\C\setminus((-\infty, 0]\cup\{q^n : n\in\Z\})$ and such analytic continuation $\widetilde{\Phi}(x; q)$ is unique.

\medskip

\textit{Step 6.}
Assume that $\{\lambda(N)\in\GT_N\}$ is a sequence of signatures such that the limit in the left hand side of $(\ref{mult:limit})$ exists for all $k\in\N$.
The prelimit expression is a symmetric Laurent polynomial in $x_1, \ldots, x_k$, therefore it can be written as a linear combination of Schur polynomials $s_{\mu}(x_1, \ldots, x_k)$, $\mu\in\GT_k$.
Write the expansion as
\begin{equation}\label{eqn:expansionA}
\begin{gathered}
\frac{s_{\lambda(N)}(1, q, \ldots, q^{b(N)-1}, q^{b(N)}x_1, \ldots, q^{b(N)}x_k, q^{b(N)+k}, \ldots, q^{N-1})}
{s_{\lambda(N)}(1, q, \ldots, q^{N-2}, q^{N-1})}\\
= \sum_{\mu\in\GT_k} \Lambda^N_k(\lambda(N), \mu)
\frac{s_{\mu}(x_1, \ldots, x_k)}{s_{\mu}(1, q, \ldots, q^{k-1})}.
\end{gathered}
\end{equation}

Theorem $\ref{prop:branchingrule}$ shows that the coefficients $\Lambda^N_k(\lambda(N), \mu)$ of the expansion are nonnegative.
Moreover, by plugging $x_1 = 1, x_2 = q, \dots, x_k = q^{k-1}$ into $(\ref{eqn:expansionA})$, we find that $\sum_{\mu\in\GT_k}\Lambda^N_k(\lambda(N), \mu) = 1$, i.e., $\Lambda^N_k(\lambda(N), \cdot)$ is a probability measure on $\GT_k$.
The proof of Proposition $\ref{prop:construction}$ below shows that our assumption (existence of the limit in the left hand side of $(\ref{mult:limit})$ on compact subsets of $(\C^*)^k$) implies the weak convergence of the sequence $\Lambda^N_k(\lambda(N), \cdot)$, as $N$ goes to infinity, to a probability measure $M_k$ on $\GT_k$.
The proof of Theorem $\ref{thm:minimalschur}$ below then shows that the limits
\begin{equation}\label{eqn:limitLambdas}
\lim_{N \rightarrow \infty}{ \lambda(N)_{b(N)+m} }
\end{equation}
exist for all $m\in\Z$.
If we denote by $t_{1 - m}$ the limit $(\ref{eqn:limitLambdas})$, then $\dots \leq t_{-1} \leq t_0 \leq t_{1} \leq \cdots$.
In other words, $\bft := (\dots, t_{-1}, t_0, t_1, \dots)$ belongs to $\X$ and $\{\lambda(N)\in\GT_N\}_{N \geq 1}$ $\bfb$-stabilizes to $\bft$.
We can now apply the first part of the proof (Steps 1--5) to show that the limit $(\ref{mult:limit})$ holds uniformly on compact subsets of $(\C^*)^k$ and we are done.
\end{proof}

\subsection{Type B-C-D characters: statements of results}

\label{Section_Type_BCD_results}

\begin{df}\label{stabilizationsymplectic}
Let $\Y$ be the set of nondecreasing, nonnegative integer sequences, i.e.,
$$\Y := \{\bfy = (y_1, y_2, y_3, \ldots) \in \N_0^{\infty} \mid 0 \leq y_1 \leq y_2 \leq \cdots \}.$$
We say that a sequence $\{\lambda(N) \in \GTp_N\}_{N \geq 1}$ \textit{stabilizes to $\bfy\in\Y$} if
\begin{equation*}
\lim_{N\rightarrow\infty}{\lambda(N)_{N+1-i}} = y_i, \textrm{ for all }i \geq 1.
\end{equation*}
\end{df}

To state our results, we need to define some functions parametrized by elements $\bfy\in\Y$, by a type $G\in\{B, C, D\}$, and by a nonnegative integer $m\in\N_0$:
\begin{multline}\label{typeB}
\Phi_m^{\bfy, B}(x; q) :=
(\ln{q})^2 \cdot \frac{q^{\half(m + \half)} - q^{-\half(m+\half)}}{x^{\half} - x^{-\half}}
\frac{(q^{m+1}, q^{-m}; q)_m (q, q^{2m+2}; q)_{\infty} }{(q^{\half}x, q^{\half}/x; q)_m (q^{m + \frac{3}{2}}x, q^{m + \frac{3}{2}}/x; q)_{\infty}}\\
\times \left\{\int^{\frac{\pi\ii}{\ln{q}}}_{-\frac{\pi\ii}{\ln{q}}} \frac{d v}{2\pi\ii }
\int_{\LL} \frac{d u}{2\pi \ii }
\ x^u
\frac{(q^{\frac{u}{2} - (m+1)v} - q^{-\frac{u}{2} + (m+1)v})(q^{\frac{v}{2}} - q^{-\frac{v}{2}})}
{(q^{\frac{v-u}{2}} - q^{\frac{u - v}{2}})(q^{\frac{u}{2}} - q^{-\frac{u}{2}})}
\prod_{i = 1}^{\infty} \frac{(1- q^{y_i - v})(1 - q^{y_i + v})}{(1 - q^{y_i - u})(1 - q^{y_i + u})} \right.\\
\left. - \frac{1}{\ln{q}} \int^{\frac{\pi\ii}{2\ln{q}}}_{-\frac{\pi\ii}{2\ln{q}}}
x^v \left(q^{(m + \half)v} - q^{-(m + \half)v}\right) \frac{d v}{2 \pi \ii } \right\};
\end{multline}

\begin{multline}\label{typeC}
\Phi_m^{\bfy, C}(x; q) :=
(\ln{q})^2 \cdot \frac{q^{m+1} - q^{-(m+1)}}{x - x^{-1}}
\frac{(q^{m+2}, q^{-m}; q)_m (q, q^{2m+3}; q)_{\infty} }{(qx, q/x; q)_m (q^{m+2}x, q^{m+2}/x; q)_{\infty}}\\
\times \left\{\int^{\frac{\pi\ii}{\ln{q}}}_{-\frac{\pi\ii}{\ln{q}}} \frac{d v}{2\pi\ii }
\int_{\LL} \frac{d u}{2\pi \ii }
\ x^u\frac{q^{(m+1)v} - q^{-(m+1)v}}{q^{u-v} - 1}
\prod_{i = 1}^{\infty} \frac{(1- q^{y_i - v})(1 - q^{y_i + v})}{(1 - q^{y_i - u})(1 - q^{y_i + u})} \right.\\
\left. - \frac{1}{\ln{q}} \int^{\frac{\pi\ii}{2\ln{q}}}_{-\frac{\pi\ii}{2\ln{q}}}
x^v \left( q^{(m+1)v} - q^{-(m+1)v} \right) \frac{d v}{2 \pi \ii } \right\};
\end{multline}

\begin{multline}\label{typeD}
\Phi_m^{\bfy, D}(x; q) :=
\frac{\left( 2 - \mathbf{1}_{\{m=0\}} \right) (\ln{q})^2}{2} \frac{ (q^m, q^{-m}; q)_m (q, q^{2m+1}; q)_{\infty} }{(x, 1/x; q)_m (q^{m+1} x, q^{m+1}/x; q)_{\infty}}\\
\times \left\{\int^{\frac{\pi\ii}{\ln{q}}}_{-\frac{\pi\ii}{\ln{q}}} \frac{d v}{2\pi\ii }
\int_{\LL} \frac{d u}{2\pi \ii }
\ x^u\frac{q^{\frac{u}{2} - (m + \half)v} + q^{-\frac{u}{2} + (m + \half)v}}{q^{\frac{v-u}{2}} - q^{\frac{u-v}{2}}}
\prod_{i = 1}^{\infty} \frac{(1- q^{y_i - v})(1 - q^{y_i + v})}{(1 - q^{y_i - u})(1 - q^{y_i + u})} \right. \\
\left. + \frac{1}{\ln{q}} \int^{\frac{\pi\ii}{2\ln{q}}}_{-\frac{\pi\ii}{2\ln{q}}}
x^v \left( q^{mv} + q^{-mv} \right) \frac{d v}{2 \pi \ii } \right\}.
\end{multline}
In the expressions above, $x^u$, $x^v$ and $x^{\half}$ (for $G = B$) stand for $e^{u\ln{x}}$, $e^{v\ln{x}}$ and $e^{\ln{x}/2}$, where $\ln{x}$ is defined over the branch of the logarithm with the cut along the negative real axis.
The $u$-contour $\LL$ consists of the directed lines $(+\infty-\frac{\pi\ii}{2\ln{q}},-\infty-\frac{\pi\ii}{2\ln{q}})$ and $(-\infty + \frac{\pi\ii}{2\ln{q}}, +\infty + \frac{\pi\ii}{2\ln{q}})$, see Figure $\ref{fig:contours}$.

\begin{lem}\label{eqn:analyticBCD}
The functions $\Phi^{\bfy, B}_m(x; q)$, $m\in\N_0$, defined by $(\ref{typeB})$, are analytic on $\C \setminus ((-\infty, 0] \cup \{ 1 \} \cup \{ q^{n+\frac{1}{2}} : n\in\Z \})$.
The functions $\Phi^{\bfy, G}_m(x; q)$, $m\in\N_0$, $G\in\{C, D\}$, defined by $(\ref{typeC})$, $(\ref{typeD})$, are analytic on $\C \setminus ((-\infty, 0] \cup \{ q^n : n\in\Z \})$.

Moreover, all of these functions admit analytic continuations to $\C^*$, and satisfy $\Phi_m^{\bfy, G}(x; q) = \Phi_m^{\bfy, G}(1/x; q)$.
\end{lem}

\begin{thm}\label{thm:singlevariablesymplectic}
Let $\{\lambda(N)\in\GTp_N\}_{N \geq 1}$ be a sequence that stabilizes to some $\bfy\in\Y$, and let $m\in\N_0$.
Then
\begin{equation}\label{eqn:asymptotics2}
\lim_{N \rightarrow \infty}
\frac{\chi^G_{\lambda(N)}(q^{\epsilon}, \ldots, q^{m-1+\epsilon}, x, q^{m+1+\epsilon}, \ldots, q^{N-1+\epsilon})}
{\chi^G_{\lambda(N)}(q^{\epsilon}, q^{1+\epsilon}, \ldots, q^{N-1+\epsilon})}
 = \Phi_m^{\bfy, G}(x; q)
\end{equation}
holds uniformly on compact subsets of $\C^*$.
\end{thm}

For any $k\in\N$ and $G\in\{B, C, D\}$, define the constant
$$c_k^G(q) := \frac{ 1 }{ \prod_{i=1}^k (q^{i-1+2\epsilon}, q^{1-i}, q, q^{2k-2i+2\epsilon+1}; q)_{i-1} },$$
and the functions
\begin{multline}\label{def:Phik2}
\Phi^{\bfy, G}(x_1, \ldots, x_k; q) :=  c_k^G(q) \cdot
\frac{V^s(q^{\epsilon}, q^{1+\epsilon}, \ldots, q^{k - 1 + \epsilon})}{V^s(x_1, x_2, \ldots, x_k)}\\
\times \det_{1\leq i, j\leq k}\left[
\Phi^{\bfy, G}_j(x_i; q) \cdot (q^{\epsilon}x_i, q^{\epsilon}/x_i; q)_{j-1} (q^{j + \epsilon}x_i, q^{j+\epsilon}/x_i; q)_{k-j}
\right].
\end{multline}
The determinant in $(\ref{def:Phik2})$ vanishes when $x_i = x_j$ or $x_i = 1/x_j$, for some $i \neq j$.
Then the poles of $V^s(x_1, \ldots, x_k)$ are cancelled out so that the functions $\Phi^{\bfy, G}(x_1, \ldots, x_k; q)$ are analytic on $(\C^*)^k$.

\begin{thm}\label{thm:multivariablesymplectic}
Let $k\in\N$, $G \in \{B, C, D\}$, and $\{\lambda(N)\in\GTp_N\}_{N \geq 1}$ be a sequence of nonnegative signatures that stabilizes to $\bfy\in\Y$.
Then
\begin{equation}\label{eqn:multiasymptotics2}
\lim_{N \rightarrow \infty}
\frac{\chi^G_{\lambda(N)}(x_1, \ldots, x_k, q^{k+\epsilon}, \ldots, q^{N-1+\epsilon})}{\chi^G_{\lambda(N)}(q^{\epsilon}, q^{1+\epsilon}, \ldots, q^{N-1+\epsilon})}
 = \Phi^{\bfy, G}(x_1, \ldots, x_k; q)
\end{equation}
holds uniformly on compact subsets of $(\C^*)^k$.
Conversely, if $\{\lambda(N)\in\GTp_N\}_{N \geq 1}$ is a sequence of nonnegative signatures such that the limit in the left hand side of $(\ref{eqn:multiasymptotics2})$ exists and is uniform on compact subsets of $(\C^*)^k$, for any $k\in\N$, then there exists $\bfy\in\Y$ such that $\{\lambda(N)\in\GT_N\}_{N \geq 1}$ stabilizes to $\bfy$, and the limit relation $(\ref{eqn:multiasymptotics2})$ holds.
\end{thm}

\subsection{Integral representations and multivariate formulas for B-C-D type characters}

\label{Section_BC_char_integrals}

\begin{thm}\label{thm:bctype}
Let $0\leq m\leq N-1$ be integers, $\lambda\in\GTp_N$ and $a\in\C$; then
\begin{multline}\label{thm:typeB}
\frac{\chi_\lambda^B(q^{\half}, \ldots, q^{m-\half}, q^a, q^{m+\frac{3}{2}}, \ldots, q^{N - \half})}{\chi_\lambda^B(q^{\half}, q^{\frac{3}{2}}, \ldots, q^{N - \half})} =
\frac{q^{\half (m + \half)} - q^{-\half (m + \half)}}{q^{\frac{a}{2}} - q^{-\frac{a}{2}}}\\
\times\frac{(q^{m+1}, q^{-m}; q)_m (q, q^{2m+2}; q)_{N-m-1}}{(q^{\half + a}, q^{\half - a}; q)_m(q^{m + \frac{3}{2} + a}, q^{m + \frac{3}{2} - a}; q)_{N-m-1}}\\
\times\oint \frac{d z}{2\pi \ii } \oint \frac{d w}{2 \pi \ii } \frac{{z}^{a + N - \half} (z - w^{2m+2})}{w^{N+m+2}(w - z)} \frac{(w - 1)}{(z - 1)}\prod_{i = 1}^N \frac{(w- q^{\lambda_i + N + \half - i})(w- q^{-(\lambda_i + N + \half - i)})}{(z - q^{\lambda_i + N + \half - i})(z - q^{-(\lambda_i + N + \half - i)})};
\end{multline}

\begin{multline}\label{thm:typeC}
\frac{\chi_\lambda^C(q, \ldots, q^m, q^a, q^{m+2}, \ldots, q^N)}{\chi_\lambda^C(q, q^2, \ldots, q^N)} =
\frac{(q^{m+1} - q^{-(m+1)})(q^{m+2}, q^{-m}; q)_m(q, q^{2m+3}; q)_{N-m-1}}{(q^a - q^{-a})(q^{1+a}, q^{1-a}; q)_m(q^{m+2+a}, q^{m+2-a}; q)_{N-m-1}}\\
\times\oint \frac{d z}{2\pi \ii } \oint \frac{d w}{2 \pi \ii }
\frac{{z}^{a + N - 1} (1 - w^{2m+2}	)}{w^{N+m+1}(w - z)} \prod_{i = 1}^N \frac{(w- q^{\lambda_i + N + 1 - i})(w- q^{-(\lambda_i + N + 1 - i)})}{(z - q^{\lambda_i + N + 1 - i})(z - q^{-(\lambda_i + N + 1 - i)})};
\end{multline}

\begin{multline}\label{thm:typeD}
\frac{\chi_\lambda^D(1, \ldots, q^{m-1}, q^a, q^{m+1}, \ldots, q^{N-1})}{\chi_\lambda^D(1, q, \ldots, q^{N-1})} =
\frac{\left(2 - \mathbf{1}_{\{m=0\}}\right)(q^m, q^{-m}; q)_m(q, q^{2m+1}; q)_{N-m-1}}{2(q^a, q^{-a}; q)_m (q^{m+1+a}, q^{m+1-a}; q)_{N-m-1}}\\
\times\oint \frac{d z}{2\pi \ii } \oint \frac{d w}{2 \pi \ii } \frac{{z}^{a + N - 1} (z + w^{2m+1})}{w^{N+m+1}(w - z)}
\prod_{i = 1}^N \frac{(w- q^{\lambda_i + N - i})(w- q^{-(\lambda_i + N - i)})}{(z - q^{\lambda_i + N - i})(z - q^{-(\lambda_i + N - i)})}.
\end{multline}
In the right hand sides of $(\ref{thm:typeB})$, $(\ref{thm:typeC})$ and $(\ref{thm:typeD})$, the $z$-contour encloses the smallest closed interval $I\subseteq \R$ that contains all the singularities $\{q^{\pm(\lambda_i + N - i + \epsilon)} \}_{1 \leq i \leq N}$ (where $\epsilon = \half, 1, 0$ for types $B, C, D$, respectively), but it does not enclose the origin, whereas the $w$-contour encloses the $z$-contour and the origin.
\end{thm}

\begin{rem}
There are some apparent singularities in the right sides of $(\ref{thm:typeB})$, $(\ref{thm:typeC})$ and $(\ref{thm:typeD})$, as functions of $a$, coming from Pochhammer symbols in the denominators.
However, the left sides of these identities are entire functions of $a$.
Therefore the integrals in the identities vanish at the apparent singularities.
\end{rem}

\begin{rem}
If the limit $q \rightarrow 1$, these integral representations recover \cite[Thm. 3.18]{GP}, for type $C$, and the integral representations coming from \cite[Props. 7.3, 7.4]{BG} for types $B$, $D$.
\end{rem}

\begin{proof}
The proofs of the three integral representations are very similar.
We give full details for the case $G = B$, and leave the cases $G = C, D$ to the reader.

\medskip

\textit{Step 1.}
From Lemma $\ref{lem:dualitysymplectic}$ and Proposition $\ref{prop:evalsymplectic}$, for any half-integer $a \geq N + \half$ ($a-\half\in\Z$), and $0 \leq m \leq N-1$, we obtain
\begin{equation}\label{step1}
\begin{gathered}
\frac{\chi_{\lambda}^B(q^{\half}, \ldots, q^{m - \half}, q^{m+\frac{3}{2}}, \ldots, q^{N - \half}, q^a)}
{\chi_{\lambda}^B(q^{N-\half}, \ldots, q^{\half})}
= \frac{\chi^B_{(a-N+\half, 1^{N-m-1}, 0^m)}(\{q^{\lambda_i + N - i + \half} \})}{\chi^B_{(a-N+\half, 1^{N-m-1}, 0^m)}(q^{N - \half}, \ldots, q^{\half})} = (-1)^{N-m+1} \times\\
\times\frac{(q^{\half (m + \half)} - q^{-\half (m + \half)})(q^{m+1}, q^{-m}; q)_m (q, q^{2m+2}; q)_{N-m-1}}{(q^{\frac{a}{2}} - q^{-\frac{a}{2}})(q^{\half + a}, q^{\half - a}; q)_m(q^{m + \frac{3}{2} + a}, q^{m + \frac{3}{2} - a}; q)_{N-m-1}}
\chi^B_{(a-N+\half, 1^{N-m-1}, 0^m)}(\{q^{\lambda_i + N - i + \half} \}).
\end{gathered}
\end{equation}
In the next two steps, we find a double integral representation for $\chi^B_{(b+1, 1^c, 0^{N-c-1})}(y_1, \ldots, y_N)$, for $b, c \in\N_0$ with $b$ is large enough, $N \geq c+1$, and $y_1, \ldots, y_N > 0$. Then we specialize these parameters to obtain a double integral representation for $(\ref{step1})$, when $a$ is very large.

\medskip

\textit{Step 2.}
Let us work temporarily with any $b, c\in\N_0$ such that $b > c+1$, $N \geq c+1$, and any positive real numbers $y_1, \ldots, y_N > 0$.
Let us denote
\begin{equation*}
H_n := h_n(1, y_1, y_1^{-1}, \ldots, y_N, y_N^{-1}), \ n \geq 0;\ E_n := e_n(1, y_1, y_1^{-1}, \ldots, y_N, y_N^{-1}), \ 2N+1 \geq n \geq 0.
\end{equation*}
We shall use the following generating functions:
\begin{align}
\sum_{n = 0}^{\infty}{H_n v^n} &= \frac{1}{1 - v}\prod_{i = 1}^N{\frac{1}{(1 - vy_i)(1 - vy_i^{-1})}};\label{generatingH}\\
\sum_{n = 0}^{2N+1}{E_n v^n} &= (1+v)\prod_{i=1}^N{(1+vy_i)(1+vy_i^{-1})}.\label{generatingE}
\end{align}
From Proposition $\ref{prop:hooksymplectic}$, and $(\ref{generatingH})$, we obtain
\begin{equation}\label{step2:0}
\begin{gathered}
\chi^B_{(b+1, 1^c, 0^{N-c-1})}(y_1, \ldots, y_N) = (H_{b+1} - H_{b-1})E_c - \ldots + (-1)^c (H_{b+c+1} - H_{b-c-1})E_0\\
= \oint_{\{0\}} \frac{dv}{2\pi\ii} \frac{\left\{ (v^{-b-2} - v^{-b})E_c - \ldots + (-1)^c (v^{-b-c-2} - v^{-b+c}) E_0 \right\}}{(1 - v)\prod_{i = 1}^N(1 - v y_i)(1 - v y_i^{-1})}.
\end{gathered}
\end{equation}
The sum in brackets in $(\ref{step2:0})$ can be written as two sums, with $c + 1$ terms each.
The first of those sums is
\begin{multline*}
v^{-b-2}E_c - \ldots + (-1)^c v^{-b-c-2}E_0 = (-1)^c v^{-b-c-2} (E_0 - vE_1 + \ldots + (-v)^cE_c)\\
= (-1)^c v^{-b-c-2} \oint_{\{0\}} \frac{du}{2\pi\ii} (u - v)\prod_{i = 1}^N{(u - vy_i)(u - vy_i^{-1})}\cdot \left( u^{-2N-2} + u^{-2N-1} + \ldots + u^{-2N-2+c} \right)\\
= (-1)^c v^{-b-c-2} \oint_{\{0\}} \frac{du}{2\pi\ii} (u - v)\prod_{i = 1}^N{(u - vy_i)(u - vy_i^{-1})} u^{-2N-2} \cdot \frac{1 - u^{c+1}}{1 - u}.
\end{multline*}
To give an integral representation for the second sum, we use:
\[
E_n = E_{2N+1-n}, \textrm{ for all }0 \leq n \leq 2N+1.
\]
The second sum is
\begin{gather*}
- v^{-b}E_c - \ldots + (-1)^{c+1} v^{-b+c} E_0 = - v^{-b}E_{2N+1-c} - \ldots + (-1)^{c+1} v^{-b+c} E_{2N+1}\\
= (-1)^c v^{-b+c-2N-1} \left( (-v)^{2N+1-c} E_{2N+1-c} - \ldots + (-v)^{2N+1}E_{2N+1} \right)\\
= (-1)^c v^{-b+c-2N-1} \oint_{\{0\}} \frac{du}{2\pi\ii} (u - v)\prod_{i = 1}^N{(u - vy_i)(u - vy_i^{-1})}\cdot (u^{-c-1} + u^{-c} + \ldots + u^{-1})\\
= (-1)^c v^{-b+c-2N-1} \oint_{\{0\}} \frac{du}{2\pi\ii} (u - v)\prod_{i = 1}^N{(u - vy_i)(u - vy_i^{-1})} u^{-c-1} \cdot \frac{1 - u^{c+1}}{1 - u}.
\end{gather*}
As a result, $\chi^B_{(b+1, 1^c)}(y_1, \ldots, y_N)$ is the sum of the following two double contour integrals:
\begin{gather}
(-1)^c \oint_{\{0\}} \frac{dv}{2\pi\ii} \oint_{\{0\}} \frac{du}{2\pi\ii} \frac{u - v}{1 - v}
\prod_{i = 1}^N \frac{(u - vy_i)(u - vy_i^{-1})}{(1 - vy_i)(1 - vy_i^{-1})} u^{-2N-2} v^{-b-c-2} \cdot \frac{1 - u^{c+1}}{1 - u};\label{step2:1}\\
(-1)^c \oint_{\{0\}} \frac{dv}{2\pi\ii} \oint_{\{0\}} \frac{du}{2\pi\ii} \frac{u - v}{1 - v}
\prod_{i = 1}^N \frac{(u - vy_i)(u - vy_i^{-1})}{(1 - vy_i)(1 - vy_i^{-1})} u^{-c-1} v^{-b+c-2N-1} \cdot \frac{1 - u^{c+1}}{1 - u}.\label{step2:2}
\end{gather}

\medskip

\textit{Step 3.}
Now we rewrite the contour integral representations in $(\ref{step2:1})$ and $(\ref{step2:2})$.
Let us begin with $(\ref{step2:1})$.
The integrand is of order $O(|v|^{-2})$, as $|v| \rightarrow \infty$.
Then we can deform the $v$-contour, pass it through $\infty$, and have the new $v$-contour enclosing the singularities $1, y_1, y_1^{-1}, \ldots, y_N, y_N^{-1}$; a minus sign appears in making this deformation.
Deform also the $u$-contour and make it very large (in particular, it encloses both $0$ and $1$).
Next, break the integral $(\ref{step2:1})$ into two integrals by writing $u^{-2N-2}(1-u^{c+1})/(1-u)$ as the difference $u^{-2N-2}/(1-u) - u^{-2N+c-1}/(1-u)$.
The first integral, corresponding to $u^{-2N-2}/(1-u)$, is of order $O(|u|^{-2})$, when $|u| \rightarrow \infty$.
Thus there is no pole outside the contour, meaning that the first integral vanishes.
From these considerations, $(\ref{step2:1})$ equals
\begin{equation*}
(-1)^c \oint_{\{1, y_i, y_i^{-1}\}} \frac{dv}{2\pi\ii} \oint_{\{\infty\}} \frac{du}{2\pi\ii} \frac{u - v}{1 - v}
\prod_{i = 1}^N \frac{(u - vy_i)(u - vy_i^{-1})}{(1 - vy_i)(1 - vy_i^{-1})} \frac{u^{-2N+c-1} v^{-b-c-2}}{1 - u},
\end{equation*}
where the $v$-contour encloses $1, y_1, y_1^{-1}, \ldots, y_N, y_N^{-1}$, but not the origin, while the $u$-contour encloses both $0$ and $1$.
After making the change of variables
\begin{equation}\label{changevars}
z = 1/v, \ w = u/v,
\end{equation}
we can have the new $z$-contour enclosing the singularities $1, y_1,\ldots, y_N^{-1}$, but not around the origin, whereas the new $w$-contour encloses the $z$-contour and $0$.
Then the integral $(\ref{step2:1})$ equals
\begin{equation}\label{step3:0}
(-1)^{c+1} \oint_{\{1, y_i, y_i^{-1}\}} \frac{dz}{2\pi\ii} \oint_{\{\infty\}} \frac{dw}{2\pi\ii} \frac{w - 1}{z - 1}
\prod_{i = 1}^N \frac{(w - y_i)(w - y_i^{-1})}{(z - y_i)(z - y_i^{-1})}  \frac{z^{2N+b+1}w^{-2N+c-1}}{z - w}.
\end{equation}
For the integral $(\ref{step2:2})$, follow the same procedure.
Pass the $v$-contour though infinity so that it encloses the singularities $1, y_1, \ldots, y_N^{-1}$, but not the origin; a negative sign appears.
Break the integral into two by using $u^{-c-1}(1 - u^{c+1})/(1 - u) = u^{-c-1}/(1-u) - 1/(1-u)$, and get rid of the second integral because no singularity at $0$ remains.
Then do the change of variables $(\ref{changevars})$ again.
After these simplifications, $(\ref{step2:2})$ equals
\begin{equation}\label{step3}
(-1)^c \oint_{\{1, y_i, y_i^{-1}\}} \frac{dz}{2\pi\ii} \oint_{\{0\}} \frac{dw}{2\pi\ii} \frac{w - 1}{z - 1}
\prod_{i = 1}^N \frac{(w - y_i)(w - y_i^{-1})}{(z - y_i)(z - y_i^{-1})} \cdot \frac{z^{2N+b} w^{-c-1}}{z - w}.
\end{equation}
The integrand, as a function of $w$, has only $w = z$ as a singularity.
Thus by expanding the $w$-contour, so that it swallows the $z$-contour, we pick up the residue $w = z$.
In enlarging the $w$-contour, the total residue that we pick up is (when $w = z$, the residue of the integrand in $(\ref{step3})$ is $-z^{2N+b-c-1}$):
\begin{gather*}
(-1)^{c+1} \oint_{\{1, y_i, y_i^{-1}\}} z^{2N+b-c-1} \frac{dz}{2\pi\ii} = 0.
\end{gather*}
Thus $(\ref{step3})$ does not change if the $w$-contour in that formula is replaced by one that encloses both the origin and the $z$-contour.
We conclude that $\chi^B_{(b+1, 1^c, 0^{N-c-1})}(y_1, \ldots, y_N)$ equals the sum of $(\ref{step3:0})$ and $(\ref{step3})$, the latter of which has the enlarged $w$-contour instead.

\medskip

\textit{Step 4.}
In step 3, specialized at $b = a - N - \half$, $c = N - m - 1$, and $y_i = q^{\lambda_i + N - i + \half}$, we found
\begin{gather*}
\chi^B_{(a-N+\half, 1^{N-m-1}, 0^{m})}(\{q^{\lambda_i + N - i + \half}\}_{i = 1, \ldots, N}) = (-1)^{N-m+1}
\oint_{\{1, q^{\pm(\lambda_i + N - i + \half)}\}} \frac{dz}{2\pi\ii} \oint_{\{0\}} \frac{dw}{2\pi\ii}\\
\left\{
\frac{w - 1}{z - 1} \prod_{i = 1}^N \frac{(w - q^{\lambda_i + N - i + \half})(w - q^{-(\lambda_i + N - i + \half)})}{(z - q^{\lambda_i + N - i + \half})(z - q^{-(\lambda_i + N - i + \half)})} \cdot \frac{z^{N+a-\half}w^{-N+m} - z^{N+a+\half}w^{-N-m-2}}{z - w}
\right\}.
\end{gather*}
Plugging the above formula into $(\ref{step1})$, and making simple algebraic manipulations, we finally arrive at the desired identity $(\ref{thm:typeB})$.
Recall that, in step 1, we assumed that $a$ was a very large positive half-integer.
However, observe that both sides of our identity are rational functions of $q^a$ (for the right side, this is seen after a residue expansion), so the identity $(\ref{thm:typeB})$ follows for all $a\in\C$.
\end{proof}

Recall the Frobenius coordinates for a partition $\lambda$.
Let $d$ be the length of the main diagonal of the Young diagram corresponding to $\lambda$.
The Frobenius coordinates $(a_1, \ldots, a_d \mid b_1, \ldots, b_d)$ of $\lambda$ are
\begin{equation*}
a_i := \lambda_i - i, \ b_i := \lambda_i' - i, \textrm{ for } i = 1, 2, \ldots, d,
\end{equation*}
where $\lambda' := (\lambda_1', \lambda_2', \ldots)$ is the conjugate partition of $\lambda$.
For example, the ``hook-shaped'' partition $\lambda = (a+1, 1^b)$ has Frobenius coordinates $(a \mid b)$.

\begin{prop}[Frobenius identities]\label{thm:frobenius}
Let $N\in\N$, and $\lambda\in\GTp_N$ be a partition with Frobenius coordinates $(a_1, \ldots, a_d \mid b_1, \ldots, b_d)$.
Denote $\chi^G_{\mu} = \chi^G_{\mu}(x_1, \ldots, x_N)$, for any $\mu\in\GTp_N$.
Also, for any $a, b\in\N_0$, $N \geq b+1$, denote $(a | b) := (a+1, 1^b, 0^{N-b-1})\in\GTp_N$.
Then
\begin{equation*}
\chi^G_{\lambda} =
\frac{1}{(1 + \mathbf{1}_{\{ G = D \}})^{d-1}}
\times\det
\begin{vmatrix}
\chi^G_{(a_1 | b_1)} & \chi^G_{(a_1 | b_2)} & \dots & \chi^G_{(a_1 | b_d)}\\
\chi^G_{(a_2 | b_1)} & \chi^G_{(a_2 | b_2)} & \dots & \chi^G_{(a_2 | b_d)}\\
\hdotsfor{4} \\
\chi^G_{(a_d | b_1)} & \chi^G_{(a_d | b_2)} & \dots & \chi^G_{(a_d | b_d)}
\end{vmatrix}.
\end{equation*}
\end{prop}
\begin{proof}
This is a well-known fact in representation theory, which was derived in \cite{ESK}.
The same statement holds for a more general class of functions, called \textit{generalised Schur functions}, see \cite[Thm 3.1]{SV}.
That theorem specializes to our Proposition in three special cases, as described in \cite[Section 4.2]{SV}.
\end{proof}

\begin{thm}\label{multivariate}
For any $1\leq k\leq N$, $G \in \{B, C, D\}$, let
$$c_{k, N}^G(q) := \prod_{i=1}^{k} \frac{(q^i, q^{k+i-1+2\epsilon}; q)_{N-k}}{(q^{i-1+2\epsilon}, q^{1-i}; q)_{i-1} (q, q^{2i-1+2\epsilon}; q)_{N-i}}.$$
Then for any $\lambda\in\GTp_N$, and $x_1, \ldots, x_k\in\C^*$, we have
\begin{multline}\label{thm2:typeG}
\frac{\chi_\lambda^G(x_1, \ldots, x_k, q^{k + \epsilon}, \ldots, q^{N - 1 + \epsilon})}
{\chi_\lambda^G(q^{\epsilon}, q^{1+\epsilon}, \ldots, q^{N - 1 + \epsilon})} =
c_{k, N}^G(q) \cdot
\frac{V^s(q^{\epsilon}, q^{1+\epsilon}, \ldots, q^{k - 1 + \epsilon})}{V^s(x_1, x_2, \ldots, x_k)}\\
\times \det_{1\leq i, j\leq k}\left[
\frac{\chi_{\lambda}^G(q^{\epsilon}, \ldots, q^{j-2+\epsilon}, x_i, q^{j+\epsilon}, \ldots, q^{N-1+\epsilon})}{\chi_{\lambda}^G(q^{\epsilon}, q^{1+\epsilon}, \ldots, q^{N-1+\epsilon})}
(q^{\epsilon}x_i, q^{\epsilon}/x_i; q)_{j-1} (q^{j + \epsilon}x_i, q^{j+\epsilon}/x_i; q)_{k-j} \right].
\end{multline}
\end{thm}

\begin{proof}
Let us give details for $G = B$ (where $\epsilon = \half$), leaving the cases $G = C, D$ to the reader.
Let $m_1 \geq \dots \geq m_k$ be any integers such that $m_k \geq k$. The partition $(m_1, \ldots, m_k, k^{N-k})$ has Frobenius coordinates
$$a_i = m_i - i, \ b_i = N - i, \textrm{ for all } i = 1, \ldots, k.$$
By Lemma $\ref{lem:dualitysymplectic}$ and Proposition $\ref{thm:frobenius}$, we have
\begin{gather}
\frac{\chi^B_\lambda(q^{m_1 + N - \half}, \ldots, q^{m_k + N - k + \half}, q^{N - \half}, \ldots, q^{k+\half})}
{\chi^B_\lambda(q^{\half}, q^{\frac{3}{2}}, \ldots, q^{N-\half})}
= \frac{\chi^B_{(m_1, \ldots, m_k, k^{N-k})}(q^{\lambda_1+N-\half}, \ldots, q^{\lambda_N+\half})}
{\chi^B_{(m_1, \ldots, m_k, k^{N-k})}(q^{\half}, q^{\frac{3}{2}}, \ldots, q^{N-\half})} \nonumber\\
= \frac{1}{\chi^B_{(m_1, \ldots, m_k, k^{N-k})}(q^{\half}, q^{\frac{3}{2}}, \ldots, q^{N - \half})}
\det_{1\leq i, j\leq k} \left[ \chi^B_{(m_i - i | N - j)}(q^{\lambda_1 + N - \half}, \ldots, q^{\lambda_N + \half}) \right]. \label{mastereqn}
\end{gather}
By Lemma $\ref{lem:dualitysymplectic}$ again, we have
\begin{align*}
&\chi^B_{(m_i - i | N - j)}(q^{\lambda_1 + N - \half}, \ldots, q^{\lambda_N + \half}) =
\chi^B_{(m_i - i + 1, 1^{N-j}, 0^{j-1})}(q^{\lambda_1 + N - \half}, \ldots, q^{\lambda_N + \half})\\
&= \frac{\chi^B_{(m_i - i + 1, 1^{N-j}, 0^{j-1})}(q^{\lambda_1 + N - \half}, \ldots, q^{\lambda_N + \half})}
{\chi^B_{(m_i - i + 1, 1^{N-j}, 0^{j-1})}(q^{\half}, q^{\frac{3}{2}}, \ldots, q^{N-\half})}
\chi^B_{(m_i - i + 1, 1^{N-j}, 0^{j-1})}(q^{\half}, \ldots, q^{N-\half})\\
&= \frac{\chi^B_{\lambda}(q^{m_i + N - i + \half}, q^{N-\half}, \ldots, q^{j+\half}, q^{j-\frac{3}{2}}, \ldots, q^{\half})}
{\chi^B_{\lambda}(q^{\half}, q^{\frac{3}{2}}, \ldots, q^{N - \half})}
\chi^B_{(m_i - i + 1, 1^{N-j}, 0^{j-1})}(q^{\half}, \ldots, q^{N - \half}).
\end{align*}
From Proposition $\ref{prop:evalsymplectic}$, we obtain
\begin{multline*}
\chi^B_{(m_i - i + 1, 1^{N-j}, 0^{j-1})}(q^{\half}, \ldots, q^{N-\half})
= (-1)^{N-j}(q^{m_i + N - i + \half})^{-\half}q^{\frac{j}{2} - \frac{1}{4}}\cdot\frac{1 - q^{m_i + N - i + \half}}{1 - q^{j - \half}}\\
\times \frac{(q^{\half}q^{m_i + N - i + \half}, q^{\half}/q^{m_i + N - i + \half}; q)_{j-1}}{(q^{j}, q^{1-j}; q)_{j-1}},
\end{multline*}
as well as
\begin{multline*}
\chi^B_{(m_1, \ldots, m_k, k^{N-k})}(q^{\half}, q^{\frac{3}{2}}, \ldots, q^{N - \half}) =
(-1)^{k(N+1)}\prod_{i=1}^k \frac{q^{\half(m_i + N - i + \half)} - q^{-\half(m_i + N - i + \half)}}{q^{\half(i - \half)} - q^{-\half(i - \half)}}\\
\times \frac{V^s(q^{m_1 + N - \half}, \ldots, q^{m_k + N - k + \half})}{V^s(q^{k-\half}, \ldots, q^{\frac{3}{2}}, q^{\half})}
\prod_{i=1}^k \frac{(q^{k+\half}q^{m_i + N - i + \half}, q^{k+\half}/q^{m_i + N - i + \half}; q)_{N-k}}{(q^i, q^{k+i}; q)_{N-k}}.
\end{multline*}
Plugging these formulas into $(\ref{mastereqn})$, and after some algebraic manipulations, we deduce that $(\ref{thm2:typeG})$ holds for $x_i = q^{m_i + N - i + \half}$, $i = 1, \ldots, k$.
Since $m_1 \geq \dots \geq m_k \geq k$ are arbitrary, and both sides of $(\ref{thm2:typeG})$ are rational functions on $x_1, \ldots, x_k$ with the only poles on both sides being $x_i = 0$, $i = 1, \ldots, k$, the identity $(\ref{thm2:typeG})$ holds for all $x_1, \ldots, x_k \in \C^*$ too.
\end{proof}

\subsection{Type B-C-D characters: proofs of results}

The following lemma is the analogous to Lemma $\ref{lem:yisunrewrite}$; its proof is also very similar and we omit it.

\begin{lem}\label{lem:symporth}
Let $0\leq m\leq N-1$ be integers, $\lambda\in\GTp_N$, $G \in \{B, C, D\}$, and $x$ be a complex number in the domain $\C\setminus((-\infty, 0] \cup \{1\} \cup \{q^{n+\epsilon} : n\in\Z\})$, such that $q^N < |x| < q^{-N}$.
Then
\[
\frac{\chi^G_\lambda(q^{\epsilon}, \ldots, q^{m - 1 + \epsilon}, x, q^{m+1 + \epsilon}, \ldots, q^{N-1+\epsilon})}
{\chi^G_\lambda(q^{\epsilon}, q^{1+\epsilon}, \ldots, q^{N -  1+\epsilon})}
\]
admits the following integral representation (recall $l_1, \ldots, l_N$ are defined in Notation $\ref{nota:BC}$):
\begin{multline*}
(\ln{q})^2 \cdot \frac{q^{\half(m + \half)} - q^{-\half(m+\half)}}{x^{\half} - x^{-\half}}
\frac{(q^{m+1}, q^{-m}; q)_m (q, q^{2m+2}; q)_{N-m-1} }{(q^{\half}x, q^{\half}/x; q)_m (q^{m + \frac{3}{2}}x, q^{m + \frac{3}{2}}/x; q)_{N-m-1}}\\
\times \left\{\int^{\frac{\pi\ii}{\ln{q}}}_{-\frac{\pi\ii}{\ln{q}}} \frac{d v}{2\pi\ii }
\int_{\LL} \frac{d u}{2\pi \ii }
\ x^u
\frac{(q^{\frac{u}{2} - (m+1)v} - q^{-\frac{u}{2} + (m+1)v})(q^{\frac{v}{2}} - q^{-\frac{v}{2}})}
{(q^{\frac{v-u}{2}} - q^{\frac{u - v}{2}} )(q^{\frac{u}{2}} - q^{-\frac{u}{2}})}
\prod_{i = 1}^N \frac{(1- q^{l_i - v})(1 - q^{l_i + v})}{(1 - q^{l_i - u})(1 - q^{l_i + u})} \right.\\
\left. - \frac{1}{\ln{q}} \int^{\frac{\pi\ii}{2\ln{q}}}_{-\frac{\pi\ii}{2\ln{q}}}
x^v \left(q^{(m + \half)v} - q^{-(m + \half)v}\right) \frac{d v}{2 \pi \ii } \right\}, \ \ \textrm{ if }G = B;
\end{multline*}
\begin{multline*}
(\ln{q})^2 \cdot \frac{q^{m+1} - q^{-(m+1)}}{x - x^{-1}}
\frac{(q^{m+2}, q^{-m}; q)_m (q, q^{2m+3}; q)_{N-m-1} }{(qx, q/x; q)_m (q^{m+2}x, q^{m+2}/x; q)_{N-m-1}}\\
\times \left\{\int^{\frac{\pi\ii}{\ln{q}}}_{-\frac{\pi\ii}{\ln{q}}} \frac{d v}{2\pi\ii }
\int_{\LL} \frac{d u}{2\pi \ii }
\ x^u \frac{q^{(m+1)v} - q^{-(m+1)v}}{q^{u-v} - 1}
\prod_{i = 1}^N \frac{(1- q^{l_i - v})(1 - q^{l_i + v})}{(1 - q^{l_i - u})(1 - q^{l_i + u})} \right.\\
\left. - \frac{1}{\ln{q}} \int^{\frac{\pi\ii}{2\ln{q}}}_{-\frac{\pi\ii}{2\ln{q}}}
x^v \left( q^{(m+1)v} - q^{-(m+1)v} \right) \frac{d v}{2 \pi \ii } \right\}, \ \ \textrm{ if }G = C;
\end{multline*}
\begin{multline*}
\frac{\left( 2 - \mathbf{1}_{\{ m=0 \}} \right)(\ln{q})^2}{2} \frac{  (q^m, q^{-m}; q)_m (q, q^{2m+1}; q)_{N-m-1} }{(x, 1/x; q)_m (q^{m+1} x, q^{m+1}/x; q)_{N-m-1}}\\
\times \left\{\int^{\frac{\pi\ii}{\ln{q}}}_{-\frac{\pi\ii}{\ln{q}}} \frac{d v}{2\pi\ii }
\int_{\LL} \frac{d u}{2\pi \ii }
\ x^u\frac{q^{\frac{u}{2} - (m + \half)v} + q^{-\frac{u}{2} + (m + \half)v}}{q^{\frac{v-u}{2}} - q^{\frac{u-v}{2}}}
\prod_{i = 1}^N \frac{(1- q^{l_i - v})(1 - q^{l_i + v})}{(1 - q^{l_i - u})(1 - q^{l_i + u})} \right. \\
\left. + \frac{1}{\ln{q}} \int^{\frac{\pi\ii}{2\ln{q}}}_{-\frac{\pi\ii}{2\ln{q}}}
x^v \left( q^{mv} + q^{-mv} \right) \frac{d v}{2 \pi \ii } \right\}, \ \ \textrm{ if }G = D.
\end{multline*}
\end{lem}

\begin{rem}
The condition $q^{N} < |x| < q^{-N}$ ensures that the double integrals converge.
\end{rem}

\smallskip

\begin{proof}[Proof of Theorems $\ref{thm:singlevariablesymplectic}$, $\ref{thm:multivariablesymplectic}$, and Lemma $\ref{eqn:analyticBCD}$.]

It is not difficult to show that the formulas $(\ref{typeB})$, $(\ref{typeC})$ and $(\ref{typeD})$ define analytic functions on $\C \setminus ((-\infty, 0]\cup \{1\} \cup \{q^{n+\epsilon} : n\in\Z\})$, proving the first part of Lemma $\ref{eqn:analyticBCD}$.
Theorem $\ref{thm:multivariablesymplectic}$ is a consequence of Theorems $\ref{thm:singlevariablesymplectic}$ and $\ref{multivariate}$, thus it suffices to prove Theorem $\ref{thm:singlevariablesymplectic}$.
We claim that \eqref{eqn:asymptotics2} holds for $x$ belonging to compact subsets of $\C \setminus ((-\infty, 0] \cup \{1\} \cup \{ q^{n + \epsilon} : n\in\Z \})$.
This uses Lemma $\ref{lem:symporth}$ and follows closely the proof of Theorem $\ref{thm:asymptoticsmultivariate}$ above, so we omit the details.
Note that $x$ should avoid $(-\infty, 0]$ because $x^u$, $x^v$ and $x^{\half}$ are only defined on $\C \setminus (-\infty, 0]$; also $x$ should avoid $\{1\}\cup\{q^{n+\epsilon} : n\in\Z\}$ to avoid singularities in the right hand sides of $(\ref{thm:typeB})$, $(\ref{thm:typeC})$ and $(\ref{thm:typeD})$.
Also note that, since the Laurent polynomials $\chi_{\lambda(N)}^G(q^{\epsilon}, \cdots, q^{m-1+\epsilon}, x, q^{m+1+\epsilon}, \cdots, q^{N-1+\epsilon})$ are invariant with respect to the inversion $x \mapsto x^{-1}$, \eqref{eqn:asymptotics2} for $x\in\C\setminus\R$ implies $\Phi^{\bfy, G}_m(x; q) = \Phi^{\bfy, G}_m(1/x; q)$, for $x\in\C\setminus\R$.

\smallskip

We still have to show: (A) $\Phi^{\bfy, G}_m(x; q)$ admits an analytic continuation to $\C^*$ and the limit $(\ref{eqn:asymptotics2})$ continues to hold uniformly for $x$ in compact subsets of $\C^*$;
and (B) the converse statement  that if the limit in the left side of $(\ref{thm:multivariablesymplectic})$ exists, then $\{\lambda(N)\in\GTp_N\}_{N \geq 1}$ stabilizes to some $\bfy\in\Y$.

Observe that (A) would immediately imply $\Phi^{\bfy, G}_m(x; q) = \Phi^{\bfy, G}_m(1/x; q)$, for all $x\in\C^*$.

\smallskip

The proof of (B) is identical to the proof of the analogous result of Theorem $\ref{thm:asymptoticsmultivariate}$, except that it uses the proof of Proposition $\ref{prop:sympconstruction}$ (instead of Proposition $\ref{prop:construction}$) and the proof of Theorem $\ref{minimalSymplectic}$ (instead of Theorem $\ref{thm:minimalschur}$).
Let us finish by proving (A).

\smallskip

Let $R>1$ be a (large) real number such that $R \notin \{q^{n+\epsilon} : n\in\Z\}$.
We claim that the sequence
\begin{equation}\label{eqn:sequencesymplectic}
\left\{ \frac{\chi_{\lambda(N)}(q^{\epsilon}, \ldots, q^{m-1+\epsilon}, z, q^{m+1+\epsilon}, \ldots, q^{N-1+\epsilon})}{\chi_{\lambda(N)}(q^{\epsilon}, q^{1+\epsilon}, \ldots, q^{N-1+\epsilon})} \right\}_{N \geq 1}
\end{equation}
of holomorphic functions on $\C^*$ is uniformly bounded on $\{1/R < |z| < R\}$.
In fact, from the nonnegativity of the branching coefficients for symplectic/orthogonal characters (see Prop. $\ref{prop:branchsymplectic}$), and the fact that $1/R < |z| < R$ implies $|z|^n < R^n + 1/R^n$, for any $n\in\Z$, we obtain
\begin{gather*}
\left| \frac{\chi_{\lambda(N)}(q^{\epsilon}, \ldots, q^{m-1+\epsilon}, z, q^{m+1+\epsilon}, \ldots, q^{N-1+\epsilon})}
{\chi_{\lambda(N)}(q^{\epsilon}, q^{1+\epsilon}, \ldots, q^{N-1+\epsilon})}  \right| \leq
\frac{\chi_{\lambda(N)}(q^{\epsilon}, \ldots, q^{m-1+\epsilon}, |z|, q^{m+1+\epsilon}, \ldots, q^{N-1+\epsilon})}
{\chi_{\lambda(N)}(q^{\epsilon}, q^{1+\epsilon}, \ldots, q^{N-1+\epsilon})}\\
\leq \frac{\chi_{\lambda(N)}(\ldots, q^{m-1+\epsilon}, R, q^{m+1+\epsilon}, \ldots)}
{\chi_{\lambda(N)}(q^{\epsilon}, q^{1+\epsilon}, \ldots, q^{N-1+\epsilon})}
+ \frac{\chi_{\lambda(N)}(\ldots, q^{m-1+\epsilon}, 1/R, q^{m+1+\epsilon}, \ldots)}
{\chi_{\lambda(N)}(q^{\epsilon}, q^{1+\epsilon}, \ldots, q^{N-1+\epsilon})}.
\end{gather*}

By our choice of $R$, the limit $(\ref{eqn:asymptotics2})$ holds pointwise for $x = R, 1/R$, implying that
\[
\frac{\chi_{\lambda(N)}(\ldots, q^{m-1+\epsilon}, R, q^{m+1+\epsilon}, \ldots)}
{\chi_{\lambda(N)}(q^{\epsilon}, q^{1+\epsilon}, \ldots, q^{N-1+\epsilon})}
\textrm{ and }
\frac{\chi_{\lambda(N)}(\ldots, q^{m-1+\epsilon}, 1/R, q^{m+1+\epsilon}, \ldots)}
{\chi_{\lambda(N)}(q^{\epsilon}, q^{1+\epsilon}, \ldots, q^{N-1+\epsilon})}
\]
are both uniformly bounded sequences, thus implying our claim.
Next we can apply Montel's theorem: it implies that any subsequence of $(\ref{eqn:sequencesymplectic})$ has subsequential limits, which are analytic on $\{ 1/R < |z| < R \}$.
Each such holomorphic function agrees with $\Phi_m^{\bfy, G}(z; q)$ on an open set, so by analytic continuation, they must all be the same.
Since $R>1$ was arbitrary, the result follows.
\end{proof}

\section{Branching graphs}\label{sec:qschurgraph}

\subsection{Generalities on Branching Graphs}\label{sec:generalboundary}

We recall some general facts about branching graphs and their boundaries.
Our exposition is similar to that of \cite{Ol}; see also \cite{GO}.
Similar statements, in equivalent forms, can be also found in \cite{DF, D, W}.

For a Borel space $X$, let $\M(X)$ denote the set of probability measures on $X$.
The set $\M(X)$ is a convex subset of the vector space of finite, signed measures on $X$.
A Markov kernel between Borel spaces $K: X \dashrightarrow Y$ induces an affine map $K: \M(X) \rightarrow \M(Y)$ between convex sets, that we denote by the same letter.

Let $X_1, X_2, \ldots$ be a sequence of countable sets, equipped with their Borel structure coming from the discrete topology.
Assume we have, for each $N \geq 1$, a Markov kernel $\Lambda^{N+1}_N : X_{N+1} \dashrightarrow X_N$.
Then the chain of Markov kernels
\begin{equation*}
X_1 \dashleftarrow X_2 \dashleftarrow X_3 \dashleftarrow \dots
\end{equation*}
naturally induces the chain
\begin{equation*}
\M(X_1) \leftarrow \M(X_2) \leftarrow \M(X_3) \leftarrow \dots
\end{equation*}
of affine maps of convex sets.
Let $\varprojlim{\M(X_N)}$ be the projective limit.
As a set, it consists of \textit{coherent systems} $\{M_N \in \M(X_N)\}_{N \geq 1}$, i.e., for each $N \geq 1$, the probability measures $M_N$ and $M_{N+1}$ are related by
\[
M_{N+1}\Lambda^{N+1}_N = M_N.
\]
More explicitly, the coherency property is
\begin{equation*}
\sum_{x\in X_{N+1}} M_{N+1}(x) \Lambda^{N+1}_N(x, y) = M_N(y), \textrm{ for all } y\in X_N, \ N \geq 1.
\end{equation*}

The sequence $\{ X_N, \Lambda^{N+1}_N : N \geq 1\}$ is called a \textit{branching graph}.

Equip $\varprojlim{\M(X_N)}$ with the Borel structure arising from the embedding
\begin{equation}\label{qschur:inclusion}
\varprojlim{\M(X_N)} \hookrightarrow \prod_{N \geq 1}{\M(X_N)}.
\end{equation}
The product space in $(\ref{qschur:inclusion})$ is convex and $\varprojlim{\M(X_N)}$ is a convex subset.

\begin{df}
The set of extreme points of the convex space $\varprojlim{\M(X_N)}$ is called the \textit{(minimal) boundary} of the branching graph $\{ X_N, \Lambda^{N+1}_N : N \geq 1\}$.
Let us denote it by $\Omega$.
Given $\omega \in \Omega$, we denote by $\{ M_N^{\omega} \}_{N \geq 1}$ the corresponding coherent system.
\end{df}

\begin{thm}[\cite{Ol}, Thm. 9.2]\label{olshanski}
The set $\Omega$ is a Borel subset of $\varprojlim{\M(\GT_N)}$.
Moreover, for every coherent system $\{M_N\}_{N \geq 1}$, there exists a unique Borel probability measure $\pi\in\M(\Omega)$ such that
\[
M_N(x) = \int_{\Omega} M_N^{\omega}(x) \pi(d \omega), \textrm{ for all } x\in X_N, \ N \geq 1.
\]
Conversely, every $\pi\in\M(\Omega)$ gives a coherent system by the formula above.
The resulting map $\M(\Omega) \rightarrow \varprojlim{\M(X_N)}$ is a bijection.
\end{thm}

Let us also define the Martin boundary of the branching graph $\{ X_N, \Lambda^{N+1}_N : N \geq 1 \}$.
If we compose the Markov kernels $\Lambda^{n+1}_n: X_{n+1} \dashrightarrow X_n$, $N > n \geq k$, we obtain
\[
\Lambda^N_k := \Lambda^{N}_{N-1}\Lambda^{N-1}_{N-2}\cdots \Lambda^{k+1}_k : X_N \dashrightarrow X_k.
\]
In particular, if $N > k$ and $x(N)\in X_N$, then $\Lambda^N_k(x(N), \cdot)$ is a probability measure on $X_k$.

\begin{df}\label{def:martin}
Let $\{M_k \in \M(X_k)\}_{k \geq 1}$ be a coherent system for which there exists a sequence $\{ x(N) \in X_N \}_{N \geq 1}$ such that
\begin{equation}\label{eqn:martinconv}
M_k(y) = \lim_{N\rightarrow\infty}{\Lambda^N_k(x(N), y)}, \textrm{ for all $y\in X_k$, $k \geq 1$.}
\end{equation}
The \textit{Martin boundary} of the branching graph $\{ X_N, \Lambda^{N+1}_N : N \geq 1 \}$ is defined as the subset of $\varprojlim{\M(X_N)}$ consisting of the coherent systems $\{M_k \in \M(X_k)\}_{k \geq 1}$ for which such a sequence $\{x(N) \in X_N\}_{N \geq 1}$ exists. The Martin boundary is denoted $\Omega^{\Martin}$.
\end{df}

\begin{thm}[\cite{OO2}, Thm 6.1]\label{thm:minimalmartin}
The minimal boundary is contained in the Martin boundary.
In other words, for any $\omega\in\Omega$, there exists a sequence $\{x(N)\in X_N\}_{N \geq 1}$ such that
\[
M^{\omega}_k(y) = \lim_{N\rightarrow\infty} \Lambda^N_k(x(N), y), \textrm{ for all }y\in X_k, \ k\geq 1.
\]
\end{thm}

\subsection{Symmetric $q$-Gelfand-Tsetlin graph}

Let $\sigma = (\sigma_1, \sigma_2, \ldots) \in\{+1, -1\}^{\infty}$ be arbitrary; we will construct a branching graph associated to $\sigma$.

For each $N \geq 1$, define the matrix $[\Lambda^{N+1}_N(\lambda, \mu)]_{\lambda, \mu}$ of format $\GT_{N+1}\times\GT_N$ via the branching relations
\begin{equation}\label{eqn:link}
\begin{aligned}
\frac{s_{\lambda}(x_1, \ldots, x_N, q^{-N})}{s_{\lambda}(1, q^{-1}, \ldots, q^{-N})}
&= \sum_{\mu\in\GT_N}{\Lambda^{N+1}_N(\lambda, \mu)} \frac{s_{\mu}(x_1, \ldots, x_N)}{s_{\mu}(1, q^{-1}, \ldots, q^{1-N})}, &&\textrm{ if }\sigma_N = -1,\\
\frac{s_{\lambda}(x_1, \ldots, x_N, q^N)}{s_{\lambda}(1, q, \ldots, q^N)}
&= \sum_{\mu\in\GT_N}{\Lambda^{N+1}_N(\lambda, \mu)} \frac{s_{\mu}(x_1, \ldots, x_N)}{s_{\mu}(1, q, \ldots, q^{N-1})}, &&\textrm{ if }\sigma_N = +1.
\end{aligned}
\end{equation}

This means the following. For $\lambda\in\GT_{N+1}$, the left hand sides of \eqref{eqn:link} are symmetric Laurent polynomials in $N$ variables. The Schur polynomials $s_{\mu}(x_1, \ldots, x_N)$, $\mu\in\GT_N$, form a basis of the space of symmetric Laurent polynomials in $x_1, \ldots, x_N$. Thus we are guaranteed the existence and uniqueness of the coefficients $\Lambda^{N+1}_N(\lambda, \mu)$, $\mu\in\GT_N$, in the right hand sides of \eqref{eqn:link}.

From Proposition \eqref{prop:branchingrule} and the homogeneity of Schur polynomials, we obtain
\begin{equation}\label{Lambda:formulas}
\Lambda^{N+1}_N(\lambda, \mu) =
\mathbf{1}_{\{ \mu \prec \lambda \}} \times
\begin{cases}
\displaystyle\frac{s_{\mu}(q, q^2, \ldots, q^{N})}{s_{\lambda}(1, q, \ldots, q^N)} & \textrm{ if }\sigma_N = -1,\\
\displaystyle\frac{s_{\mu}(q^{-1}, q^{-2}, \ldots, q^{-N})}{s_{\lambda}(1, q^{-1}, \ldots, q^{-N})}&  \textrm{ if }\sigma_N = +1.
\end{cases}
\end{equation}

\begin{lem}\label{lem:stochastic}
For each $N\geq 1$, the matrix $[\Lambda^{N+1}_N(\lambda, \mu)]$ is stochastic.
\end{lem}
\begin{proof}
For any $\lambda\in\GT_{N+1}$, $\mu\in\GT_N$, let us show $\Lambda^{N+1}_N(\lambda, \mu) \geq 0$.
From Proposition $(\ref{prop:evaluation})$, it follows that $s_{\kappa}(1, q, \ldots, q^K) > 0$ if $q > 0$.
Then $\Lambda^{N+1}_N(\lambda, \mu) \geq 0$ follows from the explicit formula $(\ref{Lambda:formulas})$.
Next, for any fixed $\lambda\in\GT_{N+1}$, set $x_i = q^{i-1}$, $i = 1, 2, \ldots, N$, in the polynomial equalities $(\ref{eqn:link})$ to show $1 = \sum_{\mu}{\Lambda^{N+1}_N(\lambda, \mu)}$ for both $\sigma_N = 1$ and $\sigma_N = -1$.
\end{proof}

Instead of $\sigma = (\sigma_1, \sigma_2, \ldots) \in\{+1, -1\}^{\infty}$, we can consider the sequence $\bfb = (b(1), b(2), \ldots)\in\N_0^{\infty}$ given by
\begin{equation}\label{sigmab}
b(n) := \#\{ 1\leq i\leq n : \sigma_i = -1 \}, \textrm{ so that } n - b(n) = \#\{ 1\leq i\leq n : \sigma_i = +1 \}.
\end{equation}
The sequence $\bfb$ satisfies $b(1)\in\{0, 1\}$ and $b(n+1) - b(n) \in \{0, 1\}$ for all $n \geq 1$; conversely any sequence $\bfb$ with these properties arises from some $\sigma\in\{+1, -1\}^{\infty}$ via the relations $(\ref{sigmab})$.

\begin{df}
The sequence $\{\GT_N, \Lambda^{N+1}_N : N \geq 1\}$, defined by $(\ref{eqn:link})$, is called the \textit{symmetric $q$-Gelfand-Tsetlin graph associated to $\sigma \in \{\pm 1\}^{\infty}$} (or \textit{associated to $\bfb\in\N_0^{\infty}$}).
The minimal boundary of the symmetric $q$-Gelfand-Tsetlin graph associated to $\sigma\in\{\pm 1\}^{\infty}$ will be denoted $\Omega_q(\sigma)$, or simply $\Omega_q$,  if there is no confusion about $\sigma$.
Its Martin boundary will be denoted $\Omega^{\Martin}_q(\sigma)$, or just $\Omega_q^{\Martin}$.
\end{df}

It will be shown later in Theorems \ref{thm:minimalschur} and \ref{thm:symmetricschur} that whenever $\sigma, \sigma' \in \{\pm 1\}^{\infty}$ satisfy the generic condition \eqref{genericSigma}, the topological spaces $\Omega_q(\sigma)$ and $\Omega_q(\sigma')$ are homeomorphic. The same remark applies to the Martin boundaries.

\subsection{BC type $q$-Gelfand-Tsetlin graph}\label{sec:bcgraph}

For $G = B, C$ and $D$, set $\epsilon = \half, 1$, and $0$, respectively.
For each $N \geq 1$, define the matrix $[\Lambda^{N+1}_{N}(\lambda, \mu)]_{\lambda, \mu}$ of format $\GTp_{N+1}\times\GTp_N$ via the branching relation
\begin{equation}\label{df:symplecticlinks}
\frac{\chi^G_{\lambda}(x_1, \ldots, x_N, q^{N+\epsilon})}{\chi_{\lambda}^G(q^{\epsilon}, q^{1+\epsilon}, \ldots, q^{N+\epsilon})}
= \sum_{\mu\in\GTp_N}{ \Lambda^{N+1}_N(\lambda, \mu)\frac{\chi^G_{\mu}(x_1, x_2, \ldots, x_N)}{\chi^G_{\mu}(q^{\epsilon}, q^{1+\epsilon}, \ldots, q^{N-1+\epsilon})} }.
\end{equation}
From Proposition $\ref{prop:branchsymplectic}$, we have
\begin{equation}\label{eqn:links}
\Lambda^{N+1}_N(\lambda, \mu) = \frac{\chi_{\mu}^G(q^{\epsilon}, q^{1 + \epsilon}, \ldots, q^{N-1+\epsilon})}{\chi^G_{\lambda}(q^{\epsilon}, q^{1+\epsilon}, \ldots, q^{N+\epsilon})}\chi^G_{\lambda/\mu}(q^{N+\epsilon}).
\end{equation}
Note that the kernels $\Lambda^{N+1}_N$ depend on $G$, but we suppress it from the notation.

Repeating the proof of Lemma $\ref{lem:stochastic}$, we obtain:

\begin{lem}
For each $N \geq 1$, the matrix $[\Lambda^{N+1}_N(\lambda, \mu)]$ is stochastic.
\end{lem}

\begin{df}
The sequence $\{\GTp_N, \Lambda^{N+1}_N : N \geq 1\}$, defined by $(\ref{df:symplecticlinks})$, is called the \textit{BC type $q$-Gelfand-Tsetlin graph}.
The \textit{(minimal) boundary of the BC type $q$-Gelfand-Tsetlin graph} will be denoted $\Omega_q^G$.
The \textit{Martin boundary of the BC type $q$-Gelfand-Tsetlin graph} will be denoted $\Omega_q^{G, \Martin}$.
\end{df}

The boundary $\Omega^G_q$ depends on the type $G\in\{B, C, D\}$. We show later in Theorems \ref{minimalSymplectic} and \ref{thm:minimalsymplectic} that, as topological spaces, all three of them are homeomorphic to each other (but they correspond to three distinct families of coherent systems on the BC type $q$-Gelfand-Tsetlin graph). The same remark applies to the Martin boundaries $\Omega_q^{G, \Martin}$.

\section{Boundary of the symmetric $q$-Gelfand-Tsetlin graph}\label{boundary:qschur}

In this section, let us fix a sequence $\sigma = (\sigma_1, \sigma_2, \ldots) \in\{ \pm 1 \}^{\infty}$ satisfying the assumption
\begin{equation}\label{genericSigma}
\lim_{N\rightarrow\infty}{\#\{ 1\leq i\leq N : \sigma_i = -1 \}} = \lim_{N\rightarrow\infty}{\#\{ 1\leq i\leq N : \sigma_i = +1 \}} = +\infty.
\end{equation}
Equivalently, in terms of the sequence $\bfb$ given by $(\ref{sigmab})$,
\begin{equation*}
\lim_{N\rightarrow\infty}{b(N)} = \lim_{N\rightarrow\infty}{(N - b(N))} = +\infty.
\end{equation*}

The goal of this section is to characterize the minimal boundary of the symmetric $q$-Gelfand-Tsetlin graph associated to $\bfb$.
To proceed via the ergodic method of Vershik-Kerov, \cite{VK, V}, we need to first characterize the Martin boundary of the symmetric $q$-Gelfand-Tsetlin graph.

In subsections $\ref{sec:inject}$, $\ref{sec:prelim}$ and $\ref{sec:martin}$, we identify the Martin boundary with the set $\X$ of all doubly infinite, nondecreasing integer sequences (recall Definition $\ref{def:stabilization}$).
In Subsection $\ref{sec:minimal}$, we show that the minimal boundary coincides with the Martin boundary.
The proof of the latter statement is based on the Law of Large Numbers of Theorem $\ref{thm:LLN}$.

\subsection{A family of coherent probability measures}\label{sec:inject}

In this section, we define a map from the set $\X$ into the Martin boundary of the symmetric $q$-Gelfand-Tsetlin graph.

Let $k\in\N$ and let $C_k$ be the space of analytic functions $f(z_1, \ldots, z_k)$ on $(\C^*)^k$ for which there exist $a_{m_1, \ldots, m_k}\in\C$, $m_1, \ldots, m_k\in\Z$, such that
\begin{equation}\label{eqn:fourierexample}
f(z_1, \ldots, z_k) = \sum_{m_1, \dots, m_k\in\Z}{a_{m_1, \ldots, m_k}z_1^{m_1}\cdots z_k^{m_k}}\ \ \ \ \forall (z_1, \ldots, z_k)\in (\C^*)^k,
\end{equation}
and the right side is an absolutely and uniformly convergent sum on compact subsets of $(\C^*)^k$.
Then necessarily
\begin{equation*}
a_{m_1, \ldots, m_k} = \oint_{|z_1| = 1} \cdots \oint_{|z_k| = 1} \frac{f(z_1, \ldots, z_k)}{z_1^{m_1+1}\cdots z_k^{m_k+1}} \prod_{i=1}^k{\frac{dz_i}{2\pi\ii}}.
\end{equation*}
Let us call $C_k^{S_k}$ the linear subspace of $C_k$, consisting of functions $f(z_1, \ldots, z_k)$ for which the coefficients in the expansion $(\ref{eqn:fourierexample})$ satisfy
\begin{equation*}
a_{m_1, \ldots, m_k} = a_{m_{\sigma(1)}, \ldots, m_{\sigma(k)}}, \textrm{ for all }\sigma\in S_k, \ m_1, \ldots, m_k\in\Z.
\end{equation*}
For example, finite linear combinations of Schur polynomials $s_{\lambda}(z_1, \ldots, z_k)$, $\lambda\in\GT_k$, are elements of $C_k^{S_k}$.
For each $\mu\in\GT_k$, define the linear functional $F_{\mu} : C_k^{S_k} \rightarrow \C$ by
\begin{equation}\label{def:Fmufunctional}
F_{\mu}(f) := \frac{1}{k!}\oint_{|z_1|=1} \cdots \oint_{|z_k|=1}
f(z_1, \ldots, z_k) \overline{s_{\mu}(z_1, \ldots, z_k)}  \prod_{1\leq i<j\leq k}{|z_i - z_j|^2} \prod_{i=1}^k{\frac{dz_i}{2\pi\ii}}.
\end{equation}

\begin{lem}\label{lem:schurfunctionals}
The linear functionals $\{F_{\mu} : C_k^{S_k} \rightarrow \C\}_{\mu\in\GT_k}$, defined by $(\ref{def:Fmufunctional})$, satisfy
\begin{enumerate}
	\item $F_{\mu}(s_{\nu}(z_1, \ldots, z_k)) = \delta_{\mu, \nu}$, for all $\mu, \nu\in\GT_k$.
	\item If $\{f_N\}_{N \geq 1}\subset C_k^{S_k}$ is such that $f_N \rightarrow f$ uniformly on compact subsets of $(\C^*)^k$, then $f \in C_k^{S_k}$ and $F_{\mu}(f_N) \rightarrow F_{\mu}(f)$ for all $\mu\in\GT_k$.
	\item Let $h\in C_k^{S_k}$ be such that $F_{\mu}(h) = 0$, for all $\mu\in\GT_k$. Then $h$ is identically zero.
\end{enumerate}
\end{lem}
\begin{proof}
The first item of the lemma is well known, see for instance \cite[I.4, VI.9]{M}.

The second item is obvious.

For the third item, let $h\in C_k^{S_k}$ be such that $F_{\mu}(h) = 0$, for all $\mu\in\GT_k$.
The span of the Schur polynomials $\{s_{\mu}(z_1, \ldots, z_k)\}_{\mu\in\GT_k}$ is the linear space of symmetric Laurent polynomials on $k$ variables, and this is dense in the space of continuous, symmetric functions on the $k$--dimensional torus $\mathbb{T}^k$.
Therefore
\begin{equation}\label{innerzero}
\oint_{|z_1|=1} \cdots \oint_{|z_k|=1}
h(z_1, \ldots, z_k) \overline{g(z_1, \ldots, z_k)}  \prod_{1\leq i<j\leq k}{|z_i - z_j|^2} \prod_{i=1}^k{\frac{dz_i}{2\pi\ii}} = 0,
\end{equation}
for any continuous, symmetric function $g$ on $\mathbb{T}^k$.
Since $h$ is analytic on $(\C^*)^k$, it is continuous on $\mathbb{T}^k$.
It follows from \eqref{innerzero} that $h$ vanishes on $\mathbb{T}^k$ and, from its analyticity, $h$ also vanishes on $(\C^*)^k$.
\end{proof}

\begin{prop}\label{prop:construction}
Take any $\bft\in\X$ and let $\{ \lambda(N) \in\GT_N \}_{N \geq 1}$ be a sequence that $\bfb$-stabilizes to $\bft$.
For each $k \geq 1$, $\mu\in\GT_k$, the limit
\begin{equation*}
M_k^{\bft} (\mu) := \lim_{N \rightarrow\infty} \Lambda^N_k(\lambda(N), \mu)
\end{equation*}
exists and does not depend on the choice of $\{\lambda(N)\}_{N \geq 1}$.
Moreover, $M_k^{\bft}$ is a probability measure on $\GT_k$, and $\{M_k^{\bft}\}_{k \geq 1}$ is a coherent system.
\end{prop}

\begin{proof}
Let $\{\lambda(N)\in\GT_N\}_{N \geq 1}$ be any sequence of signatures that $\bfb$-stabilizes to $\bft$.
Let $k \geq 1$ be arbitrary, and define $b_k(N) := \max\{0,\ b(N) - b(k)\}$, so that $b_k(N) = \#\{k < i\leq N : \sigma_i = -1\}$ whenever $N > k$.
Consider the equality
\begin{equation}\label{branch1}
\begin{gathered}
\frac{s_{\lambda(N)}(1, q, \ldots, q^{b_k(N) - 1}, q^{b_k(N)}x_1, \ldots, q^{b_k(N)}x_k, q^{b_k(N)+k}, \ldots, q^{N-2}, q^{N-1})}{s_{\lambda(N)}(1, q, \ldots, q^{N-2}, q^{N-1})}\\
= \sum_{\mu\in\GT_k} \Lambda^N_k(\lambda(N), \mu)  \frac{s_{\mu}(x_1, \ldots, x_k)}{s_{\mu}(1, q, \ldots, q^{k-1})},
\end{gathered}
\end{equation}
which follows from $(\ref{eqn:link})$, by induction on $N - k$.
Let us look at the first line of $(\ref{branch1})$.
It is clear that $0 \leq b(N) - b_k(N) \leq k$ and thus
\begin{equation*}
\lim_{N\rightarrow\infty}{b_k(N)} = \lim_{N\rightarrow\infty}{(N - b_k(N))} = +\infty.
\end{equation*}
Moreover if we let
\begin{equation*}
\bfb_k := (b_k(1), b_k(2), \ldots),
\end{equation*}
then the sequence $\{ \lambda(N) \}_{N \geq 1}$ $\bfb_k$-stabilizes to $\bft(k) := (\ldots, t_{-1}^k, t_0^k, t^k_1, \ldots)$, $t_i^k := t_{i+b(k)}$, for all $i\in\Z$.
Then, from the limit \eqref{mult:limit} in Theorem $\ref{thm:asymptoticsmultivariate}$, the first line of $(\ref{branch1})$ converges to $\Phi^{\bft(k)}(x_1, \ldots, x_k; q)$ uniformly on compact subsets of $(\C^*)^k$, as $N$ goes to infinity.
Then by $(\ref{branch1})$ and Lemma $\ref{lem:schurfunctionals}$ (2), the following limit exists:
$$\lim_{N\rightarrow\infty}{\Lambda^N_k(\lambda(N), \mu)} = s_{\mu}(1, q, \ldots, q^{k-1}) F_{\mu}(\Phi^{\bft(k)}) =: a_{\mu}, \ \textrm{ for any $\mu\in\GT_k$}.$$

We show that $\{a_{\mu} : \mu\in\GT_k\}$ determines a probability measure on $\GT_k$.
Since $\Lambda^N_k(\lambda(N), \mu) \geq 0$ and $\sum_{\mu\in\GT_k}{\Lambda^N_k(\lambda(N), \mu)} = 1$ for all $N > k$, then
\begin{equation}\label{ineqas}
a_{\mu} \geq 0, \textrm{ for all } \mu\in\GT_k;\ \sum_{\mu\in\GT_k}{a_{\mu}} \leq 1.
\end{equation}
We need an extra argument to show $\sum_{\mu\in\GT_k}{a_{\mu}} = 1$.
Consider the function
\begin{equation*}
g(z_1, \ldots, z_k) := \sum_{\mu\in\GT_k} a_{\mu} \frac{s_{\mu}(z_1, \ldots, z_k)}{s_{\mu}(1, q, \ldots, q^{k-1})}.
\end{equation*}
Since $(\ref{branch1})$ converges uniformly on compact subsets, then in particular, it converges pointwise at the point $(x_1, x_2, \ldots, x_k) = (t, tq, \ldots, tq^{k-1})$, for any $t>0$.
But at this point, each summand
\begin{equation*}
\Lambda^N_k(\lambda(N), \mu) \cdot \frac{s_{\mu}(t, tq, \ldots, tq^{k-1})}{s_{\mu}(1, q, \ldots, q^{k-1})} = \Lambda^N_k(\lambda(N), \mu) \cdot t^{|\mu|}
\end{equation*}
is nonnegative.
Then the sum $\sum_{\mu\in\GT_k}{a_{\mu}t^{|\mu|}}$ must be convergent, for any $t > 0$.
Let $R>1$ be arbitrary; for any $(z_1, \ldots, z_k)\in(\C^*)^k$ such that $Rq^{i-1} < |z_i| < R^{-1}q^{i-1}$, the nonnegativity of the branching coefficients for Schur polynomials (see Proposition $\ref{prop:branchingrule}$) and homogeneity of Schur polynomials imply
\begin{equation*}
\left| \frac{s_{\mu}(z_1, \ldots, z_k)}{s_{\mu}(1, \ldots, q^{k-1})} \right| \leq
\frac{s_{\mu}(|z_1|, \ldots, |z_k|)}{s_{\mu}(1, \ldots, q^{k-1})} \leq
\frac{s_{\mu}((R+R^{-1}), \ldots, (R+R^{-1})q^{k-1})}{s_{\mu}(1, \ldots, q^{k-1})} = (R+R^{-1})^{|\mu|}.
\end{equation*}

As $R>1$ is arbitrary, we conclude that the sum defining $g$ is absolutely and uniformly convergent on compact subsets of $(\C^*)^k$, and $g$ is a well-defined analytic function in that domain.
Moreover, $g \in C_k^{S_k}$ because each $s_{\mu}$ is a symmetric polynomial.

From items (1)--(2) of Lemma $\ref{lem:schurfunctionals}$, we have $F_{\mu}(g) = a_{\mu}/s_{\mu}(1, q, \ldots, q^{k-1})$, for all $\mu\in\GT_k$.
But we also know $F_{\mu}(\Phi^{\bft(k)}) = a_{\mu}/s_{\mu}(1, q, \ldots, q^{k-1})$.
Therefore $h := g - \Phi^{\bft(k)} \in C_k^{S_k}$ is such that $F_{\mu}(h) = F_{\mu}(g) - F_{\mu}(\Phi^{\bft(k)}) = 0$, for all $\mu\in\GT_k$.
Item (3) of Lemma $\ref{lem:schurfunctionals}$ then implies $h = 0$, i.e., $g = \Phi^{\bft(k)}$.
In particular, evaluating this equality at the point $(1, q, \ldots, q^{k-1})$ and using the limit \eqref{mult:limit} of Theorem $\ref{thm:asymptoticsmultivariate}$ yields (below $\{ \lambda(N) \}_{N \geq 1}$ is any sequence that $\bfb$-stabilizes to $\bft$):
\begin{gather*}
\sum_{\mu\in\GT_k}{a_{\mu}} = g(1, q, \ldots, q^{k-1}) = \Phi^{\bft(k)}(1, q, \ldots, q^{k-1}) = \\
\lim_{N \rightarrow\infty} \left. \frac{s_{\lambda(N)}(1, \ldots, q^{b_k(N) - 1}, q^{b_k(N)}x_1, \ldots, q^{b_k(N)}x_k, q^{b_k(N) + k}, \ldots, q^{N-1})}{s_{\lambda(N)}(1, q, \ldots, q^{N-2}, q^{N-1})} \right|_{x_i = q^{i - 1} \forall i}
= \lim_{N \rightarrow \infty}{1} = 1.
\end{gather*}

Thus if we let $M^{\bft}_k(\mu) := a_{\mu}$, for any $\mu\in\GT_k$, we have that $M^{\bft}_k$ is a probability measure on $\GT_k$.
Moreover, because $\Phi^{\bft(k)}$ does not depend on the choice of $\bfb$-stabilizing sequence $\{\lambda(N)\}_{N \geq 1}$, neither does each $a_{\mu} = s_{\mu}(1, q, \ldots, q^{k-1}) F_{\mu}(\Phi^{\bft(k)})$.
The coherency is immediate from the definitions.
\end{proof}

\subsection{Concentration bound}\label{sec:prelim}

We keep the notations from the previous subsection.

\begin{prop}\label{prop:schurestimate}
Let $k\in\N_0$ and $\{\lambda(N) \in \GT_N\}_{N\geq 1}$ be arbitrary.
Further, let $\{\mu(N)\in\GT_N\}_{N \geq 1}$ be random variables such that each $\mu(N)$ is distributed according to the probability distribution $\Lambda^{N+1}_{N}(\lambda(N+1), \cdot)$.
Then there exists a constant $c > 0$, independent of $N$, such that
\begin{equation*}
\Prob\left( \mu(N)_{b(N) + i} = \lambda(N+1)_{b(N+1) + i}, \textrm{ for all } -k \leq i \leq k \right) >
\begin{cases}
1 - cq^{b(N)},& \textrm{if } \sigma_{N+1} = -1;\\
1 - cq^{N-b(N)},& \textrm{if } \sigma_{N+1} = +1.
\end{cases}
\end{equation*}
\end{prop}

We need some preparations for the proof.
Given $M\in\N$, $\nu\in\GT_{M+1}$, define two probability distributions on $\GT_M$, to be denoted $P^+(\cdot | \nu)$ and $P^-(\cdot| \nu)$, and given by
\begin{equation*}
P^+(\kappa \mid \nu) := \mathbf{1}_{\{\kappa \prec \nu\}} \frac{s_{\kappa}(q, q^2, \ldots, q^M)}{s_{\nu}(1, q, \ldots, q^M)};\
P^-(\kappa \mid \nu) := \mathbf{1}_{\{\kappa \prec \nu\}} \frac{s_{\kappa}(q^{-1}, q^{-2}, \ldots, q^{-M})}{s_{\nu}(1, q^{-1}, \ldots, q^{-M})}.
\end{equation*}
There is an explicit closed formula for the evaluation of a Schur polynomial on a $q$-geometric series, see Proposition $\ref{prop:evaluation}$; it gives
\begin{equation}\label{Pminus}
P^+(\kappa \mid \nu) = \mathbf{1}_{\{\kappa \prec \nu\}} q^{|\kappa| + n(\kappa)} \prod_{1\leq i < j\leq M}{( 1 - q^{\kappa_i - \kappa_j + j - i} )}
\times \frac{q^{-n(\nu)}(q; q)_M}{\prod_{1\leq i < j\leq M+1}{(1 - q^{\nu_i - \nu_j + j - i})}}.
\end{equation}

\begin{lem}\label{lem:est1}
Let $m\in\N$ be fixed.
There exists a constant $c' > 0$ such that, for sufficiently large $M\in\N$, any $\nu\in\GT_{M+1}$ and any integer $1 \leq i \leq M-m+1$, the following holds:

Let $\kappa\in\GT_M$ be distributed according to $P^+(\cdot \mid \nu)$. Then $$\Prob\left( \kappa_j = \nu_{j+1}, \textrm{ for all } j = i, i+1, \ldots, i+m-1 \right) > 1 - c'q^i.$$
\end{lem}
\begin{proof}
The proof is a generalization of the proof of \cite[Lem. 3.13]{GO}.

From Boole's inequality, it is clear that we only need to consider the case $m = 1$.
In other words, take any $1 \leq i \leq M$ and let us find a constant $c' > 0$ such that
\begin{equation*}
\Prob\left( \kappa_i = \nu_{i+1} \right) > 1 - c'q^i.
\end{equation*}

Set $a := \nu_i$, $b := \nu_{i+1}$, so that $a \geq \kappa_i \geq b$ almost surely.
If $a = b$, then $\Prob\left( \kappa_i = b = \nu_{i+1} \right) = 1$ and there is nothing to prove.
Then assume $a \geq b+1$, so there exist values $m\in\Z$ such that $a - b \geq m \geq 1$.
For each such value, we estimate the ratio
\begin{equation}\label{eqn:ratio}
\frac{\Prob(\kappa_i = b+m)}{\Prob(\kappa_i = b)}
\end{equation}
and show that it is of order $O(q^{mi})$.
Let us actually estimate the ratio $(\ref{eqn:ratio})$, where both probabilities are conditional on given values $\kappa_1, \ldots, \kappa_{i-1}, \kappa_{i+1}, \ldots, \kappa_M$, which are fixed, but arbitrary, and satisfy $\nu_j \geq \kappa_j \geq \nu_{j+1}$ for all relevant $j$.
Under such conditional probability, and from the formula $(\ref{Pminus})$ for $P^+(\cdot | \nu)$, we have
\begin{equation*}
\Prob(\kappa_i = d) \propto q^{id}\prod_{r=1}^{i - 1}(1 - q^{\kappa_r - d + i - r})\prod_{s=i+1}^M(1 - q^{d - \kappa_s + s - i}),
\end{equation*}
for any $a \geq d \geq b$, where the hidden constant is independent of $d$. Then for any $a - b \geq m \geq 1$, and using $q\in (0, 1)$, we obtain
\begin{equation}\label{bound1}
\begin{gathered}
\frac{\Prob(\kappa_i = b + m)}{\Prob(\kappa_i = b)} = q^{mi}\prod_{r=1}^{i-1}\frac{1 - q^{\kappa_r - b - m + i - r}}{1 - q^{\kappa_r - b + i - r}}\prod_{s = i+1}^M\frac{1 - q^{b + m - \kappa_s + s - i}}{1 - q^{b - \kappa_s + s - i}}\\
\leq \frac{q^{mi}}{\prod_{r=1}^{i - 1}(1 - q^{\kappa_r - b + i - r})\prod_{s = i+1}^M{(1 - q^{b - \kappa_s + s - i})}}.
\end{gathered}
\end{equation}

Next, since $\kappa_i - i$ is strictly decreasing on $i$ and $1 + b \leq a = \nu_i \leq \kappa_{i-1}$:
\begin{equation*}
2 \leq \kappa_{i - 1} - b + 1 < \kappa_{i - 2} - b + 2 < \ldots < \kappa_1 - b + i - 1
\end{equation*}
and thus
\begin{equation}\label{bound2}
\prod_{r=1}^{i-1}(1 - q^{\kappa_r - b + i - r}) > (1 - q^2)(1 - q^3)\cdots = (q^2; q)_{\infty}.
\end{equation}
Similarly, we can deduce $1 \leq b + 1 - \kappa_{i+1} < b + 2 - \kappa_{i+2} < \ldots < b + M - i - \kappa_M$ and
\begin{equation}\label{bound3}
\prod_{s = i+1}^M(1 - q^{b - \kappa_s + s - i}) > (1 - q)(1 - q^2)\cdots = (q; q)_{\infty}.
\end{equation}

From the bounds $(\ref{bound1})$, $(\ref{bound2})$ and $(\ref{bound3})$, we have
\begin{equation}\label{expbound}
\frac{\Prob(\kappa_i = b+m)}{\Prob(\kappa_i = b)} \leq \frac{ q^{mi} }{(q^2; q)_{\infty}(q; q)_{\infty}}\ \ \forall \ 1 \leq m \leq a-b.
\end{equation}
By adding the inequalities $(\ref{expbound})$ over $1\leq m\leq a-b$, and using $1 + q + \ldots + q^{a-b-1} < 1/(1-q)$:
\begin{equation}\label{expbound3}
\frac{\Prob(\kappa_i \geq b+1)}{\Prob(\kappa_i = b)} \leq \frac{q^i}{(q; q)^2_{\infty}}.
\end{equation}
Since $\Prob(\kappa_i \geq b+1) = 1 - \Prob(\kappa_i = b)$, we finally obtain
\begin{equation*}
\Prob(\kappa_i = b) \geq \left( 1 + \frac{q^i}{(q; q)^2_{\infty}} \right)^{-1} \geq
1 - \frac{q^i}{(q; q)^2_{\infty}} = 1 - c'q^i,
\end{equation*}
with positive constant $c' := 1/(q; q)_{\infty}^2$.
Note that the bound just proven was uniform over $\kappa_1, \ldots, \kappa_{i - 1}, \kappa_{i+1}, \ldots, \kappa_M$ that we were conditioning over. Thus the bound also holds without the conditioning, and we are done.
\end{proof}

\begin{proof}[Proof of Proposition $\ref{prop:schurestimate}$]
If $\sigma_{N+1} = -1$, note that the probability measures $\Lambda^{N+1}_N(\lambda(N+1), \cdot)$ and $P^+(\cdot | \lambda(N+1))$ are the same.
We can apply Lemma $\ref{lem:est1}$ to show that if $\mu(N)$ is $\Lambda^{N+1}_N(\lambda(N+1), \cdot)$-distributed, then
\begin{equation}\label{event}
\Prob( \mu(N)_{b(N)+i-1} = \lambda(N)_{b(N)+i} \textrm{ for all } -k \leq i\leq k ) > 1 - cq^{b(N)},
\end{equation}
for some $c>0$ independent of $N$.
As $\sigma_{N+1} = -1$, then $b(N+1)=b(N)+1$, so the left side of $(\ref{event})$ equals
\[
\Prob( \mu_{b(N)+i-1} = \lambda(N)_{b(N+1)+i-1} \textrm{ for all } -k \leq i\leq k ),
\]
giving us the desired result.

Let us proceed to the case $\sigma_{N+1} = +1$, which will be deduced from the case $\sigma_{N+1}=-1$.
From the determinantal definition of Schur polynomials, we have
\begin{equation*}
s_{\kappa}(x_1, \ldots, x_M) = s_{\kappa^-}(1/x_1, \ldots, 1/x_M), \ \kappa\in\GT_M,
\end{equation*}
where $\kappa^- := (-\kappa_M, \ldots, -\kappa_1) \in \GT_M$.
Therefore $P^-(\kappa | \nu) := P^+(\kappa^- | \nu^-)$, for any $\kappa\in\GT_M$, $\nu\in\GT_{M+1}$.
In other words, if $\kappa$ is distributed according to $P^-(\cdot | \nu)$, then $\kappa^-$ is distributed according to $P^+( \cdot | \nu^-)$.

Next, let $\mu(N)^- := (-\mu(N)_N, \ldots, -\mu(N)_1)$, $\lambda(N+1)^- := (-\lambda(N+1)_{N+1}, \ldots, -\lambda(N+1)_1)$; if $c > 0$ is the constant for the case $\sigma_{N+1} = -1$ treated above, then
\begin{multline*}
\Prob\left( \mu(N)_{b(N) + i} = \lambda(N+1)_{b(N+1) + i}, \textrm{ for all } -k \leq i \leq k \right)\\
= \Prob\left( \mu(N)^-_{N + 1 - b(N) - i} = \lambda(N+1)^-_{N + 2 - b(N+1) - i}, \textrm{ for all } -k \leq i \leq k \right)\\
> 1 - c q^{N + 2 - b(N+1)} = 1 - (cq^2)q^{N - b(N)},
\end{multline*}
where the last equality comes from $b(N+1) = b(N)$, in the case $\sigma_{N+1} = +1$.
\end{proof}

\subsection{Law of Large Numbers and the Martin boundary}\label{sec:martin}

\begin{thm}\label{thm:LLN}
Let $\{\mu(N)\in\GT_N\}_{N \geq 1}$ be a sequence of random signatures such that each $\mu(N)$ is $M_N^{\bft}$-distributed (see Proposition $\ref{prop:construction}$ for the definition of the probability measures $M_N^{\bft}$).
Then, for each $k\in\Z$, the probability of the event
\begin{equation*}
\mu(L)_{b(L) + k} = t_{1-k}
\end{equation*}
tends to $1$, as $L$ goes to infinity.
\end{thm}
\begin{proof}
Let $\{\lambda(N) \in \GT_N\}_{N \geq 1}$ be a sequence of signatures that $\bfb$-stabilizes to $\bft$.
From Proposition $\ref{prop:construction}$, the probability measures $\Lambda^N_K(\lambda(N), \cdot)$ converge weakly to $M_K^{\bft}$, for all $K \geq 1$.
So if we let $\{\mu^{N, K}\}_{N \geq K} \subset \GT_K$ be a sequence of random signatures with $\mu^{N, K}$ being $\Lambda^N_K(\lambda(N), \cdot)$-distributed, and let $\mu^K \in \GT_K$ be random $M^{\bft}_K$-distributed, then $\mu^{N, K}$ converges weakly to $\mu^K$.

Let $\epsilon > 0$ be arbitrary.
Let $c>0$ be the constant given in Proposition $\ref{prop:schurestimate}$ for the given $k\in\N$ in the statement of the Theorem.
There exists $N_0\in\N$ large enough so that
\[
\prod_{i=N_0}^{\infty}(1 - cq^i)^2 > 1 - \epsilon.
\]
As $\lim_{N\rightarrow\infty} b(N) = \lim_{N\rightarrow\infty} (N - b(N)) = + \infty$, there exists a large enough $L\in\N$ such that
\[
\min\{b(N), N - b(N)\} \geq N_0, \textrm{ for all } N > L.
\]
Then Proposition $\ref{prop:schurestimate}$ implies
\[
\Prob \left(  \mu^{N, L}_{b(L) + k} = \lambda(N)_{b(N) + k}  \right) \geq \prod_{i=L}^{N-1} (1 - cq^{x(i)}) \  \ \ \forall N > L,
\]
where
\[
x(i) := \begin{cases}
b(i),& \textrm{ if } \sigma_{i+1} = -1;\\
i-b(i),& \textrm{ if } \sigma_{i+1} = +1.
\end{cases}
\]
By our choice of $L$, $x(i) \geq N_0$ for all $i \geq L_0$.
Also, the sequence $(x(L), x(L+1), \dots)$ does not contain the same number more than twice.
As a result,
\begin{equation}\label{eqn:ineq1}
\Prob\left( \mu^{N, L}_{b(L) + k} = \lambda(N)_{b(N) + k} \right) \geq \prod_{i=N_0}^{\infty} (1 - cq^i)^2 > 1 - \epsilon \ \ \ \forall N > L.
\end{equation}

Now since $\lambda(N)$ $\bfb$-stabilizes to $\bft$, we have $\lambda(N)_{b(N) + k} = t_{1 - k}$, for large enough $N\in\N$.
Also, as mentioned above, $\mu^{N, L}$ converges weakly to $\mu^L$, as $N$ goes to infinity.
Therefore, from $(\ref{eqn:ineq1})$ and these observations, we obtain the inequality
\begin{equation}\label{eqn:ineq2}
\Prob\left( \mu^{L}_{b(L) + k} = t_{1 - k} \right) > 1 - \epsilon.
\end{equation}
Since $\epsilon > 0$ was arbitrary, $(\ref{eqn:ineq2})$ implies the theorem.
\end{proof}

\begin{thm}\label{thm:minimalschur}
The Martin boundary of the symmetric $q$-Gelfand-Tsetlin graph associated to $\sigma$ is in bijection with the set $\X$, under the map $\bft \mapsto \{ M_N^{\bft} \}_{N \geq 1}$ of Proposition $\ref{prop:construction}$.
\end{thm}
\begin{proof}
By Proposition $\ref{prop:construction}$ and Theorem $\ref{thm:LLN}$, each coherent system $\{M_N^{\bft}\}_{N \geq 1}$ belongs to the Martin boundary, and all of them are distinct.
It remains to show that there are no other points in the Martin boundary.

Let $\{M_k\}_{k \geq 1}$ be an element of $\Omega_q^{\Martin}$; then there exists a sequence $\{\lambda(N) \in \GT_N\}_{N \geq 1}$ such that
\begin{equation}\label{eqn:schurdfnMartin}
M_k(\mu) = \lim_{N\rightarrow\infty} \Lambda^N_k(\lambda(N), \mu), \textrm{ for all }\mu\in\GT_k, k \geq 1.
\end{equation}
Let $m\in\Z$ be fixed, but arbitrary, and let $c > 0$ be the constant in Proposition $\ref{prop:schurestimate}$ for $k = |m|+1$.
Let $N_0\in\N$ be such that
\[
\prod_{i=N_0}^{\infty} (1 - cq^i)^2 > 2/3.
\]

Since $\lim_{N\rightarrow\infty} b(N) = \lim_{N\rightarrow\infty} (N - b(N)) = +\infty$, there exists $L\in\N$ such that $\min\{b(N), N - b(N)\} \geq N_0$, for all $N \geq L$.
Fix $L$ for the moment.
For any $N \geq L$, let $\mu^N\in\GT_L$ be a random signature which is $\Lambda^N_L(\lambda(N), \cdot)$-distributed.
Also let $\xi^N := \mu^N_{b(L)+m}\in\Z$, so it is a random integer.
From Proposition $\ref{prop:schurestimate}$,
\[
\Prob(\xi^N = \lambda(N)_{b(N)+m}) \geq
\Prob \left( \mu^N_{b(L) + i} = \lambda(N)_{b(N)+i}, \textrm{ for all } -k \leq i \leq k \right) \geq
\prod_{i=L}^{N-1} (1 - cq^{x(i)}),
\]
where $x(i) := b(i)$, if  $\sigma_{i+1} = -1$ and $x(i) := i - b(i)$, if $\sigma_{i+1} = +1$.
By our choice of $L$, $x(i) \geq N_0$ for $i \geq L$.
Moreover, the sequence $(x(L), x(L+1), \ldots)$ does not contain the same number more than twice. As a result,
\begin{equation}\label{eqn:boundxi}
\Prob(\xi^N = \lambda(N)_{b(N)+m}) \geq \prod_{i=N_0}^{\infty} (1 - cq^i)^2 > 2/3.
\end{equation}
By assumption, $M_L$ is the weak limit of $\Lambda^N_L(\lambda(N), \cdot)$, as $N$ tends to infinity.
Therefore, if we let $M$ be the pushforward of $M_L$ from $\GT_L$ to the coordinate $b(N)+m$, then $M$ is a probability measure and the weak limit of the laws of the random variables $\xi^N$.
From $(\ref{eqn:boundxi})$, we deduce that $M(\{p\}) > 2/3$, for any $p\in\Z$ which is a subsequential limit of the sequence $\{\lambda(N)_{b(N)+m}\}_{N \geq 1}$.
This cannot occur for more than one $p\in\Z$, or the total measure of $M$ would be at least than $4/3$.
On the other hand, there has to be at least one subsequential limit of $\{\lambda(N)_{b(N)+m}\}_{N \geq 1}$, since otherwise the total measure of $M$ would be at most $1/3$.
Therefore the sequence $\{\lambda(N)_{b(N)+m}\}_{N \geq 1}$ converges as $N$ goes to infinity.
Since $m\in\Z$ was arbitrary, the limit
\[
\lim_{N\rightarrow\infty} \lambda(N)_{b(N)+m}
\]
exists for any $m\in\Z$.
This implies that $\{\lambda(N)\}_{N \geq 1}$ must $\bfb$-stabilize to some $\bft\in\X$.
From $(\ref{eqn:schurdfnMartin})$ and Proposition $\ref{prop:construction}$, it follows that $M_k = M_k^{\bft}$ for any $k\in\N$, i.e., $\{M_k\}_k$ coincides with the coherent system $\{ M_N^{\bft} \}_{N \geq 1}$, concluding the proof.
\end{proof}

\subsection{Characterization of the minimal boundary}\label{sec:minimal}

\begin{thm}\label{thm:symmetricschur}
The minimal boundary $\Omega_q$ of the symmetric $q$-Gelfand-Tsetlin graph associated to $\sigma$ is equal to the Martin boundary $\Omega_q^{\Martin}$.
\end{thm}
\begin{proof}
This is a Corollary of Theorems $\ref{thm:LLN}$, $\ref{thm:minimalschur}$ and $\ref{thm:minimalmartin}$.
Indeed, by the latter one, the minimal boundary is contained in the Martin boundary and it remains to show that the coherent systems $\{ M_N^{\bft} \}_{N \geq 1}$ are extreme points of the convex space $\lim_{\leftarrow}{\mathcal{M}(\GT_N)}$ of coherent systems on the symmetric $q$-Gelfand-Tsetlin graph.
By Theorem $\ref{thm:LLN}$, all the measures $\{M_N^{\bft}\}_{N \geq 1}$ have pairwise disjoint supports, hence one cannot be a convex combination of the others, which finishes the proof.
See also \cite[proof of Thm. 3.12]{GO} and \cite[Step 3 in proof of Thm. 6.2]{Ol2} for more detailed expositions of similar proofs.
\end{proof}

\section{Boundary of the BC type $q$-Gelfand-Tsetlin graph}\label{sec:BCboundary}

In this section, we characterize the Martin and minimal boundaries of the BC type $q$-Gelfand-Tsetlin graph.
We suppress the type $G\in\{B, C, D\}$ from most notations, for simplicity.
The parameter $\epsilon$ is $\half, 1, 0$, for $G = B, C, D$, respectively.

The reader should compare Section $\ref{boundary:qschur}$ with this one.
Both follow the same approach and yield similar results.

\subsection{A family of coherent probability measures}

In Proposition $\ref{prop:sympconstruction}$, we define a map from $\Y$ into construct the measures of the Martin boundary of the BC type $q$-Gelfand-Tsetlin graph.

Let $k\in\N$ be arbitrary.
Recall the linear space $C_k$, defined at the beginning of Section $\ref{sec:inject}$, of analytic functions $f(z_1, \ldots, z_k)$ on $(\C^*)^k$ for which an absolutely and uniformly convergent expansion
\begin{equation}\label{eqn:fourierbc}
f(z_1, \ldots, z_k) = \sum_{m_1, \dots, m_k\in\Z}{a_{m_1, \ldots, m_k}z_1^{m_1}\cdots z_k^{m_k}}\ \ \forall (z_1, \ldots, z_k)\in (\C^*)^k,
\end{equation}
exists.
Let us call $C_k^{W_k}$ the linear subspace of $C_k$, consisting of functions $f(z_1, \ldots, z_k)$ for which the coefficients in the expansion $(\ref{eqn:fourierbc})$ satisfy
\begin{equation*}
a_{m_1, \ldots, m_k} = a_{\epsilon_1m_{\sigma(1)}, \ldots, \epsilon_km_{\sigma(k)}}, \textrm{ for all }\sigma\in S_k, \ m_1, \ldots, m_k\in\Z, \ \epsilon_1, \ldots, \epsilon_k\in\{1, -1\}.
\end{equation*}

For example, finite linear combinations of the symplectic/orthogonal polynomials $\chi^G_{\lambda}(z_1, \ldots, z_k)$, $G\in\{B, C, D\}$, $\lambda\in\GTp_k$,
belong to $C_k^{W_k}$.
For each $G\in\{B, C, D\}$ and $\mu\in\GTp_k$, define the functionals $F_{\mu}^G: C_k^{W_k} \rightarrow \C$ via
\begin{equation*}
F_{\mu}^G(f) := \frac{\widetilde{F}_{\mu}^G(f)}{\widetilde{F}_{\mu}^G(\chi^G_{\mu}(z_1, \ldots, z_k))},
\end{equation*}
where
\begin{equation*}
\widetilde{F}_{\mu}^G(f) := \oint_{|z_1| = 1} \cdots \oint_{|z_k| = 1}
f(z_1, \ldots, z_k) \overline{\chi_{\mu}^G(z_1, \ldots, z_k)} \mathfrak{m}^G(z_1, \ldots, z_k) \prod_{i=1}^k{\frac{dz_i}{2\pi\ii}};
\end{equation*}
\begin{equation}\label{eqn:weightm}
\mathfrak{m}^G(z_1, \ldots, z_k) := \prod_{1 \leq i < j \leq N} |z_i - z_j|^{2} |1 - z_iz_j|^{2}
\times  \begin{cases}
	\prod_{i=1}^N{|1-z_i|^2}, & \text{if}\ G = B; \\
	\prod_{i=1}^N{|1-z_i|^2|1+z_i|^2}, & \text{if}\ G = C; \\
	1, & \textrm{if}\ G = D.
    \end{cases}
\end{equation}

There is an explicit formula for $\widetilde{F}_{\mu}^G(\chi_{\mu}^G(z_1, \ldots, z_k))$, see \cite[Lem. 3.5]{OlOs}, which shows that it is nonzero and therefore the definition of $F_{\mu}^G$ is well-posed.

\begin{lem}\label{technicalsymplectic}
The linear functionals $\{F_{\mu}^G : C_k^{W_k} \rightarrow \C\}_{\mu\in\GT_k}$ satisfy
\begin{enumerate}
	\item $F^G_{\mu}(\chi^G_{\nu}(z_1, \ldots, z_k)) = \delta_{\mu, \nu}$, for all $\mu, \nu\in\GTp_k$.
	\item If $\{f_N\}_{N \geq 1} \subset C_k^{W_k}$ is such that $f_N \rightarrow f$ uniformly on compact subsets of $(\C^*)^k$, then $f\in C_k^{W_k}$ and $F_{\mu}^G(f_N) \rightarrow F_{\mu}^G(f)$ for all $\mu\in\GTp_k$.
	\item If $h\in C_k^{W_k}$ and $F_{\mu}^G(h) = 0$, for all $\mu\in\GTp_k$, then $h$ is identically zero.
\end{enumerate}
\end{lem}

\begin{proof}
The first item of the lemma is well known, see for instance \cite{OlOs}.
The second item is obvious.
As for the third item, its proof is similar to the proof of Lemma \ref{lem:schurfunctionals} (3).
The only difference is that now we need the fact that the linear space of $W_k$--symmetric Laurent polynomials on $k$ variables is dense in the space of continuous, $W_k$--symmetric functions on the $k$-dimensional torus $\mathbb{T}^k$.
Indeed note that, under the change of variables $z \in \mathbb{T} \rightarrow [-1, 1] \ni x := (z+z^{-1})/2$, the $W_k$--symmetric Laurent polynomials become usual symmetric polynomials.
Our statement is then equivalent to the fact that symmetric polynomials are dense on the space of continuous, symmetric functions on $[-1, 1]^k$, and this follows from Weierstrass approximation theorem.
\end{proof}

\begin{prop}\label{prop:sympconstruction}
Take any $\bfy\in\Y$. Let $\{ \lambda(N) \in \GTp_N \}_{N \geq 1}$ be a sequence that stabilizes to $\bfy$.
Then for each $k \geq 1$, $\mu\in\GTp_k$, the limit
\begin{equation*}
M^{\bfy}_k(\mu) = \lim_{N\rightarrow\infty} \Lambda^N_k(\lambda(N), \mu)
\end{equation*}
exists and does not depend on the choice of $\{ \lambda(N) \}_{N \geq 1}$.
For each $k \geq 1$, $M_k^{\bfy}$ is a probability measure on $\GTp_k$ and, moreover, $\{M^{\bfy}_k \in \M(\GTp_k)\}_{k \geq 1}$ is a coherent system.
\end{prop}

\begin{proof}
Let $\{\lambda(N)\in\GTp_N\}_{N \geq 1}$ be any sequence of nonnegative signatures that stabilizes to $\bfy\in\Y$; also let $k\in\N$ be arbitrary.
From the definition $(\ref{df:symplecticlinks})$ of the Markov kernels $\Lambda^{N+1}_N$, we have
\begin{equation}\label{eqn:prelimitdefining}
\frac{\chi^G_{\lambda}(x_1, \ldots, x_k, q^{k+\epsilon}, \ldots, q^{N-1+\epsilon})}{\chi_{\lambda}^G(q^{\epsilon}, q^{1+\epsilon}, \ldots, q^{N-1+\epsilon})} =
\sum_{\mu \in \GTp_k} \Lambda^N_k(\lambda(N), \mu) \frac{\chi^G_{\mu}(x_1, \ldots, x_k)}{\chi_{\mu}^G(q^{\epsilon}, \ldots, q^{k-1+\epsilon})}.
\end{equation}
By virtue of Theorem $\ref{thm:multivariablesymplectic}$, the left side of $(\ref{eqn:prelimitdefining})$ converges to the analytic function $\Phi^{\bfy, G}(x_1, \ldots, x_k; q)$, on compact subsets of $(\C^*)^k$.
The rest of the proof is similar to that of Proposition $\ref{prop:construction}$.
One shows that $\{\Lambda^N_k(\lambda(N), \mu) \}_{N \geq k}$ has a limit, for any $\mu\in\GTp_k$, and moreover each limit
\begin{equation}\label{symplecticamu}
\lim_{N \rightarrow \infty} \Lambda^N_k(\lambda(N), \mu) =
\chi_{\mu}^G(q^{\epsilon}, q^{1+\epsilon}, \ldots, q^{k-1+\epsilon}) F^G_{\mu}(\Phi^{\bfy, G}) =: a_{\mu}
\end{equation}
does not depend on the stabilizing sequence of nonnegative signatures $\{ \lambda(N) \}_{N \geq 1}$.
One then shows $a_{\mu} \geq 0$ and $\sum_{\mu\in\GTp_k}{a_{\mu}} = 1$, and therefore the measure $M^{\bfy}_k \in \M(\GTp_k)$, $M^{\bfy}_k(\mu) := a_{\mu}$, satisfies the desired properties.
The coherency of the sequence $\{ M^{\bfy}_k \}_{k \geq 1}$ is evident.
In working out these outlined steps, one should follow the proof of Proposition $\ref{prop:construction}$ with the only difference that Lemma $\ref{technicalsymplectic}$ is used instead of Lemma $\ref{lem:schurfunctionals}$.
\end{proof}

\subsection{Concentration bound}

\begin{prop}\label{prop:sympestimate}
Let $k\in\N$ and let $\{\lambda(N) \in \GTp_N \}_{N \geq 1}$ be an arbitrary sequence.
Let $\mu(N)\in\GTp_N$ be random variables such that each $\mu(N)$ is distributed according to the probability distribution $\Lambda^{N+1}_{N}(\lambda(N+1), \cdot)$.
Then there exists a constant $c > 0$, independent of $N$, such that
\begin{equation*}
\Prob\left( \mu(N)_{N + 1 -  i} = \lambda(N+1)_{N + 2 - i} \textrm{ for all } i = 1, 2, \ldots, k \right) > 1 - cq^N.
\end{equation*}
\end{prop}
\begin{proof}
By Boole's inequality, it suffices to show that for a fixed integer $m \geq 1$, there exists a constant $c > 0$, independent of $N$, such that
\begin{equation}\label{generalreduction}
\Prob\left( \mu(N)_{N + 1 -  m} = \lambda(N+1)_{N + 2 - m} \right) > 1 - cq^N.
\end{equation}

Let us do the case $m = 1$ and then remark on the small differences needed for the argument in the general case $m \geq 2$.
For $m = 1$, it will suffice to show the existence of constants $c_1, c_2 > 0$ such that
\begin{align}
\Prob( \mu(N)_{N} \leq \lambda(N+1)_{N+1} - 1 ) < c_1 q^N;\label{firstline}\\
\Prob( \mu(N)_{N} \geq \lambda(N+1)_{N+1} + 1 ) < c_2 q^N.\label{secondline}
\end{align}

Denote $\lambda(N+1)$ by $\lambda$, and $\mu(N)$ by $\mu$.

\smallskip

We begin with $(\ref{firstline})$.
If $\lambda_{N+1} = 0$, then $\Prob( \mu_{N} \leq \lambda_{N+1} - 1 ) = 0$ because $\mu_N \geq 0$ almost surely, so assume $\lambda_{N+1} \geq 1$.
For each $0 \leq b < \lambda_{N+1}$, we want to bound the ratio
\begin{equation}\label{ratiosmall}
\frac{\Prob(\mu_N = b)}{\Prob(\mu_N = b+1)}
\end{equation}
and show it is of order $O(q^N)$.
It is convenient to define a measure $M_N(\cdot | \lambda)$ on $\GTp_{N+1} \times \GTp_N$ via
\[
M_N(\nu, \mu | \lambda) := \mathbf{1}_{\{\mu \prec \nu \prec \lambda\}}\frac{\chi_{\mu}^G(q^{\epsilon}, q^{1 + \epsilon}, \ldots, q^{N-1+\epsilon})}{\chi^G_{\lambda}(q^{\epsilon}, q^{1+\epsilon}, \ldots, q^{N+\epsilon})}
(q^{N+\epsilon})^{2|\nu| - |\lambda| - |\mu|}
\cdot \tau^G(q^{N+\epsilon}; \lambda, \nu, \mu).
\]
The formula for $\tau^G(q^{N+\epsilon}; \lambda, \nu, \mu)$ is in the statement of Proposition $\ref{prop:branchsymplectic}$.
Using Proposition $\ref{prop:evalsymplectic}$, and the fact that $0 < q < 1$, we deduce that $M_N(\cdot | \lambda)$ is a positive measure.
From Proposition $\ref{prop:branchsymplectic}$, we also deduce that the pushforward of $M_N(\cdot | \lambda)$, under the projection $\pi : \GTp_{N+1} \times \GTp_N \rightarrow \GTp_N$, is $\Lambda^{N+1}_N(\lambda, \cdot)$.
In particular, $M_N(\cdot | \lambda)$ is a probability measure.
So instead of estimating the ratio $(\ref{ratiosmall})$ with respect to the probability measure $\Lambda^{N+1}_N(\lambda, \cdot)$, we estimate it with respect to the probability measure $M_N(\cdot | \lambda)$ and show that it is of order $O(q^N)$.

We can do a further simplification. Instead of considering $M_N(\cdot | \lambda)$, let us consider arbitrary nonnegative integers $\nu_1, \ldots, \nu_N$ and $\mu_1, \ldots, \mu_{N-1}$, such that $\lambda_i \geq \nu_i \geq \lambda_{i+1}$, $\nu_j \geq \mu_j \geq \nu_{j+1}$, for relevant $i, j$, and let us estimate the ratio $(\ref{ratiosmall})$ with respect to $M_N(\cdot | \lambda)$ conditioned on $\nu_1, \ldots, \nu_N$, $\mu_1, \ldots, \mu_{N-1}$.
For any $0 \leq e \leq d \leq \lambda_{N+1}$ (and with respect to this probability measure), we get
\begin{gather*}
\Prob(\mu_N = d) = \Prob(\mu_N = d,\ \nu_{N+1} = d) + \ldots + \Prob(\mu_N = d, \ \nu_{N+1} = 0),\\
\Prob(\mu_N = d,\ \nu_{N+1} = e) \propto q^{(N+\epsilon)(2e - d)} \tau^G(q^{N+\epsilon}; \lambda, \nu, \mu)  \chi_{\mu}^G(q^{\epsilon}, q^{1+\epsilon}, \ldots, q^{N-1+\epsilon});
\end{gather*}
the proportionality constant in the last equation does not depend on $d, e$.

If $G = B$, $\tau^B(q^{N+\epsilon}; \lambda, \nu, \mu) = 1 + q^{-N-\epsilon} = 1 + q^{-N-\half}$, unless $\nu_{N+1} = 0$ in which case $\tau^B(q^{N+\epsilon}; \lambda, \nu, \mu) = 1$.
Moreover,
\begin{gather*}
\chi_{\mu}^B(q^{\half}, q^{\frac{3}{2}}, \ldots, q^{N-1+\half}) \propto q^{-\frac{d}{2}} (1 - q^{d + \half}) \prod_{i=1}^{N-1} ( 1 - q^{\mu_i + N - i - d})(1 - q^{\mu_i + N - i + d + 1}).
\end{gather*}
As a result, we obtain
\begin{gather*}
\Prob(\mu_N = d,\ \nu_{N+1} = e) \propto
q^{(N+\half)(2e - d) - \frac{d}{2}} (1 - q^{d + \half}) (1 + q^{-N - \half} - \mathbf{1}_{\{ e = 0 \}}q^{-N - \half})\\
\times\prod_{i=1}^{N-1} ( 1 - q^{\mu_i + N - i - d})(1 - q^{\mu_i + N - i + d + 1}),
\end{gather*}
and therefore
\begin{gather*}
\Prob(\mu_N = d) \propto
q^{-(N+\half)d - \frac{d}{2}} (1 - q^{d + \half})\prod_{i=1}^{N-1} ( 1 - q^{\mu_i + N - i - d})(1 - q^{\mu_i + N - i + d + 1})\\
\times\left( (1 + q^{-N-\half})\frac{1 - q^{(2N+1)(d+1)}}{1 - q^{2N+1}} - q^{-N-\half} \right), \textrm{ for }G = B.
\end{gather*}

If $G = C$, $\tau^C = 1$ always, and
\begin{equation*}
\chi_{\mu}^C(q, q^2, \ldots, q^N) \propto q^{-d}(1 - q^{2(d+1)})
\prod_{i=1}^{N-1} (1 - q^{\mu_i + N - i - d})(1 - q^{\mu_i + N - i + d + 2}).
\end{equation*}
As a result, we have
\begin{gather*}
\Prob(\mu_N = d) \propto q^{-(N+1)d - d} (1 - q^{2(d+1)})
\prod_{i=1}^{N-1} (1 - q^{\mu_i + N - i - d})(1 - q^{\mu_i + N - i + d + 2})\\
\times \left( \frac{1 - q^{(2N+2)(d+1)}}{1 - q^{2N+2}} \right), \textrm{ for } G = C.
\end{gather*}

If $G = D$, $\tau^D = 0$ unless $\nu_{N+1} \in \{0, \mu_N\}$ (at least in this case that we assume $\mu_N \leq \lambda_{N+1}$) and $\tau^D = 1$ in those two cases.
Moreover,
\begin{equation*}
\chi_{\mu}^D(1, q, \ldots, q^{N-1}) \propto
\prod_{i=1}^{N-1} (1 - q^{\mu_i + N - i - d})(1 - q^{\mu_i + N - i + d}).
\end{equation*}
As a result, we have
\begin{gather*}
\Prob(\mu_N = d) \propto q^{-Nd}(1 + q^{2Nd})
\prod_{i=1}^{N-1} (1 - q^{\mu_i + N - i - d})(1 - q^{\mu_i + N - i + d}), \textrm{ for } G = D.
\end{gather*}

Then, for any $0 \leq b < \lambda_{N+1}$, the ratio $\Prob(\mu_N = b)/\Prob(\mu_N = b+1)$ equals
\begin{align*}
q^{N+1} \frac{1 - q^{b + \half}}{1 - q^{b + \frac{3}{2}}} \prod_{i=1}^{N-1} \frac{ (1 - q^{\mu_i + N - i - b})(1 - q^{\mu_i + N - i + b + 1}) }{ (1 - q^{\mu_i + N - i - b - 1})(1 - q^{\mu_i + N - i + b + 2}) }\times \frac{1 - q^{(2N+1)(b+1)}}{1 - q^{(2N+1)(b+2)}},& \textrm{ for }G = B;\\
q^{N+2} \frac{1 - q^{2(b+1)}}{1 - q^{2(b+2)}}
\prod_{i=1}^{N-1} \frac{ (1 - q^{\mu_i + N - i - b})(1 - q^{\mu_i + N - i + b + 2}) }{ (1 - q^{\mu_i + N - i - b - 1})(1 - q^{\mu_i + N - i + b + 3}) } \times \frac{1 - q^{(2N+2)(b+1)}}{1 - q^{(2N+2)(b+2)}},& \textrm{ for } G = C;\\
q^N \frac{1 + q^{2Nb}}{1 + q^{2N(b+1)}} \prod_{i=1}^{N-1} \frac{ (1 - q^{\mu_i + N - i - b})(1 - q^{\mu_i + N - i + b}) }{ (1 - q^{\mu_i + N - i - b - 1})(1 - q^{\mu_i + N - i + b + 1}) },& \textrm{ for }G = D.
\end{align*}

By following the same analysis as in the proof of Lemma $\ref{lem:est1}$, we can upper bound all three expressions above by (the exponentially small) $2q^N/(q; q)_{\infty}^2$.
Let us give, as an example, the details for the type $G = B$.
Use the bounds $0 < 1 - q^{b+\half} < 1 - q^{b + \frac{3}{2}}$, $1 - q^{(2N+1)(b+1)} < 1 - q^{(2N+1)(b+2)}$ and $(1 - q^{\mu_i + N - i - b})(1 - q^{\mu_i + N - i + b + 1}) < 1$ for any $1 \leq i \leq N-1$, to obtain
\begin{equation}\label{schur1stbound}
\frac{\Prob(\mu_N = b)}{\Prob(\mu_N = b+1)} \leq q^{N+1} \frac{1}{\prod_{i=1}^{N-1} (1 - q^{\mu_i + N - i - b - 1})(1 - q^{\mu_i + N - i + b + 2})}.
\end{equation}
Since $\mu_1 \geq \mu_2 \geq \dots \geq \mu_{N-1}$, then
\begin{gather*}
\mu_1 + N - b - 2 >  \dots  >  \mu_i + N - i - b - 1 > \dots > \mu_{N-1} - b \geq \mu_{N - 1} - (\lambda_{N+1} - 1) \geq 1;\\
\mu_1 + N + b + 1 > \dots > \mu_i + N - i + b + 2 > \dots > \mu_{N-1} + b + 3 \geq 3.
\end{gather*}
So $(\ref{schur1stbound})$ implies
\begin{equation*}
\frac{\Prob(\mu_N = b)}{\Prob(\mu_N = b+1)} \leq
q^{N+1} \frac{1}{(q; q)_{\infty}(q^3; q)_{\infty}} \leq
\frac{q^N}{(q; q)_{\infty}^2},
\textrm{ for type }G=B.
\end{equation*}
Similar bounds can be achieved for types $C, D$.
Recall that the probabilities we are dealing with are $M_N(\cdot | \lambda)$ after conditioning over values for $\nu_1, \ldots, \nu_N, \mu_1, \ldots, \mu_{N-1}$.
But the bound achieved did not take these values into consideration, therefore $\Prob(\mu_N = b)/\Prob(\mu_N = b+1) < q^N/(q; q)^2_{\infty}$, if $\mu_N$ is $M_N(\cdot | \lambda)$-distributed.
The bound just shown holds for any $0 \leq b \leq \lambda_{N+1}$.
This implies the existence of a desired constant $c_1 > 0$ for $(\ref{firstline})$.

\smallskip

Next, let us prove $(\ref{secondline})$; the proof is similar to that of $(\ref{firstline})$.
We can assume $\lambda_N \geq \lambda_{N+1} + 1$, since otherwise $\Prob(\mu_N = \lambda_{N+1}+1) = 0$. For each $\lambda_N - 1 \geq b \geq \lambda_{N+1}$, we want to show the ratio
\begin{equation}\label{symplecticratio2}
\frac{\Prob\left( \mu_N = b+1 \right)}{\Prob\left( \mu_N = b \right)}
\end{equation}
is of order $O(q^N)$.
Define the measure $M_N'(\cdot | \lambda)$ on $\GTp_{N+1} \times \GTp_N$ via
\[
M_N'(\nu, \mu | \lambda) := \mathbf{1}_{\{\mu \prec \nu \prec \lambda\}}\frac{\chi_{\mu}^G(q^{\epsilon}, q^{1 + \epsilon}, \ldots, q^{N-1+\epsilon})}{\chi^G_{\lambda}(q^{\epsilon}, q^{1+\epsilon}, \ldots, q^{N+\epsilon})}
(q^{N+\epsilon})^{|\lambda| + |\mu| - 2|\nu|}
\cdot \tau^G(q^{-(N+\epsilon)}; \lambda, \nu, \mu).
\]
As before, Proposition $\ref{prop:evalsymplectic}$ and Proposition $\ref{prop:branchsymplectic}$, see also Remark $\ref{symmetry:inverse}$, show that $M_N'(\cdot | \lambda)$ is a probability measure whose pushforward under the projection $\GTp_{N+1} \times \GTp_N \rightarrow \GTp_N$ is $\Lambda^{N+1}_N(\lambda, \cdot)$.
As we did before, instead of estimating the ratio $(\ref{symplecticratio2})$ with respect to the probability measure $\Lambda^{N+1}_N(\lambda, \cdot)$, we estimate it when (in both the numerator and denominator) $\mu$ is distributed according to $M_N'(\cdot | \lambda)$ after conditioning over fixed values for $\nu_1, \nu_2, \ldots, \nu_{N-1}, \nu_{N+1}$ and $\mu_1, \ldots, \mu_{N-1}$.
For any $\lambda_{N+1} \leq d\leq e \leq \lambda_N$, with respect to this probability measure, we get
\begin{equation}\label{eqn:prob}
\begin{gathered}
\Prob(\mu_N = d) = \Prob(\mu_N = d,\ \nu_N = d) + \ldots + \Prob(\mu_N = d, \ \nu_N = \lambda_N),\\
\Prob(\mu_N = d,\ \nu_N = e) \propto q^{(N+\epsilon)(d - 2e)} \tau^G(q^{-(N+\epsilon)}; \lambda, \nu, \mu)  \chi_{\mu}^G(q^{\epsilon}, q^{1+\epsilon}, \ldots, q^{N-1+\epsilon}).
\end{gathered}
\end{equation}
We know that $\tau^G(q^{-(N+\epsilon)}; \lambda, \nu, \mu)$ depends only on $\lambda_{N+1}, \nu_{N+1}$ and $\mu_N$, so let us denote it $\tau^G_N(\lambda_{N+1}, \nu_{N+1}, \mu_N)$. The factor $\tau^G_N(\lambda_{N+1}, \nu_{N+1}, d)$ appears in each term of the first line of $(\ref{eqn:prob})$, so it can be factored out.
On the other hand, from the formulas in Proposition $\ref{prop:evalsymplectic}$, we obtain that $\Prob(\mu_N = d)$ is proportional (up to a nonzero factor that does not depend on $d$) to:
\begin{align*}
\tau^B_N(\lambda_{N+1}, \nu_{N+1}, d) \cdot q^{(N+\half)(d - 2\lambda_N) - \frac{d}{2}} (1 - q^{d + \half})&\prod_{i=1}^{N-1} ( 1 - q^{\mu_i + N - i - d})(1 - q^{\mu_i + N - i + d + 1})\\
\times\left( \frac{1 - q^{(2N+1)(\lambda_N - d + 1)}}{1 - q^{2N+1}} \right),& \textrm{ for }G = B;\\
\tau^C_N(\lambda_{N+1}, \nu_{N+1}, d) \cdot q^{(N+1)(d - 2\lambda_N) - d}(1 - q^{2(d+1)})
&\prod_{i=1}^{N-1} (1 - q^{\mu_i + N - i - d})(1 - q^{\mu_i + N - i + d + 2}) \\
\times \left( \frac{1 - q^{(2N+2)(\lambda_N - d + 1)}}{1 - q^{2N+2}} \right),& \textrm{ for }G = C;\\
\tau^D_N(\lambda_{N+1}, \nu_{N+1}, d)  \cdot q^{N(d - 2\lambda_N)} \prod_{i=1}^{N-1} (1 - q^{\mu_i + N - i - d})&(1 - q^{\mu_i + N - i + d})\\
\times \left( \frac{1 - q^{2N(\lambda_N - d +1)}}{1 - q^{2N}} \right),& \textrm{ for }G = D.
\end{align*}
Next we note that for any $\lambda_N - 1 \geq b \geq \lambda_{N+1}$, we have $\tau^G_N(\lambda_{N+1}, \nu_{N+1}, b) = \tau^G_N(\lambda_{N+1}, \nu_{N+1}, b+1)$, except in the case that $G = D$, $b = 0$, in which case $\tau^D_N(\lambda_{N+1}, \nu_{N+1}, 0) = 1$, $\tau^D_N(\lambda_{N+1}, \nu_{N+1}, 1) = 2$.
Therefore the ratio $\Prob(\mu_N = b+1)/\Prob(\mu_N = b)$ equals
\begin{align*}
q^N \frac{(1 - q^{b+\frac{3}{2}})(1 - q^{(2N+1)(\lambda_N - b)})}{(1 - q^{b+\half})(1 - q^{(2N+1)(\lambda_N - b + 1)})}
\prod_{i=1}^{N-1} \frac{(1 - q^{\mu_i + N - i - b - 1})(1 - q^{\mu_i + N - i + b + 2})}{(1 - q^{\mu_i + N - i - b})(1 - q^{\mu_i + N - i + b + 1})},& \textrm{ for } G = B;\\
q^N \frac{(1 - q^{2(b+2)})(1 - q^{2N(\lambda_N - b)})}{(1 - q^{2(b+1)})(1 - q^{2N(\lambda_N - b +1)})}
\prod_{i=1}^{N-1} \frac{(1 - q^{\mu_i + N - i - b - 1})(1 - q^{\mu_i + N - i + b + 3})}{(1 - q^{\mu_i + N - i - b})(1 - q^{\mu_i + N - i + b + 2})},& \textrm{ for }G = C;\\
q^N (1 + \mathbf{1}_{\{ b = 0 \}}) \frac{1 - q^{2N(\lambda_N - b)}}{1 - q^{2N(\lambda_N - b + 1)}}
\prod_{i = 1}^{N-1} \frac{(1 - q^{\mu_i + N - i - b - 1})(1 - q^{\mu_i + N - i + b + 1})}{(1 - q^{\mu_i + N - i - b})(1 - q^{\mu_i + N - i + b})},& \textrm{ for }G = D.
\end{align*}
Just as we have done already a few times, we can upper bound all these three expressions by $cq^N/(q; q)_{\infty}^2$, for some constant $c > 0$; the bound is uniform on the values $\nu_1, \ldots, \nu_{N-1}, \nu_{N+1}$, $\mu_1, \ldots, \mu_{N-1}$ that we conditioned over. Therefore, if $\mu_N$ is $M_N'(\cdot | \lambda)$-distributed, we also have that $\Prob\left( \mu_N = b+1 \right)/\Prob\left( \mu_N = b \right)$ is upper bounded by $cq^N/(q; q)^2_{\infty}$,
for any $\lambda_N - 1 \geq b \geq \lambda_{N+1}$.
The bound $(\ref{secondline})$ follows.

\smallskip

The general case $m \geq 2$ of $(\ref{generalreduction})$ is very similar.
It suffices to show the existence of constants $c_1, c_2 > 0$ such that
\begin{align}
\Prob( \mu(N)_{N+1-m} \leq \lambda(N+1)_{N+2-m} - 1 ) < c_1 q^N;\label{generalfirstline}\\
\Prob( \mu(N)_{N+1-m} \geq \lambda(N+1)_{N+2-m} + 1 ) < c_2 q^N.\label{generalsecondline}
\end{align}
Denote $\lambda(N+1)$ by $\lambda$, and $\mu(N)$ by $\mu$.
For $(\ref{generalfirstline})$, we can assume $\lambda_{N+2-m} > \lambda_{N+3-m}$, since otherwise $\Prob( \mu_{N+1-m} \leq \lambda_{N+2-m} - 1 ) = 0$.
Then the inequality can be deduced if we show that, for each $\lambda_{N+3-m} \leq b < \lambda_{N+2-m}$, the ratio
\[
\frac{\Prob(\mu_{N+1-i} = b)}{\Prob(\mu_{N+1-i} = b+1)}
\]
is of order $O(q^N)$.
As before, it suffices to consider $\mu_N$ with law $M_N(\cdot | \lambda)$ and to show the last statement with respect to this probability measure.
In fact, we can even condition over any values of $\nu_1, \ldots, \nu_{N+1-m}, \nu_{N+3-m}, \ldots, \nu_{N+1}$ and $\mu_1, \ldots, \mu_{N-m}, \mu_{N+2-m}, \ldots, \mu_N$ such that $\lambda_i \geq \nu_i \geq \lambda_{i+1}$, $\nu_j \geq \mu_j \geq \nu_{j+1}$, for relevant $i, j$.
For any $\lambda_{N + 3 - m} \leq e \leq d \leq \lambda_{N + 2 - m}$, we use
\begin{gather*}
\Prob(\mu_{N+1-m} = d) = \Prob(\mu_{N+1-m} = d,\ \nu_{N+2-m} = d) + \ldots + \Prob(\mu_{N+1-m} = d, \ \nu_{N+2-m} = 0),\\
\Prob(\mu_{N+1-m} = d,\ \nu_{N+2-m} = e) \propto q^{(N+\epsilon)(2e - d)} \tau^G(q^{N+\epsilon}; \lambda, \nu, \mu) \chi_{\mu}^G(q^{\epsilon}, q^{1+\epsilon}, \ldots, q^{N-1+\epsilon}).
\end{gather*}
Then we can repeat the analysis above to prove the desired bound.
The only difference is that now $\tau^G(q^{N+\epsilon}; \lambda, \nu, \mu)$ does not depend on $\mu_{N+1-m}$ or on $\nu_{N+2-m}$ (it only depends on $\lambda_{N+1}, \nu_{N+1}, \mu_N$), so we could even use instead
\[
\Prob(\mu_{N+1-m} = d,\ \nu_{N+2-m} = e) \propto q^{(N+\epsilon)(2e - d)} \chi_{\mu}^G(q^{\epsilon}, q^{1+\epsilon}, \ldots, q^{N-1+\epsilon}),
\]
which makes the remaining steps of the proof of $(\ref{generalfirstline})$ even simpler.
For the proof of $(\ref{generalsecondline})$, the same outline works, in particular, the factor $\tau^G(q^{N+\epsilon}; \lambda, \nu, \mu)$ does not play a role in the proof.
\end{proof}

\subsection{Law of Large Numbers and the Martin boundary}

\label{sec:martin_BC}

\begin{thm}[Law of Large Numbers]\label{thm:LLNsymplectic}
Let $k\in\N$ be arbitrary, let $\bfy \in \Y$, and $\{M_N^{\bfy}\}_{N \geq 1}$ be the coherent system of Proposition $\ref{prop:sympconstruction}$.
Let $\{\mu(N)\in\GTp_N\}_{N \geq 1}$ be a sequence of random nonnegative signatures such that each $\mu(N)$ is $M_N^{\bfy}$-distributed.
Then, for any $k\in\N$, the probability of the event
\begin{equation*}
\mu(L)_{L + 1 - k} = y_k
\end{equation*}
tends to $1$, as $L$ goes to infinity.
\end{thm}
\begin{proof}
The theorem is a consequence of Proposition $\ref{prop:sympestimate}$; the proof is very similar to that of Theorem $\ref{thm:LLN}$, and therefore we leave the details to the reader.
\end{proof}

The following theorem is the analogue of Theorem $\ref{thm:minimalschur}$; its proof is similar and we omit it.

\begin{thm}\label{minimalSymplectic}
The Martin boundary of the BC type $q$-Gelfand-Tsetlin graph is in bijection with the set $\Y$, under the map $\bfy \mapsto \{M_N^{\bfy}\}_{N \geq 1}$ of Proposition $\ref{prop:sympconstruction}$.
\end{thm}

\subsection{Characterization of the minimal boundary}

The following is the main theorem of this section; its proof is similar to that of Theorem $\ref{thm:symmetricschur}$ and we omit it.

\begin{thm}\label{thm:minimalsymplectic}
The minimal boundary $\Omega_q^G$ of the BC type $q$-Gelfand-Tsetlin graph is equal to the Martin boundary $\Omega_q^{G, \Martin}$.
\end{thm}

\end{document}